\documentclass[10pt]{articleHJ}
\usepackage[english]{babel}
\usepackage{amsmath, amsthm, amsfonts, mathrsfs, amsfonts}
\usepackage{mathtools}
\usepackage{bbm}
\usepackage{amssymb}
\usepackage{graphicx}
\usepackage[toc,page]{appendix}

\usepackage[text={6.0in,8.6in},centering,letterpaper]{geometry}
\usepackage[numbers,comma,square,sort&compress]{natbib}

\usepackage{tabularx}
\usepackage{enumitem}

\newlength{\defbaselineskip}
\newcommand{\setlinespacing}[1]%
           {\setlength{\baselineskip}{#1 \defbaselineskip}}

\newcommand{\singlespacing}{\setlength{\baselineskip}{\defbaselineskip}}

\newcommand{\M}{\ensuremath{\mathcal{M}}}
\newcommand{\OL}{\ensuremath{\overline{\mathcal{L}}}}

\newcommand{\D}{\ensuremath{\mathcal{D}}}

\newcommand{\Z}{\ensuremath{\mathbb{Z}}}

\newcommand{\R}{\ensuremath{\mathbb{R}}}
\newcommand{\Real}{\ensuremath{\mathbb{R}}}
\newcommand{\C}{\ensuremath{\mathbb{C}}}

\renewcommand{\L}{\ensuremath{\mathcal{L}}}
\renewcommand{\epsilon}{\ensuremath{\varepsilon}}

\renewcommand{\Re}{\mathop\mathrm{Re}\nolimits}
\renewcommand{\Im}{\mathop\mathrm{Im}\nolimits}
\renewcommand{\ker}{\ensuremath{\mathrm{ker}}}
\newcommand{\Range}{\ensuremath{\mathrm{Range}}}

\newcommand{\codim}{\ensuremath{\mathrm{codim}}}
\renewcommand{\span}{\ensuremath{\mathrm{span}}}
\newcommand{\norm}[1]{\left\Vert#1\right\Vert}		% operator norm
\newcommand{\sref}[1]{(\ref{#1})}
\DeclarePairedDelimiter\abs{\lvert}{\rvert}

\DeclarePairedDelimiter{\ip}\langle\rangle
\DeclarePairedDelimiter{\nrm}\lVert\rVert
\theoremstyle{plain}
\newtheorem{theorem}{Theorem}[section]
\newtheorem{proposition}[theorem]{Proposition}
\newtheorem{lemma}[theorem]{Lemma}
\newtheorem{corollary}[theorem]{Corollary}

\theoremstyle{definition}

\newtheorem*{assumption*}{\assumptionnumber}
\providecommand{\assumptionnumber}{}
\makeatletter
\newenvironment{assumption}[2]
 {%
  \renewcommand{\assumptionnumber}{Assumption ($#2$#1)}%
  \begin{assumption*}%
  \protected@edef\@currentlabel{#1}%
 }
 {%
  \end{assumption*}
 }
\makeatother
\newcommand{\asref}[2]{$#2$\ref{#1}}

\newtheorem*{assumptionrev*}{\assumptionnumber}



\numberwithin{equation}{section}

\begin{document}

\begin{frontmatter}

\title{Travelling waves for spatially discrete systems of FitzHugh-Nagumo type with periodic
coefficients}
\journal{...}
\author[LDA]{W. M. Schouten\corauthref{coraut}},
\corauth[coraut]{Corresponding author. }
\author[LDB]{H. J. Hupkes}
\address[LDA]{
  Mathematisch Instituut - Universiteit Leiden \\
  P.O. Box 9512; 2300 RA Leiden; The Netherlands \\ Email:  {\normalfont{\texttt{w.m.schouten@math.leidenuniv.nl}}}
}
\address[LDB]{
  Mathematisch Instituut - Universiteit Leiden \\
  P.O. Box 9512; 2300 RA Leiden; The Netherlands \\ Email:  {\normalfont{\texttt{hhupkes@math.leidenuniv.nl}}}
}

\date{\today}

\begin{abstract}
\singlespacing
We establish the existence and nonlinear stability of travelling wave solutions for a class of lattice differential equations (LDEs)
that includes the discrete FitzHugh-Nagumo system with alternating scale-separated diffusion coefficients. In particular,
we view such systems as singular perturbations of spatially homogeneous LDEs, for which stable travelling wave solutions are known to exist
in various settings.

The two-periodic waves considered in this paper are described by singularly perturbed multi-component
functional differential equations of mixed type (MFDEs).
In order to analyze these equations, we generalize the spectral convergence technique
that was developed by Bates, Chen and Chmaj to analyze the scalar Nagumo LDE. This allows us to
transfer several crucial Fredholm properties from the spatially homogeneous to the spatially periodic setting.
Our results hence do not require the use of comparison principles or exponential dichotomies.

\end{abstract}

%\begin{subjclass}
%\singlespacing
%34K31 \sep 37L15.
%\end{subjclass}

\end{frontmatter}

\section{Introduction}

In this paper we consider a class of lattice differential equations (LDEs)
that includes the FitzHugh-Nagumo system
\begin{equation}
\label{eq:int:2per:fhn}
\begin{array}{lcl}
\dot{u}_j&=&d_j (u_{j+1}+u_{j-1}-2u_j)+g(u_j;a_j)-w_j ,
\\[0.2cm]
\dot{w}_j&=&\rho_j[u_j-\gamma_j w_j]  ,
\end{array}
\end{equation}
with cubic nonlinearities
\begin{equation}
\begin{array}{lcl}
g(u;a) &=& u ( 1 - u) (u - a)
\end{array}
\end{equation}
and two-periodic coefficients
\begin{equation}\begin{array}{lcl}
(0, \infty) \times (0, 1) \times (0,1)\times (0, \infty) \ni (d_j, a_j, \rho_j,\gamma_j) &=& \left\{ \begin{array}{lcl}
  (\epsilon^{-2}, a_o,\rho_o ,\gamma_o) & & \hbox{for odd } j, \\[0.2cm]
  (1 , a_e, \rho_e ,\gamma_e) & & \hbox{for even }j  .
  \end{array} \right.\end{array}
\end{equation}
We assume that the diffusion coefficients are of different orders
in the sense $0 < \epsilon \ll 1$. Building on the results obtained in \cite{HJHFZHNGM, HJHSTBFHN}
for the spatially homogeneous FitzHugh-Nagumo LDE, we show that (\ref{eq:int:2per:fhn})
admits stable travelling pulse solutions with separate waveprofiles for the even and odd
lattice sites. The main ingredient in our approach is a spectral convergence argument,
which allows us to transfer Fredholm properties between linear operators
acting on different spaces.

\paragraph{Signal propagation}
The LDE (\ref{eq:int:2per:fhn}) can be interpreted as a spatially inhomogeneous
discretisation of the FitzHugh-Nagumo PDE
\begin{equation}\label{FHNPDE}
\begin{array}{lcl}
u_t&=&u_{xx}+g(u;a)-w,
\\[0.2cm]
w_t&=&\rho\big[u-\gamma w\big],
\end{array}
\end{equation}
again with $\rho > 0$ and $\gamma > 0$. This PDE was proposed in the 1960s \cite{FITZHUGH19662,FITZHUGH1966}
as a simplification of the four-component system that Hodgkin and Huxley developed
to describe the propagation of spike signals through the nerve fibres of giant squids \cite{HODHUX1952}.
Indeed, for small $\rho > 0$ (\ref{FHNPDE}) admits isolated pulse solutions of the form
\begin{equation}
\begin{array}{lcl}
(u,w)(x,t)&=&(\overline{u}_0,\overline{w}_0)(x+c_0 t),
\end{array}
\end{equation}
in which $c_0$ is the wavespeed and the wave profile $(\overline{u}_0,\overline{w}_0)$ satisfies the limits
\begin{equation}
\begin{array}{lcl}
\lim\limits_{|\xi|\rightarrow\infty}(\overline{u}_0,\overline{w}_0)(\xi)&=&0.
\end{array}
\end{equation}

Such solutions were first observed numerically by FitzHugh \cite{FITZHUGH1968},
but the rigorous analysis of these pulses turned out to be a major mathematical
challenge that is still ongoing. Many techniques have been developed to obtain the existence
and stability of such pulse solutions in various settings,
including geometric singular perturbation theory \cite{CARP1977,HAST1976,JONESKOPLAN1991,JONES1984},
Lin's method \cite{KRUSANSZM1997,CARTER2015,CARTER2016},
the variational principle \cite{chen2015traveling}
and the Maslov index \cite{cornwell2017opening,cornwell2017existence}.\\

%Let us return to the signal propagation through nerve fibres.
It turns out that electrical
signals can only reach feasible speeds when travelling through
nerve fibres that are insulated by a myeline coating.
Such coatings are known to admit regularly spaced gaps at the nodes of Ranvier \cite{RANVIER1878},
where propagating signals can be chemically reinforced.
In fact, the action potentials effectively jump from one node to the next
through a process caused saltatory conduction \cite{LILLIE1925}.
In order to include these effects, it is natural \cite{EVVPD18}
to replace (\ref{FHNPDE}) by the FitzHugh-Nagumo LDE
\begin{equation}\label{finiterangeversion}
\begin{array}{lcl}\dot{u}_j&=&\frac{1}{\epsilon^2}(u_{j+1}+u_{j-1}-2u_j)+g(u_j;a)-w_j,\\[0.2cm]
\dot{w}_j&=&\rho[u_j-\gamma w_j].
\end{array}
\end{equation}
In this equation the variable $u_j$ describes the potential at the
node $j \in \mathbb{Z}$ node, while $w_j$ describes the dynamics of the recovery variables.
We remark that this LDE arises directly from (\ref{FHNPDE})
by using the nearest-neighbour discretisation of the Laplacian on a grid
with spacing $\epsilon > 0$.\\

In \cite{HJHSTBFHN,HJHFZHNGM}, Hupkes and Sandstede studied
(\ref{finiterangeversion})
and showed that for $a$ sufficiently far from $\frac{1}{2}$
and small $\rho > 0$,
there exists a stable locally unique travelling pulse solution
\begin{equation}
\begin{array}{lcl}
(u_j,w_j)(t)&=&(\overline{u},\overline{w})(j +c t).
\end{array}
\end{equation}
%that is
%asymptotically stable with an asymptotic phase shift.
The techniques
relied on exponential dichotomies and Lin's method to
develop an infinite-dimensional analogue of the exchange lemma.
In \cite{Faye2015} the existence part of these results was generalized
to versions of (\ref{finiterangeversion}) that feature infinite-range discretisations
of the Laplacian that involve all neighbours instead of only the nearest-neighbours.
The stability results were also recently
generalized to this setting \cite{HJHFHNINFRANGE}, but only for small $\epsilon > 0$ at present.
Such systems with infinite-range interactions play an important role in
neural field models \cite{BRESS2014,BRESS2011,PINTO2001,SNEYD2005},
which aim to describe the dynamics of large networks of neurons.\\

Our motivation here for studying the 2-periodic version
(\ref{eq:int:2per:fhn}) of the FitzHugh-Nagumo LDE (\ref{finiterangeversion})
comes from  recent developments  in optical nanoscopy. Indeed,
the results in \cite{XU2013,ESTE2015,ESTE2016} clearly show that certain proteins in the
cytoskeleton of nerve fibres are organized periodically.
This periodicity turns out to be a universal feature of all nerve systems, not just those
which are insulated with a myeline coating. Since it also manifests
itself at the nodes of Ranvier, it is natural to allow the parameters
in (\ref{finiterangeversion}) to vary in a periodic fashion. The results
in this paper are a first step in this direction. The restriction
on the diffusion parameters is rather severe, but the absence
of a comparison principle forces us to take a perturbative approach.\\

\paragraph{Periodicity}
Periodic patterns are frequently encountered when studying
the behaviour of physical systems that have a discrete
underlying spatial structure.
Examples include
the presence of twinning microstructures
in shape memory alloys \cite{Bhattacharya2003microstructure}
%(caused by martensitic phase transitions).
and the formation of
domain-wall microstructures
in dielectric crystals \cite{Tagantsev2010domains}.\\

At present however, the mathematical analysis of such models has
predominantly focussed on one-component systems. For example,
the results in \cite{CHENGUOWU2008} cover
the bistable Nagumo LDE
\begin{equation}
\label{eq:int:per:nag:lde}
\begin{array}{lcl}
\dot{u}_j&=&d_j (u_{j+1}+u_{j-1}-2u_j)+g(u_j;a_j),
\end{array}
\end{equation}
with spatially periodic coefficients $(d_j, a_j) \in (0, \infty) \times (0, 1)$.
Exploiting the comparison principle, the authors were able to
establish the existence of stable travelling wave solutions.
Similar results were obtained in \cite{GUOWU2009}
for mono-stable versions of (\ref{eq:int:per:nag:lde}).\\

Let us also mention the results
in \cite{Faver2017nanopteron,Faver2018exact, Hoffman2017nanopteron},
where the authors
consider chains of alternating masses connected by identical springs
(and vice versa). The dynamical behaviour of such systems
can be modelled by LDEs
of Fermi-Pasta-Ulam type with periodic coefficients. In certain limiting
cases the authors were able to construct
so-called nanopterons, which are multi-component wave solutions
that have low-amplitude oscillations in their tails.\\

In the examples above the underlying periodicity
is built into the spatial system itself.
However, periodic patterns
also arise naturally as solutions to spatially homogeneous discrete systems.
As an example, systems of the
form (\ref{eq:int:per:nag:lde}) with
homogeneous but negative diffusion coefficients $d_j = d < 0$
have been used to
describe phase transitions
for grids of particles that have visco-elastic interactions
\cite{CAHNNOV1994,CAHNVLECK1999,VAIN2009}.
Upon introducing separate scalings for the odd and even
lattice sites,
this one-component LDE can be turned into a $2$-periodic system of the form
\begin{equation}
\begin{array}{lcl}
\dot{v}_j&=&d_e\big(w_j+w_{j-1}-2v_j\big)-f_e(v_j) ,\\[0.2cm]
\dot{w}_j&=&d_o\big(v_{j+1}+v_{j}-2w_j\big)-f_o(w_j)
\end{array}
\end{equation}
with positive coefficients $d_e > 0$ and $d_o > 0$.
Systems of this type have been analyzed in considerable detail
in \cite{BRUC2011,Vainchtein2015propagation}, where the authors
establish the co-existence of patterns that can be both monostable
and bistable in nature.\\

As a final example, let us
mention that the LDE (\ref{eq:int:per:nag:lde})
with positive spatially homogeneous diffusion coefficients
$d_j = d > 0$ can admit many periodic equilibria
\cite{VL30}. In \cite{Morelli2018} the authors construct
bichromatic travelling waves that connect spatially homogeneous
rest-states with such 2-periodic equilibria. Such waves can actually
travel in parameter regimes where the standard  monochromatic waves
that connect zero to one are trapped.
This presents a secondary mechanism by which the stable
states zero and one can spread throughout the spatial domain.\\

\paragraph{Wave equations}
Returning to the 2-periodic FitzHugh-Nagumo LDE (\ref{eq:int:2per:fhn}),
we use the travelling wave Ansatz
%\cite{Faver2017nanopteron
%Following \cite[Eq. (2.1)]{Faver2017nanopteron},\cite[Eq. (2.1)]{Faver2018exact},
\begin{equation}
\begin{array}{lcl}
(u,w)_j(t)&=&\begin{cases}(\overline{u}_o,\overline{w}_o)(j+ct),\enskip\text{when }j\text{ is odd},\\[0.2cm]
(\overline{u}_e,\overline{w}_e)(j+ct),\enskip\text{when }j\text{ is even}
\end{cases}\end{array}
\end{equation}
to arrive at the coupled system
%wave eqns ...
%Plugging this Ansatz into (\ref{COUPLEDFHN}) yields the MFDE
\begin{equation}\label{FHNMFDE}
\begin{array}{lcl}
c\overline{u}_o'(\xi)&=&\frac{1}{\epsilon^2}\big(\overline{u}_e(\xi+1)+\overline{u}_e(\xi-1)-2\overline{u}_o(\xi)\big)
  +g(\overline{u}_o(\xi);a_o)-\overline{w}_o(\xi), \\[0.2cm]
c\overline{w}_o'(\xi)&=&\rho_o[\overline{u}_o(\xi)-\gamma_o \overline{w}_o(\xi)],
 \\[0.2cm]
c\overline{u}_e'(\xi)&=&\big(\overline{u}_o(\xi+1)+\overline{u}_o(\xi-1)-2\overline{u}_e(\xi)\big)+g(\overline{u}_e(\xi);a_e)-\overline{w}_e(\xi),
  \\[0.2cm]
c\overline{w}_e'(\xi)&=&\rho_e[\overline{u}_e(\xi)-\gamma_e \overline{w}_e(\xi)].
\end{array}
\end{equation}
Multiplying the first line by $\epsilon^2$ and then taking $\epsilon \downarrow 0$,
we obtain the direct relation
\begin{equation}\begin{array}{lcl}
\overline{u}_o(\xi) &=& \frac{1}{2} \big[ \overline{u}_e(\xi + 1) + \overline{u}_e(\xi - 1) \big] ,\end{array}
\end{equation}
which can be substituted into the last two lines to yield
\begin{equation}
\label{eq:int:wv:eq:hom:fhn}
\begin{array}{lcl}
c\overline{u}_e'(\xi)&=&\frac{1}{2}\big(\overline{u}_e(\xi+2)+\overline{u}_e(\xi-2)-2\overline{u}_e(\xi)\big)+g(\overline{u}_e(\xi);a_e)-\overline{w}_e(\xi),
  \\[0.2cm]
c\overline{w}_e'(\xi)&=&\rho_e[\overline{u}_e(\xi)-\gamma_e \overline{w}_e(\xi)].
\end{array}
\end{equation}
All the odd variables have been eliminated from this last equation,
which in fact describes pulse solutions to the spatially homogeneous FitzHugh-Nagumo LDE
(\ref{finiterangeversion}). Plugging these pulses into the remaining equation
we arrive at
\begin{equation}
\begin{array}{lcl}
c\overline{w}_o'(\xi) + \rho_o \gamma_o \overline{w}_o(\xi)
&=& \frac{1}{2} \rho_o \big[ \overline{u}_e(\xi + 1) + \overline{u}_e(\xi - 1) \big].
\end{array}
\end{equation}
This can be solved to yield the remaining second component
of a singular pulse solution that we denote by
\begin{equation}\begin{array}{lcl}
\label{eq:int:def:sing:orbit}
\overline{U}_0 = \big(\overline{u}_{o;0}, \overline{w}_{o;0}, \overline{u}_{e;0}, \overline{w}_{e;0}\big).\end{array}
\end{equation}

The main task in this paper is to
construct stable travelling wave solutions
to (\ref{eq:int:2per:fhn}) by continuing this singular pulse
into the regime $0 < \epsilon \ll 1$. We use a functional analytic
approach to handle this singular perturbation,
focussing on the linear operator associated to the
linearization of (\ref{FHNMFDE}) with $\epsilon > 0$
around the singular pulse. We show that this operator
inherits several crucial Fredholm properties
that were established in \cite{HJHSTBFHN}
for the linearization of (\ref{eq:int:wv:eq:hom:fhn})
around the even pulse $\big(\overline{u}_{e;0}, \overline{w}_{e;0}\big)$.\\

Our results are not limited to the two-component system
(\ref{eq:int:2per:fhn}). Indeed,
we consider general $(n+k)$-dimensional reaction
diffusion systems with 2-periodic coefficients,
where $n\geq 1$ is the number of components with a non-zero diffusion
term and $k\geq 0$ is the number of components that do not diffuse.
We can handle both travelling fronts and travelling pulses,
but do impose conditions on the end-states that are stronger
than the usual temporal stability requirements. Indeed, at times we will
require (submatrices of) the corresponding Jacobians to be negative definite
instead of merely spectrally stable.
We emphasize that these distinctions disappear for scalar problems.
In particular, our framework also covers the Nagumo LDE (\ref{eq:int:per:nag:lde}),
but does not involve the use of a comparison principle.

\paragraph{Spectral convergence}
The main inspiration for our approach is the spectral convergence
technique that was developed in \cite{BatesInfRange}
to establish the existence of travelling wave solutions
to the homogeneous Nagumo LDE\footnote{
The power of the results in \cite{BatesInfRange} is that they
also apply to variants of (\ref{eq:int:per:nag:lde})
with infinite-range interactions. We describe their ideas here in a
finite-range setting for notational clarity.
}
(\ref{eq:int:per:nag:lde})
with diffusion coefficients $d_j = 1/\epsilon^2 \gg 1$.
The linear operator
\begin{equation}
\label{scalaroperator}
\begin{array}{lcl}
\mathcal{L}_{\epsilon}v(\xi)&=&
c_{0}v'(\xi)-\frac{1}{\epsilon^2}\Big[v(\xi+\epsilon)+v(\xi-\epsilon)-2v(\xi)\Big]
-g_u(\overline{u}_{0}(\xi);a)v(\xi),
\end{array}\end{equation}
plays a crucial role in this approach,
where the pair $(c_{0}, \overline{u}_{0})$ is
the travelling front solution of the Nagumo PDE
\begin{equation}
\label{eq:int:nagumo:pde}
\begin{array}{lcl}
u_t& =& u_{xx} + g( u;a) .\end{array}
\end{equation}
This front solutions satisfies the system
\begin{equation}
\begin{array}{lclclcl}
c_0 \overline{u}_0'(\xi) &=& \overline{u}_0''(\xi) + g(\overline{u}(\xi) ; a),
\qquad
\overline{u}_0(-\infty) &=& 0, \qquad \overline{u}_0(+\infty) &=& 1,\end{array}
\end{equation}
to which we can associate the linear operator
\begin{equation}
\begin{array}{lcl}
[\mathcal{L}_0 v](\xi) &=& c_0 v'(\xi) - v''(\xi) - g_u\big(\overline{u}(\xi) ; a \big) v(\xi),\end{array}
\end{equation}
which can be interpreted as the formal $\epsilon \downarrow 0$ limit
of (\ref{scalaroperator}).
It is well-known that $\mathcal{L}_0 + \delta: H^2 \to L^2$ is invertible for all
$\delta > 0$. By considering sequences
\begin{equation}\begin{array}{lclclcl}
w_j &=& (\mathcal{L}_{\epsilon_j} + \delta) v_j,
\qquad
\norm{v_j}_{H^1} &=& 1,
\qquad
\epsilon_j &\to & 0\end{array}
\end{equation}
that converge weakly to a pair
\begin{equation}\begin{array}{lcl}
w_0 &=& (\mathcal{L}_0 + \delta)v_0,\end{array}
\end{equation}
the authors show that also $\mathcal{L}_{\epsilon} + \delta: H^1 \to L^2$
is invertible. To this end one needs to establish a lower bound for $\norm{w_0}_{L^2}$,
which can be achieved by exploiting inequalities of the form
\begin{equation}
\label{eq:int:coercv:est:diffusion}
\begin{array}{lclcl}
\big\langle
  v( \cdot + \epsilon) + v(\cdot - \epsilon) - 2 v(\cdot) , v(\cdot)
\big\rangle_{L^2}  
&\le &0
, \qquad \qquad \langle v', v \rangle_{L^2}&= &0
\end{array}
\end{equation}
and using the bistable structure of the nonlinearity $g$.\\

In \cite{HJHFHNINFRANGE} we showed that these ideas
can be generalized to infinite-range versions
of the FitzHugh-Nagumo LDE (\ref{finiterangeversion}).
The key issue there, which we must also face in this paper, is that
problematic cross terms arise that must be kept under control when taking inner products.
We are aided in this respect
by the fact that the off-diagonal terms in the linearisation of (\ref{eq:int:2per:fhn}) are constant multiples of each other.\\

A second key complication that we encounter here is that
the scale separation in the diffusion terms prevents
us from using the direct multi-component analogue of the inequality
(\ref{eq:int:coercv:est:diffusion}). We must carefully include $\epsilon$-dependent weights
into our inner products to compensate for these imbalances.
This complicates the fixed-point argument used to control the nonlinear terms
during the construction of the travelling waves. In fact, it forces us to
take an additional spatial derivative of the travelling wave equations.\\

This latter situation was also encountered in \cite{HJHADPGRID}, where
the spectral convergence method was used to construct travelling wave solutions
to  adaptive-grid discretisations of the Nagumo PDE (\ref{eq:int:nagumo:pde}).
Further applications of this technique can be found
in \cite{HJHBDF, HJHFHNBDF}, where full spatial-temporal discretisations
of the Nagumo PDE (\ref{eq:int:nagumo:pde})
and the FitzHugh-Nagumo PDE (\ref{FHNPDE}) are considered.

\paragraph{Overview}
%This paper is organized as follows.
After stating our main results in \S\ref{sec:mr}
we apply the spectral convergence method discussed above
to the system of travelling wave equations (\ref{FHNMFDE})
in \S\ref{sec:lim:op}-\ref{singularoperator}.
This allows us to follow the spirit of
\cite[Thm. 1]{BatesInfRange} to establish the existence
of travelling waves in \S\ref{sectionexistence}.
In particular, we use a fixed point argument that
mimics the proof of the standard implicit function theorem.\\

We follow the approach developed in \cite{HJHFHNINFRANGE}
to analyze the spectral stability of these travelling waves
in \S\ref{sectionstability}.
In particular,
we recycle the spectral convergence argument
to analyze the linear operators $\overline{\mathcal{L}}_{\epsilon}$
that arise after linearizing (\ref{FHNMFDE}) around
the newly-found waves, instead of around the singular pulse
$\overline{U}_0$ defined in (\ref{eq:int:def:sing:orbit}).
The key complication here is that for fixed small values of $\epsilon > 0$
we need results on the invertibility
of $\overline{\mathcal{L}}_{\epsilon} + \lambda$
for all $\lambda$ in a half-strip. By contrast,
the spectral convergence method gives a range of
admissible values for $\epsilon > 0$ for each fixed $\lambda$.
Switching between these two points of view is a delicate
task, but fortunately the main ideas from \cite{HJHFHNINFRANGE} can be transferred
to this setting.\\
%in a relatively straightforward fashion.\\

The nonlinear stability of the travelling waves
can be inferred from their spectral stability in a relatively straightforward
fashion by appealing to the theory
developed in \cite{HJHSTBFHN} for discrete systems with finite range
interactions. A more detailed description of this
procedure in an infinite-range setting can be found in
\cite[\S 7-8]{HJHFHNINFRANGEFULL}.\\

\textbf{Acknowledgements.}\\
Both authors acknowledge support from the Netherlands Organization for Scientific Research (NWO) (grant 639.032.612).

\section{Main Results}\label{sec:mr}

Our main results concern the LDE
%following 2-periodic (FitzHugh-)Nagumo type system
\begin{equation}\label{ditishetprobleem}
\begin{array}{lcl}
\dot{u}_{j}(t)&=&
 d_j \D\big[u_{j+1}(t)+u_{j-1}(t)-2u_{j}(t)\big]
 +f_j\big(u_{j}(t),w_{j}(t) \big),
 \\[0.2cm]
\dot{w}_{j}(t)&=&g_j\big(u_{j}(t),w_{j}(t)\big),
\end{array}
\end{equation}
posed on the one-dimensional lattice $j \in \Z$,
where we take $u_{j}\in\R^n$ and $w_{j}\in \R^k$
for some pair of integers $n\geq 1$ and $k\geq 0$.
We assume that the system is 2-periodic in
the sense that
there exists a set of four nonlinearities
\begin{equation}
f_o: \Real^{n+k} \to \Real^n,
\qquad
f_e : \Real^{n+k} \to \Real^n,
\qquad
g_o: \Real^{n+k} \to \Real^k,
\qquad
g_e: \Real^{n+k} \to \Real^k
\end{equation}
for which we may write
\begin{equation}\begin{array}{lcl}
 (d_j, f_j, g_j) &=& \left\{ \begin{array}{lcl}
  (\epsilon^{-2}, f_o,g_o ) & & \hbox{for odd } j, \\[0.2cm]
  (1 , f_e, g_e ) & & \hbox{for even }j .
  \end{array} \right.\end{array}
\end{equation}

Introducing the shorthand notation
\begin{equation}
\begin{array}{lclcl}
F_o(u,w) &=&  \big( f_o(u,w), g_o(u,w) \big), %^{T},
\qquad
F_e(u,w) &=&  \big( f_e(u,w), g_e(u,w) \big), %^{T},
\end{array}
\end{equation}
we impose the following structural condition on our system
that concerns the roots of the nonlinearities
$F_o$ and $F_e$. These roots
correspond with temporal equilibria of \sref{ditishetprobleem}
that have a spatially homogeneous $u$-component. On the other hand,
the $w$-component of these equilibria is allowed to be 2-periodic.
%but the $u$-component must be spatially homogeneous.
%
\begin{assumption}{N1}{\text{H}}\label{aannamesconstanten} The matrix
$\D \in \Real^{n \times n}$ is a diagonal
matrix with strictly positive diagonal entries.
In addition, the nonlinearities $F_o$ and $F_e$ are $C^3$-smooth and there exist
four vectors
\begin{equation}\begin{array}{lclcl}
U_e^{\pm}&=&(u_e^{\pm},w_e^{\pm}) \in \Real^{n + k },
\qquad \qquad
U_o^{\pm}&=&(u_o^{\pm},w_o^{\pm}) \in \Real^{n + k },\end{array}
\end{equation}
for which we have the identities
$u^-_o = u^-_e$  and $u^+_o= u^+_e$,
together with
\begin{equation}\begin{array}{lclcl}
F_o( U^\pm_o) &=& F_e(U^\pm_e) &=& 0.\end{array}
\end{equation}
\end{assumption}

We emphasize that any subset of the four vectors $U^\pm_o$ and $U^\pm_e$
is allowed to be identical. In order to
address the temporal stability of these equilibria,
we introduce two separate auxiliary conditions on triplets
\begin{equation}
\big( G , U^-, U^+ \big) \in C^1\big( \R^{n + k} ; \R^{n+k} \big) \times \R^{n + k}
 \times \R^{n+k},
\end{equation}
which are both stronger\footnote{
See the proof of Lemma \ref{Lepsfredholm} for details.
}
than the requirement that all the eigenvalues
of $DG(U^\pm)$ have
strictly negative real parts.
\begin{assumption}{$\alpha$}{\text{h}}\label{aannamesconstanten1}
The matrices $-DG(U^-)$ and $-DG(U^+)$ are positive definite.\end{assumption}
\begin{assumption}{$\beta$}{\text{h}}\label{aannamesconstanten2}
For any $U \in \R^{n+k}$, write $DG(U)$ in the block form
%Writing $DG(U)$ in the block form
\begin{equation}\label{blockstructure}
\begin{array}{lcl}
DG(U)&=&\left(\begin{array}{ll}G_{1,1}(U)&G_{1,2}(U)\\ G_{2,1}(U)& G_{2,2}(U)
\end{array}\right)
\end{array}
\end{equation}
with $G_{1,1}(U) \in \Real^{n \times n}$. Then the matrices
$-G_{1,1}(U^-),-G_{1,1}(U^+),-G_{2,2}(U^-)$ and $-G_{2,2}(U^+)$ are positive definite.
In addition, there exists a constant $\Gamma > 0$
so that $G_{1,2}(U)=-\Gamma G_{2,1}(U)^T$ holds for all $U \in \R^{n \times k}$.
\end{assumption}

As an illustration, we pick $0 < a < 1$
and write
\begin{equation}
\begin{array}{lcl}
G_{\mathrm{ngm}}(u) &=&u(1-u)(u-a)
\end{array}
\end{equation}
for the nonlinearity associated with the Nagumo equation,
together with
\begin{equation}
\label{eq:mr:def:g:fhn}
\begin{array}{lcl}
G_{\mathrm{fhn};\rho,\gamma}(u,w)  &=&\left(\begin{array}{l}
u(1-u)(u-a)-w\\[0.2cm]
\rho\big[u-\gamma w\big]
\end{array}\right)
\end{array}
\end{equation}
for its counterpart corresponding to the FitzHugh-Nagumo system.
It can be easily verified that the triplet
$(G_{\mathrm{ngm}},0,1)$ satisfies (\asref{aannamesconstanten1}{\text{h}}),
while the triplet $(G_{\mathrm{fhn};\rho, \gamma},0,0)$ satisfies
(\asref{aannamesconstanten2}{\text{h}})
for $\rho > 0$ %\leq 1$
and $\gamma > 0$,
with $\Gamma = \rho^{-1}$.
When $a > 0$ is sufficiently small, the Jacobian
$DG_{\mathrm{fhn};\rho,\gamma}(0,0)$ has a pair of complex eigenvalues
with negative real part. In this case (\asref{aannamesconstanten1}{\text{h}})
may fail to hold.\\ %need not be satisfied.\\

The following assumption states that the even and odd subsystems
must both satisfy one of the two auxiliary conditions above. We emphasize
however that this does not necessarily need to be the same
condition for both systems.
\begin{assumption}{N2}{\text{H}}\label{aannamesconstanten3}
The triplet $(F_o, U^-_o, U^+_o)$ satisfies
either
(\asref{aannamesconstanten1}{\text{h}}) or (\asref{aannamesconstanten2}{\text{h}}).
The same holds for the triplet $(F_e, U^-_e, U^+_e)$.
% satisfies
%either
%(\asref{aannamesconstanten1}{\text{h}}) or (\asref{aannamesconstanten2}{\text{h}}).
\end{assumption}

We intend to find functions
\begin{equation}
(u_{\epsilon}, w_{\epsilon}) : \Real \to \ell^\infty( \mathbb{Z}; \Real^n ) \times \ell^\infty(\mathbb{Z} ; \Real^k )
\end{equation}
that take the form
\begin{equation}
\label{eq:mr:trv:wave:ansatz}
\begin{array}{lcl}
(u_{\epsilon},w_{\epsilon})_j(t)&=&\begin{cases}
(\overline{u}_{o;\epsilon},\overline{w}_{o;\epsilon})(j+c_\epsilon t),\enskip
   &\text{for odd }j,\\[0.2cm]
(\overline{u}_{e;\epsilon},\overline{w}_{e;\epsilon})(j+c_\epsilon t),\enskip
    &\text{for even }j\end{cases}
\end{array}
\end{equation}
and satisfy (\ref{ditishetprobleem}) for all $t \in \Real$. The waveprofiles
are required to be $C^1$-smooth and satisfy the limits
\begin{equation}\begin{array}{lclcl}
\lim_{\xi \to \pm \infty} \big(\overline{u}_o(\xi), \overline{w}_o(\xi) \big)& =& (u_o^\pm, w_o^\pm),
\qquad
\qquad
\lim_{\xi \to \pm \infty} \big(\overline{u}_e(\xi), \overline{w}_e(\xi) \big)& =& (u_e^\pm, w_e^\pm).\end{array}
\end{equation}

Substituting the travelling wave Ansatz (\ref{eq:mr:trv:wave:ansatz}) into the LDE (\ref{ditishetprobleem}) yields the
coupled system
\begin{equation}\label{nieuwetravellingwaveeq}
\begin{array}{lcl}
c_\epsilon\overline{u}_{o;\epsilon}'(\xi)&=&
  \frac{1}{\epsilon^2}\D \Delta_{\mathrm{mix}}[\overline{u}_{o;\epsilon},\overline{u}_{e;\epsilon}](\xi)
   +f_o\big(\overline{u}_{o;\epsilon}(\xi),\overline{w}_{o;\epsilon}(\xi)\big),
  \\[0.2cm]
c_\epsilon\overline{w}_{o;\epsilon}'(\xi)&=&
  g_o\big(\overline{u}_{o;\epsilon}(\xi),\overline{w}_{o;\epsilon}(\xi)\big),
\\[0.2cm]
c_\epsilon\overline{u}_{e;\epsilon}'(\xi)&=&
  \D\Delta_{\mathrm{mix}}[\overline{u}_{e;\epsilon},\overline{u}_{o;\epsilon}](\xi)
    +f_e\big(\overline{u}_{e;\epsilon}(\xi),\overline{w}_{e;\epsilon}(\xi) \big),
\\[0.2cm]
c_\epsilon\overline{w}_{e;\epsilon}'(\xi)&=&
  g_e\big(\overline{u}_{e;\epsilon}(\xi),\overline{w}_{e;\epsilon}(\xi)\big),
\end{array}
\end{equation}
in which we have introduced the shorthand
\begin{equation}
\begin{array}{lcl}
\Delta_{\mathrm{mix}}[\phi,\psi](\xi)&=&\psi(\xi+1)+\psi(\xi-1)-2\phi(\xi).
\end{array}
\end{equation}

Multiplying the first line of (\ref{nieuwetravellingwaveeq}) by $\epsilon^2$ and
taking the formal limit $\epsilon\downarrow 0$, we obtain
the identity
\begin{equation}
\begin{array}{lcl}\label{ongesplitstenagumo}
0&=&\D\Delta_{\mathrm{mix}}[\overline{u}_{o;0},\overline{u}_{e;0}](\xi),\\[0.2cm]
\end{array}
\end{equation}
which can be explicitly solved to yield
\begin{equation}
\label{eq:mr:id:for:ovl:u:zero}\begin{array}{lcl}
\overline{u}_{o;0}(\xi)
 &=& \frac{1}{2} \overline{u}_{e;0}(\xi+1)
   + \frac{1}{2} \overline{u}_{e;0}(\xi-1) .\end{array}
\end{equation}
In the $\epsilon \downarrow 0$ limit, the even
subsystem of (\ref{nieuwetravellingwaveeq})
hence decouples and becomes
\begin{equation}\label{nagumolde:even}
\begin{array}{lcl}
c_0\overline{u}_{e;0}'(\xi)&=&\frac{1}{2}\D\Big[\overline{u}_{e;0}(\xi+2)+\overline{u}_{e;0}(\xi-2)-2\overline{u}_{e;0}(\xi)\Big]
  +f_e\big(\overline{u}_{e;0}(\xi),\overline{w}_{e;0}(\xi)\big),
 \\[0.2cm]
c_0\overline{w}_{e;0}'(\xi)&=&g_e\big(\overline{u}_{e;0}(\xi),\overline{w}_{e;0}(\xi)\big) .
\end{array}
\end{equation}
We require this limiting even system to have a travelling wave solution
that connects $U_e^-$ to $U_e^+$.

\begin{assumption}{W1}{\text{H}}\label{aannamespuls:even}
There exists $c_0 \neq 0$ for which the system
(\ref{nagumolde:even})
has a $C^1$-smooth solution $\overline{U}_{e;0} = (\overline{u}_{e;0},\overline{w}_{e;0})$
that satisfies the limits
\begin{equation}\begin{array}{lcl}
\lim_{\xi \to \pm \infty} \big(\overline{u}_{e;0}(\xi) ,\overline{w}_{e;0}(\xi) \big)
 &=& ( u^\pm_e, w^\pm_e ) .\end{array}
\end{equation}
\end{assumption}

Finally, taking $\epsilon \downarrow 0$ in the second line of
(\ref{nieuwetravellingwaveeq}) and applying (\ref{eq:mr:id:for:ovl:u:zero}),
we obtain the identity
\begin{equation}\label{eq:mr:eq:fro:ovl:w:o:zero}
\begin{array}{lcl}
c_0\overline{w}_{o;0}'(\xi)&=&g_o\Big(  \frac{1}{2} \overline{u}_{e;0}(\xi+1)
   + \frac{1}{2} \overline{u}_{e;0}(\xi-1),\overline{w}_{o;0}(\xi)\Big),
\end{array}
\end{equation}
in which $\overline{w}_{o;0}$ is the only remaining unknown.
We impose the following compatibility condition on this system.
%, which
%is automatically satisfied\footnote{
%  In this case we have $w_o^\pm = 0$ and the nonlinearity $g_o$ is in fact linear
%  and invertible with respect to $\overline{w}_{o;0}$ on account of Lemma \ref{eigenschappen2ecompL0}.
%} for the 2-periodic FitzHugh-Nagumo equation (\ref{eq:int:2per:fhn}).
%
\begin{assumption}{W2}{\text{H}}\label{aannamespuls:odd}
The equation (\ref{eq:mr:eq:fro:ovl:w:o:zero})
has a $C^1$-smooth solution $\overline{w}_{o;0}$ that satisfies the limits
\begin{equation}
\begin{array}{lcl}
\lim_{\xi\rightarrow\pm\infty}\overline{w}_{o;0}(\xi)&=&w_o^{\pm}.
\end{array}
\end{equation}\end{assumption}

Upon writing
\begin{equation}\begin{array}{lclcl}
  \overline{U}_{0}&=&(\overline{U}_{o;0},\overline{U}_{e;0} )
  &=&(\overline{u}_{o;0},\overline{w}_{o;0},\overline{u}_{e;0},\overline{w}_{e;0}),\end{array}
\end{equation}
we intend to seek a branch of solutions to (\ref{nieuwetravellingwaveeq})
that bifurcates off the singular travelling wave $(\overline{U}_0, c_0)$.
In view of the limits
\begin{equation}
    \begin{array}{lcl}
    \lim\limits_{\xi\rightarrow\pm\infty}(\overline{U}_{o;0},\overline{U}_{e;0})(\xi)&=&(U_o^{\pm},U_e^{\pm}),
    \end{array}
\end{equation}
we introduce the spaces
\begin{equation}
\begin{array}{lclclclcl}
\mathbf{H}^1_e &=& \mathbf{H}^1_o &=&
H^1(\R;\R^n) \times H^1(\Real; \Real^k),
 \qquad \qquad
\mathbf{L}^2_e& =& \mathbf{L}^2_o  &=&
L^2(\R;\R^n) \times L^2(\Real; \Real^k)
\end{array}
\end{equation}
to analyze the perturbations from $\overline{U}_0$.\\

Linearizing (\ref{nagumolde:even}) around the solution $\overline{U}_{e;0}$,
we obtain the linear operator $\overline{L}_{e} : \mathbf{H}^1_e \to \mathbf{L}^2_e$
that acts as
\begin{equation}\label{defoverlineL}
\begin{array}{lcl}
\overline{L}_{e}&=&
  c_0 \frac{d}{d\xi} - DF_e(\overline{U}_{e;0} )
  - \frac{1}{2} \left(\begin{array}{ll} \mathcal{D}  (S_2 - 2) & 0 \\[0.2cm] 0 & 0
    \end{array} \right),
\end{array}
\end{equation}
in which we have introduced the notation
\begin{equation}\begin{array}{lcl}
[ S_2 \phi](\xi)&=&\phi(\xi+2)+\phi(\xi-2) .\end{array}
\end{equation}
Our perturbation argument to construct solutions of (\ref{nieuwetravellingwaveeq})
requires $\overline{L}_e$ to have an isolated simple eigenvalue at the origin.
%For the FitzHugh-Nagumo equation (\ref{eq:int:2per:fhn}) this has been established \cite{HJHSTBFHN}
%in the setting where $\rho > 0$ and $\gamma > 0$ are sufficiently small.

\begin{assumption}{S1}{\text{H}}\label{extraaannamespuls}
There exists $\delta_e > 0$ so that the operator $\overline{L}_{e}+\delta$ is a
Fredholm operator with index 0 for each $0\leq\delta <\delta_e$.
It has a simple eigenvalue in
$\delta=0$, i.e., we have $\mathrm{Ker} \big( \overline{L}_e \big) = \mathrm{span}( \overline{U}_{e;0}' )$
and $\overline{U}_{e;0}' \notin \mathrm{Range} \big( \overline{L}_e \big)$.
\end{assumption}

We are now ready to formulate our first main result,
which states that \sref{nieuwetravellingwaveeq} admits
a branch of solutions for small $\epsilon > 0$
that converges to the singular wave $(\overline{U}_{0} , c_0)$ as $\epsilon \downarrow 0$.
Notice that the $\epsilon$-scalings
on the norms of $\Phi'_{\epsilon}$
and $\Phi''_{\epsilon}$  are considerably better
than those suggested by a direct
inspection of \sref{nieuwetravellingwaveeq}.

\begin{theorem}[{See \S \ref{sectionexistence}}]\label{maintheorem} Assume that (\asref{aannamesconstanten}{\text{H}}),
(\asref{aannamesconstanten3}{\text{H}}), (\asref{aannamespuls:even}{\text{H}}), (\asref{aannamespuls:odd}{\text{H}}) and (\asref{extraaannamespuls}{\text{H}}) are satisfied.
There exists a constant $\epsilon_*>0$ so that for each $0<\epsilon<\epsilon_*$, there exist
$c_\epsilon\in\R$ and
$\Phi_\epsilon = (\Phi_{o;\epsilon}, \Phi_{e;\epsilon} )%=(\phi_{o;\epsilon},\psi_{o;\epsilon},\phi_{e;\epsilon},\psi_{e;\epsilon})
\in\mathbf{H}^1_o \times \mathbf{H}^1_e$ for which the function
\begin{equation}\begin{array}{lcl}
\overline{U}_\epsilon &=&%(\overline{u}_{o;\epsilon},\overline{w}_{o;\epsilon},\overline{u}_{e;\epsilon},\overline{w}_{e;\epsilon})
  %:=
  \overline{U}_{0}
  +\Phi_\epsilon\end{array}
\end{equation}
is a solution of the travelling wave system (\ref{nieuwetravellingwaveeq}) with wave speed $c = c_\epsilon$.
In addition, we have the limit
%\begin
%$\Phi_\epsilon$ satisfies the limit
\begin{equation}
\begin{array}{lcl}
\lim_{\epsilon \downarrow 0}\,\Big[
  \nrm{ \epsilon  \Phi_{o;\epsilon}''}_{\mathbf{L}^2_o}
  + \nrm{ \Phi_{e;\epsilon}''}_{\mathbf{L}^2_e}
  + \nrm{\Phi_{\epsilon}'}_{\mathbf{L}^2_o \times \mathbf{L}^2_e }
  + \nrm{\Phi_{\epsilon}}_{\mathbf{L}^2_o \times \mathbf{L}^2_e }
  + \abs{c_{\epsilon} - c_0 } \Big] &= & 0
%\lim\limits_{\epsilon\downarrow 0}\nrm{(\M_{\epsilon}^{1,2}\Phi_\epsilon'',\Phi_\epsilon',\Phi_\epsilon)}_{\mathbf{L}^2}&=&0
\end{array}
\end{equation}
and the function $\overline{U}_{\epsilon}$ is locally unique up to translation.
\end{theorem}

In order to show that our new-found travelling wave solution is stable
under the flow of the LDE (\ref{ditishetprobleem}),
we need to impose the following extra assumption on the operator $\overline{L}_{e}$.
To understand the restriction on $\lambda$, we recall that the spectrum of
$\overline{L}_e$ admits the periodicity
$\lambda \mapsto \lambda + 2 \pi i c_0$.

\begin{assumption}{S2}{\text{H}}\label{extraextraaannamespuls}
There exists a constant $\lambda_e>0$
so that the operator
$\overline{L}_{e}+\lambda: \mathbf{H}^1_e \to \mathbf{L}^2_e$ is invertible
for all $\lambda \in \mathbb{C} \setminus  2 \pi i c_0 \Z$ that have $\Re \lambda \ge - \lambda_e$.
\end{assumption}

Together with (HS1) this condition states that the wave $(\overline{U}_{e;0} , c_0)$
for the limiting even system (\ref{nagumolde:even}) is spectrally stable.
Our second main theorem shows that this can be generalized to a nonlinear stability
result for the wave solutions \sref{eq:mr:trv:wave:ansatz}
of the full system \sref{ditishetprobleem}.

\begin{theorem}[{See \S \ref{sectionstability}}]\label{nonlinearstability}  Assume that (\asref{aannamesconstanten}{\text{H}}),
(\asref{aannamesconstanten3}{\text{H}}), (\asref{aannamespuls:even}{\text{H}}),
(\asref{aannamespuls:odd}{\text{H}}), (\asref{extraaannamespuls}{\text{H}}) and
(\asref{extraextraaannamespuls}{\text{H}}) are satisfied
and pick a sufficiently small $\epsilon > 0$.
%Fix $0<\epsilon<\epsilon_{**}$ and $1\leq p\leq \infty$.
Then there exist constants $\delta>0$, $C>0$ and $\beta>0$ %, which do not depend on $p$,
%such that for all
so that for all $1 \le p \le \infty$ and all
initial conditions
\begin{equation}
(u^0, w^0) \in \ell^\infty( \mathbb{Z}; \Real^n ) \times \ell^\infty(\mathbb{Z} ; \Real^k )
\end{equation}
that admit the bound
\begin{equation}\begin{array}{lclcl}
E_0 &:= &\nrm{u^0 - u_{\epsilon}(0)}_{\ell^p(\mathbb{Z} ; \Real^n) }
  + \nrm{w^0 - w_{\epsilon}(0)}_{\ell^p(\mathbb{Z} ; \Real^k) } &< &\delta,\end{array}
\end{equation}
there exists an asymptotic
phase shift $\tilde{\theta}\in\R$ such that the solution $(u,w)$ of (\ref{ditishetprobleem}) with
the initial condition $(u,w)(0)=(u^0, w^0)$
satisfies the estimate
\begin{equation}\begin{array}{lcl}
\nrm{u(t) - u_{\epsilon}(t + \tilde{\theta})}_{\ell^p(\mathbb{Z} ; \Real^n) }
  + \nrm{w(t) - w_{\epsilon}(t + \tilde{\theta})}_{\ell^p(\mathbb{Z} ; \Real^k) }
&\leq& Ce^{-\beta t}  E_0
\end{array}
\end{equation}
for all $t>0$.
\end{theorem}

Our final result shows that our framework is broad enough
to cover the two-periodic FitzHugh-Nagumo system (\ref{eq:int:2per:fhn}).
We remark that the condition on $\gamma_e$ ensures that
$(0,0)$ is the only spatially homogeneous
equilibrium for the limiting even subsystem \sref{eq:int:wv:eq:hom:fhn}.
This allows us to apply the spatially homogeneous results obtained in
\cite{HJHFZHNGM,HJHSTBFHN}.

\begin{corollary}
Consider the LDE (\ref{eq:int:2per:fhn})
and suppose that $\gamma_o> 0$ and $\rho_o > 0$ both hold.
Suppose furthermore that $a_e$ is sufficiently far away from $\frac{1}{2}$,
that $0 < \gamma_e < 4(1-a_e)^{-2}$ and that $\rho_e > 0$ is sufficiently small.
Then for each
sufficiently small $\epsilon > 0$, there exists
a nonlinearly stable travelling pulse solution of the form (\ref{eq:mr:trv:wave:ansatz})
that satisfies the limits
\begin{equation}\begin{array}{lclcl}
\lim_{\xi \to \pm \infty} \big(\overline{u}_o(\xi), \overline{w}_o(\xi) \big) &=& (0,0),
\qquad
\qquad
\lim_{\xi \to \pm \infty} \big(\overline{u}_e(\xi), \overline{w}_e(\xi) \big) &=& (0,0).\end{array}
\end{equation}
\end{corollary}
\begin{proof}
Assumption (HN1) can be verified directly, while (HN2)
follows from the discussion above concerning the nonlinearity
$G_{\mathrm{fhn}; \rho, \gamma}$ defined in
\sref{eq:mr:def:g:fhn}. Assumption (HW1)
follows from the existence theory developed in
\cite{HJHFZHNGM}, while (HS1) and (HS2)
follow from the spectral analysis
in \cite{HJHSTBFHN}. The remaining condition (HW2)
can be verified by noting
that the nonlinearity $g_o$ is in fact linear
and invertible with respect to $\overline{w}_{o;0}$ on account of
Lemma \ref{eigenschappen2ecompL0} below.
\end{proof}

\section{The limiting system}\label{sec:lim:op}

In this section we analyze
the linear operator
that is associated to the limiting
system that arises by combining
(\ref{nagumolde:even})
and (\ref{eq:mr:eq:fro:ovl:w:o:zero}).
In order to rewrite this system in a compact fashion, we introduce the notation
\begin{equation}\label{def:Si}
\begin{array}{lcl}
[ S_i \phi](\xi)&=&\phi(\xi+i)+\phi(\xi-i)\end{array}
\end{equation}
together with the
$(n+k) \times (n+k)$-matrix $J_{\mathcal{D}}$ that has the block structure
\begin{equation}
\begin{array}{lcl}
J_{\mathcal{D}}&=&\left(\begin{array}{ll}\mathcal{D}&0\\ 0&0
\end{array}\right).
\end{array}
\end{equation}
This allows us to recast \sref{defoverlineL}
in the shortened form
\begin{equation}\label{defoverlineL:short}
\begin{array}{lcl}
\overline{L}_{e}&=&c_0\frac{d}{d \xi}-\frac{1}{2}J_{\D} (S_2 - 2 ) -DF_e(\overline{U}_{e;0}).
\end{array}
\end{equation}
One can associate a formal adjoint $\overline{L}_{e}^{\mathrm{adj}}: \mathbf{H}^1_e \to \mathbf{L}^2_e$
to this operator by writing
\begin{equation}\label{def:adj:Le}
\begin{array}{lcl}
%\overline{L}_{e}^{-}
\overline{L}_{e}^{\mathrm{adj}}
&=&-c_0\frac{d}{d\xi} -\frac{1}{2} J_{\D} (S_2 - 2) -DF_e(\overline{U}_{e;0})^{T}.
\end{array}
\end{equation}

Assumption (HS1) together with the Fredholm theory developed in \cite{MPA} imply that
\begin{equation}\begin{array}{lcl}
\text{ind}(\overline{L}_e) &=& -\text{ind}(\overline{L}_e^{\mathrm{adj}})\end{array}
\end{equation}
holds for the Fredholm indices of these operators, which are defined as
\begin{equation}\begin{array}{lcl}
\mathrm{ind}(L) & =& \dim\big(\ker(L)\big)-\codim\big(\Range(L) \big).\end{array}
\end{equation}
In particular,
(HS1) implies that there exists a function
\begin{equation}
\overline{\Phi}_{e;0}^{\mathrm{adj}} \in \mathrm{Ker} (\overline{L}_{e}^{\mathrm{adj}})\subset \mathbf{H}_e^1
\end{equation}
that can be normalized to have
\begin{equation}
\begin{array}{lcl}
\ip{\overline{U}_{e;0}',\overline{\Phi}_{e;0}^{\mathrm{adj}}}_{\mathbf{L}^2_e}
&=&1.
\end{array}
\end{equation}

We also introduce the operator
$\overline{L}_o: H^1(\Real; \Real^k ) \to L^2(\Real; \Real^k)$
associated to the linearization of
(\ref{eq:mr:eq:fro:ovl:w:o:zero}) around
$\overline{U}_{o;0}$,
which acts as
\begin{equation}
\begin{array}{lcl}
\overline{L}_{o} &=&
c_0\frac{d}{d\xi}-D_2g_o(\overline{U}_{o;0}) .
\end{array}
\end{equation}
In order to couple this operator with $\overline{L}_e$,
we introduce the spaces
\begin{equation}
    \begin{array}{lclcl}
      \mathbf{H}_\diamond^1 &=&

      H^1(\Real; \Real^k) \times \mathbf{H}^1_e,
      \qquad
      \qquad
    \mathbf{L}_\diamond^2 &=&

      L^2( \Real; \Real^k) \times \mathbf{L}^2_e,
    \end{array}
\end{equation}
together with the operator
\begin{equation}
\begin{array}{lcl}
\mathcal{L}_{\diamond;\delta}:\mathbf{H}_\diamond^1&\rightarrow & \mathbf{L}_\diamond^2
\end{array}
\end{equation}
that acts as
\begin{equation}\label{defL0}
\begin{array}{lcl}
\mathcal{L}_{\diamond;\delta}&=&\left(\begin{array}{cc}
   \overline{L}_{o} + \delta & 0\\[0.2cm]
   0 & \overline{L}_{e} + \delta
\end{array}\right) .
\end{array}
\end{equation}
Our first main result shows that $\mathcal{L}_{\diamond;\delta}$
inherits several properties of $\overline{L}_e + \delta$.

\begin{proposition}\label{eigenschappenL0:a} Assume that (\asref{aannamesconstanten}{\text{H}}),
(\asref{aannamesconstanten3}{\text{H}}),
(\asref{aannamespuls:even}{\text{H}}), (\asref{aannamespuls:odd}{\text{H}}) and (\asref{extraaannamespuls}{\text{H}}) are satisfied.
Then there exist constants $\delta_{\diamond} > 0$ and $C_{\diamond} > 0$ so that the following holds true.
\begin{enumerate}[label=(\roman*)]
\item For every $0 < \delta < \delta_{\diamond}$, the operator $\mathcal{L}_{\diamond,\delta}$ is invertible as a map from $\mathbf{H}_{\diamond}^1$ to $\mathbf{L}_{\diamond}^2$.
\item For any $\Theta_{\diamond} \in \mathbf{L}_{\diamond}^2$ and $0 < \delta < \delta_{\diamond}$
the function $\Phi_{\diamond} = \mathcal{L}_{\diamond, \delta}^{-1} \Theta_\diamond \in \mathbf{H}_{\diamond}^1$ satisfies the bound
\begin{equation}
\begin{array}{lcl}
\nrm{\Phi_{\diamond}}_{\mathbf{H}_{\diamond}^1}&\leq & C_{\diamond}\Big[\nrm{\Theta_{\diamond}}_{\mathbf{L}_{\diamond}^2}+\frac{1}{\delta}\big|\ip{ \Theta_{\diamond}, (0, \overline{\Phi}_{e;0}^{\mathrm{adj}} )}_{\mathbf{L}^2_{\diamond}}\big|\Big].
\end{array}
\end{equation}
\end{enumerate}
\end{proposition}

If (HS2) also holds, then we can consider compact sets $\lambda \in M \subset \mathbb{C}$
that avoid the spectrum of $\overline{L}_e$. To formalize this,
we impose the following assumption on $M$ and state our second main result.

\begin{assumption}{$M_{\lambda_0}$}{\text{h}}\label{Massumption} The set $M\subset \C$ is
compact with $2 \pi i c_0 \mathbb{Z} \cap  M = \emptyset$.
In addition, % recalling the constant $\lambda_*$ appearing in (\asref{extraextraaannamespuls}{\text{H}}),
we have $\Re \, \lambda \geq -\lambda_0$ for all $\lambda \in M$.
\end{assumption}

\begin{proposition}\label{eigenschappenL0:b}
Assume that (\asref{aannamesconstanten}{\text{H}}),
(\asref{aannamesconstanten3}{\text{H}}), (\asref{aannamespuls:even}{\text{H}}),
(\asref{aannamespuls:odd}{\text{H}}), (\asref{extraaannamespuls}{\text{H}}) and (\asref{extraextraaannamespuls}{\text{H}})
are all satisfied and pick a sufficiently small constant $\lambda_{\diamond} > 0$.
Then for any set $M \subset \C$ that satisfies %(\asrefrev{Massumption}{M_{\lambda_{\diamond}}})
(h$M_{\lambda_{\diamond}}$)
there exists a constant $C_{\diamond;M} > 0$ so that the following holds true.
\begin{itemize}
\item[(i)] For every $\lambda \in M$, the operator $\mathcal{L}_{\diamond,\lambda}$ is invertible as a map from $\mathbf{H}_{\diamond}^1$ to $\mathbf{L}_{\diamond}^2$.
\item[(ii)] For any $\Theta_{\diamond} \in \mathbf{L}_{\diamond}^2$ and $\lambda \in M$,
the function $\Phi_{\diamond} = \mathcal{L}_{\diamond, \lambda}^{-1} \Theta_{\diamond} \in \mathbf{H}_{\diamond}^1$ satisfies the bound
\begin{equation}
\begin{array}{lcl}
\nrm{\Phi_{\diamond}}_{\mathbf{H}_{\diamond}^1}&\leq & C_{\diamond;M} \nrm{\Theta_{\diamond}}_{ \mathbf{L}_{\diamond}^2}.
\end{array}
\end{equation}
\end{itemize}
\end{proposition}

\subsection{Properties of $\overline{L}_o$}
The assumptions (HS1) and (HS2) already contain the information on $\overline{L}_e$ that we require
to establish Propositions \ref{eigenschappenL0:a}-\ref{eigenschappenL0:b}.
Our task here is therefore to understand the operator $\overline{L}_o$.
As a preparation, we show that the top-left and bottom-right corners of the limiting Jacobians
$DF_o(U^\pm_o)$ are both negative definite,
which will help us to establish useful Fredholm properties.

\begin{lemma}\label{limitnegdef} Assume that (\asref{aannamesconstanten}{\text{H}})
and (\asref{aannamesconstanten3}{\text{H}})
%, (\asref{aannamespuls:even}{\text{H}}), (\asref{aannamespuls:odd}{\text{H}}) and
%(\asref{extraaannamespuls}{\text{H}})
are both satisfied. Then  the  matrices $D_1 f_{\#}(U^{\pm}_{\#})$ and
$D_2 g_{\#}( U^\pm_{\#})$ are all negative definite for each $\#\in\{o,e\}$.
\end{lemma}
\textit{Proof.} Note first that $D_1f_{\#}$ and $D_2g_{\#}$ correspond with $G_{1,1}$ respectively $G_{2;2}$ in the block structure
(\ref{blockstructure}) for $DF_{\#}$.
We hence  see that the matrices $D_1 f_{\#}(U^{\pm}_{\#})$ and $D_2 g_{\#}( U^\pm_{\#})$ are negative definite,
either directly by (h$\beta$)
or by the fact that they are principal submatrices of
$DF_{\#}(U_{\#}^{\pm})$, which are negative definite if (h$\alpha$) holds.\qed\\

\begin{lemma}\label{Loddfredholm} Assume that (\asref{aannamesconstanten}{\text{H}}),
(\asref{aannamesconstanten3}{\text{H}}), (\asref{aannamespuls:even}{\text{H}}) and (\asref{aannamespuls:odd}{\text{H}}) are satisfied.
Then there exists
$\lambda_o>0$ so that the operator
$\overline{L}_o+\lambda$ is Fredholm with index zero
for each  $\lambda\in\C$ with $\Re\lambda\geq-\lambda_o$.
\end{lemma}
\textit{Proof.}
%Pick a small $\beta_0>0$ and fix $\lambda\in\C$ with $\Re\lambda\geq-\beta_0$.
For any $0\leq \rho\leq 1$ and $\lambda \in \mathbb{C}$
we introduce the constant coefficient linear operator
$L_{\rho,\lambda}: H^1(\Real;\Real^k) \to L^2(\Real;\Real^k)$ that acts as
\begin{equation}
\begin{array}{lcl}
L_{\rho,\lambda}&=&c_0\frac{d}{d\xi}-\rho D_2g_o( U_o^- )-(1-\rho)D_2g_o(U_o^+)+\lambda
\end{array}
\end{equation}
%The corresponding characteristic function is given by
%For $y\in\R$ we compute
and has the characteristic function
\begin{equation}
\begin{array}{lcl}
\Delta_{L_{\rho,\lambda}}(z)&=&c_0 z-\rho D_2g_o(U_o^-)-(1-\rho)D_2g_o(U_o^+)+\lambda.
\end{array}
\end{equation}
Upon introducing the matrix
\begin{equation}\begin{array}{lcl}
B_{\rho} &=& -\rho D_2g_o(U_{o}^-)-(1-\rho)D_2g_o(U_{o}^+)
  -\rho D_2g_o(U_{o}^-)^{T}-(1-\rho)D_2g_o(U_{o}^+)^{T},\end{array}
\end{equation}
which is positive definite by Lemma \ref{limitnegdef}, we pick $\lambda_o>0$ in such a way that $B_{\rho}-2\lambda_o$ remains positive
definite for each $0\leq\rho\leq 1$. It is easy to check that the identity
\begin{equation}\begin{array}{lcl}
  \Delta_{L_{\rho,\lambda}}(iy)+\Delta_{L_{\rho,\lambda}}(iy)^\dagger &=& B_{\rho} + 2\Re\lambda\end{array}
\end{equation}
holds for any $y \in \mathbb{R}$. In particular,
if we assume that $\Re \lambda \ge - \lambda_o$ and that $\Delta_{L_{\rho,\lambda}}(iy) v_o = 0$ for some non-zero $v_o \in \mathbb{C}^k$,
$y \in \Real$ and $0 \le \rho \le 1$,
then we obtain the contradiction
\begin{equation}
\begin{array}{lcl}
0 & = & \Re \big[  v_o^\dagger \big[ \Delta_{L_{\rho}}(iy)+\Delta_{L_{\rho}}(iy)^\dagger \big] v_o \big]
\\[0.2cm]
& = & \Re v_o^\dagger \big[ B_{\rho} + 2 \Re\lambda \big] v_o
\\[0.2cm]
& > & 0 .
\end{array}
\end{equation}
Using \cite[Thm. A]{MPA} together with the spectral flow principle in
\cite[Thm. C]{MPA}, this implies that $\overline{L}_{o} + \lambda$ is a Fredholm operator with index zero.
\qed\\

\begin{lemma}\label{eigenschappen2ecompL0} Assume that (\asref{aannamesconstanten}{\text{H}}), (\asref{aannamesconstanten3}{\text{H}})
and (\asref{aannamespuls:even}{\text{H}}) and (\asref{aannamespuls:odd}{\text{H}}) are satisfied
and pick a sufficiently small constant $\lambda_o > 0$.
Then for any $\lambda\in\C$ with $\Re\lambda\geq -\lambda_o$ the operator $\overline{L}_o + \lambda$ is invertible
as a map from $H^1(\Real;\Real^k)$ into $L^2(\Real;\Real^k)$. In addition, for each
compact set
\begin{equation}
M \subset \{ \lambda: \Re \lambda \ge -\lambda_o \} \subset \mathbb{C}
\end{equation}
there exists a constant $K_{M} > 0$ so that the uniform bound
\begin{equation}
\label{eq:lim:sys:unif:bnd:l:zero}
\begin{array}{lcl}
\nrm{ \big[\overline{L}_o + \lambda ]^{-1} \chi_o }_{H^1(\Real;\Real^k)}&\leq &
  K_{M}\nrm{\chi_o}_{L^2(\Real;\Real^k)}
\end{array}
\end{equation}
holds for any $\chi_o \in L^2(\Real;\Real^k)$ and any $\lambda\in M$.
\end{lemma}
\textit{Proof.} Recall the constant $\lambda_o$
defined in Lemma \ref{Loddfredholm} and pick any $\lambda\in\C$ with $\Re\lambda\geq-\lambda_o$.
On account of Lemma \ref{Loddfredholm} it suffices to show that $\overline{L}_{o} + \lambda$ is injective.
Consider therefore any non-trivial $x \in \mathrm{Ker}\big(\overline{L}_o + \lambda\big)$, which
necessarily satisfies the ODE\footnote{The discussion at https://math.stackexchange.com/questions/2668795/bounded-solution-to-general-nonautonomous-ode gave us the inspiration for this approach.}
\begin{equation}\label{nonautODE}
\begin{array}{lcl}
x'(\xi)&=&\frac{1}{c_0}D_2g_o\big(\overline{U}_{o;0}(\xi) %\overline{u}_{o;0}(t),\overline{w}_{o;0}(t)
\big)x(\xi)-\frac{\lambda}{c_0}x(\xi)
\end{array}
\end{equation}
posed on $\C^k$. Without loss of generality we may assume that $c_0>0$.\\

Since $\overline{U}_{o;0}(\xi)\rightarrow U_o^{\pm}$ as $\xi\rightarrow \pm\infty$,
Lemma \ref{limitnegdef} allows us to pick a constant $m\gg 1$   in such a way that  the matrix
$-D_2g_o\big(\overline{U}_{o;0}(\xi)\big)-2\lambda_o$ is positive definite for each $|\xi|\geq m$,
possibly after decreasing the size of $\lambda_o > 0$.
Assuming that $\Re \lambda \ge - \lambda_o$ and picking any $\xi \leq - m$, we may hence compute
\begin{equation}
\begin{array}{lcl}
\frac{d}{d\xi}|x(\xi)|^2&=&2\Re\ip{x'(\xi),x(\xi)}_{\C^k}\\[0.2cm]
&=&\frac{2}{c_0}\Re\ip{D_2g_o\big(\overline{U}_{o;0}(\xi) \big)x(\xi), x(\xi)}_{\C^k}

  -\frac{2 \Re \lambda}{c_0} \langle x(\xi),x(\xi) \rangle_{\C^k}\\[0.2cm]
&\leq &-\frac{2 \lambda_o}{c_0}|x(\xi)|^2 ,
\end{array}
\end{equation}
%for any $\xi \leq -m$.
%
%Therefore we see that
%This implies that
which implies that
\begin{equation}
\begin{array}{lcl}
\Big(e^{\frac{2\lambda_o}{c_0}\xi}|x(\xi)|^2\Big)'&\leq &0 .
\end{array}
\end{equation}
Since $x$ cannot vanish anywhere as a non-trivial solution to a linear ODE,
we have
\begin{equation}
\begin{array}{lclcl}
|x(\xi)|^2&\geq &e^{- \frac{2\lambda_o}{c_0}(m+\xi)}|x(-m)|^2 &>& 0
\end{array}
\end{equation}
for $\xi \le -m$,
which means that $x(\xi)$ is unbounded. In particular, we see that $x \notin H^1(\Real;\Real^k)$,
which leads to the desired contradiction.
The uniform bound (\ref{eq:lim:sys:unif:bnd:l:zero}) follows easily
from continuity considerations. \qed\\

\textit{Proof of Proposition \ref{eigenschappenL0:a}.}
Since the operator $\overline{L}_{e}$ defined in (\ref{defoverlineL}) has a simple eigenvalue in zero,
we can follow the approach of \cite[Lemma 3.1(5)]{HJHFHNINFRANGE}
to pick  two constants
$\delta_{\diamond} >0$ and  $C > 0$ in such a way that
$\overline{L}_e + \delta:\mathbf{H}_e^1\rightarrow \mathbf{L}_e^2 $ is invertible with the bound
\begin{equation}
\begin{array}{lcl}
\nrm{ \big[ \overline{L}_e + \delta ]^{-1} (\theta_e, \chi_e) }_{\mathbf{H}^1_e}
&\leq & C\Big[\nrm{ (\theta_e,\chi_e)}_{\mathbf{L}^2_e}+\frac{1}{\delta}
  \big|\ip{  (\theta_e,\chi_e),\overline{\Phi}_{e;0}^{\mathrm{adj}}}_{\mathbf{L}^2_e}\big|\Big].
\end{array}
\end{equation}
for any $0 < \delta <\delta_{\diamond}$ and $(\theta_e, \chi_e) \in \mathbf{L}^2_e$. Combining this estimate with
Lemma \ref{eigenschappen2ecompL0} directly yields the desired properties.\qed\\

\textit{Proof of Proposition \ref{eigenschappenL0:b}.}
These properties can be established in a fashion
analogous to the proof of Proposition \ref{eigenschappenL0:a}.

\qed\\

%---

\section{Transfer of Fredholm properties}\label{singularoperator}

Our goal in this section is to lift the bounds obtained in {\S}\ref{sec:lim:op}
to the operators associated to the linearization of the full wave equation \sref{nieuwetravellingwaveeq}
around suitable functions. In particular, the arguments we develop here will be used in several different settings.
In order to accommodate this,
we introduce the following condition.
\begin{assumption}{Fam}{\text{h}}\label{familyassumption}
For each $\epsilon>0$ there is a function
$\tilde{U}_\epsilon=(\tilde{U}_{o;\epsilon},\tilde{U}_{e;\epsilon}) \in \mathbf{H}^1_o \times \mathbf{H}^1_e$
  %=(\tilde{u}_{o;\epsilon},\tilde{w}_{o;\epsilon},\tilde{u}_{e;\epsilon},\tilde{w}_{e;\epsilon})\in \mathbf{H}^1$
and a constant $\tilde{c}_\epsilon\neq 0$ such that $\tilde{U}_\epsilon-\overline{U}_{0}\rightarrow 0$ in $\mathbf{H}^1_o \times \mathbf{H}^1_e$ and
$\tilde{c}_\epsilon\rightarrow c_0$ as $\epsilon\downarrow 0$.  In addition,
there exists a constant $\tilde{K}_{\mathrm{fam}} > 0$ so that
\begin{equation}
\label{eq:trn:unif:bnd:in:h:fam}\begin{array}{lcl}
\abs{\tilde{c}_{\epsilon}} + \abs{\tilde{c}_\epsilon^{-1} } + \norm{\tilde{U}_{\epsilon}}_{\infty} &\le &\tilde{K}_{\mathrm{fam}}\end{array}
\end{equation}
holds for all $\epsilon > 0$.
\end{assumption}
In \S \ref{sectionexistence} we will pick $\tilde{U}_\epsilon=\overline{U}_{0}$ and $\tilde{c}_\epsilon=c_0$ in
(\asref{familyassumption}{\text{h}}) for all $\epsilon>0$. On the other hand, in \S \ref{sectionstability} we will
use the travelling wave solutions described in Theorem \ref{maintheorem}
to write $\tilde{U}_\epsilon=\overline{U}_\epsilon$ and $\tilde{c}_\epsilon=c_\epsilon$. We remark that
\sref{eq:trn:unif:bnd:in:h:fam} implies that
there exists a constant $\tilde{K}_{F} > 0$ for which the bound
\begin{equation}
\label{eq:cnst:kf:def}\begin{array}{lcl}
\nrm{DF_o(\tilde{U}_{o;\epsilon})}_{\infty} + \nrm{D^2F_o(\tilde{U}_{o;\epsilon})}_{\infty}
+ \nrm{DF_e(\tilde{U}_{e;\epsilon})}_{\infty} + \nrm{D^2F_e(\tilde{U}_{e;\epsilon})}_{\infty} &\leq & \tilde{K}_{F}\end{array}
\end{equation}
holds for all $\epsilon > 0$.
\\

For notational convenience, we introduce the product spaces
\begin{equation}\begin{array}{lclcl}
\mathbf{H}^1 &= &\mathbf{H}^1_o \times \mathbf{H}^1_e,
\qquad
\mathbf{L}^2 &=& \mathbf{L}^2_o \times \mathbf{L}^2_e.\end{array}
\end{equation}
Since we will need to consider complex-valued functions during our spectral analysis,
we also introduce the spaces
%At some point we will also consider complex valued functions. As such we introduce the spaces $\mathbf{L}^2_{\C}$ and $\mathbf{H}^1_{\C}$ by
\begin{equation}
    \begin{array}{lcl}
    \mathbf{L}^2_{\C}&=&\{\Phi+i\Psi:\Phi,\Psi\in\mathbf{L}^2\},\\[0.2cm]
    \mathbf{H}^1_{\C}&=&\{\Phi+i\Psi:\Phi,\Psi\in\mathbf{H}^1\}
    \end{array}
\end{equation}
and remark that any $L \in \mathcal{L}( \mathbf{H}^1; \mathbf{L}^2)$
can be interpreted as an operator in $\mathcal{L}(\mathbf{H}^1_{\mathbb{C}} ; \mathbf{L}^2_{\mathbb{C}})$
by writing
\begin{equation}
    \begin{array}{lcl}
    L(\Phi+i\Psi)&=&L\Phi+iL\Psi.
    \end{array}
\end{equation}
It is well-known that taking the complexification of an operator preserves injectivity, invertibility and other Fredholm properties.\\

Recall the family $(\tilde{U}_{\epsilon}, \tilde{c}_{\epsilon})$ introduced in  (\asref{familyassumption}{\text{h}}).
For any $\epsilon > 0$ and $\lambda\in\C$
we introduce the linear operator
\begin{equation}
\tilde{\mathcal{L}}_{\epsilon, \lambda} : \mathbf{H}^1_{\mathbf{C}} \to \mathbf{L}^2_{\mathbf{C}}
\end{equation}
that acts as
\begin{equation}\label{deftildeL}
\begin{array}{lcl}
\tilde{\mathcal{L}}_{\epsilon,\lambda}&=&\left(\begin{array}{ll}\tilde{c}_\epsilon\frac{d}{d\xi}+\frac{2}{\epsilon^2}J_\D-DF_o(\tilde{U}_{o;\epsilon})+\lambda & -\frac{1}{\epsilon^2}J_\D S_1\\
-J_\D S_1 & \tilde{c}_\epsilon\frac{d}{d\xi}+2J_\D-DF_e(\tilde{U}_{e;\epsilon})+\lambda\end{array}\right).
\end{array}
\end{equation}

In order to simplify our notation,
we introduce the diagonal matrices
\begin{equation}
\begin{array}{lcl}
\M_{\epsilon}^1&=&
\mathrm{diag}\big( \epsilon, 1, 1, 1 \big) ,
 \\[0.2cm]
\M_{\epsilon}^2&=&
\mathrm{diag}\big( 1, \epsilon, 1, 1 \big) ,
 \\[0.2cm]
\M_{\epsilon}^{1,2}&=&
\mathrm{diag}\big( \epsilon, \epsilon, 1, 1 \big) .
\end{array}
\end{equation}
In addition, we recall the sum $S_1$ defined in (\ref{def:Si})
and introduce the operator
\begin{equation}\label{Jmix}
\begin{array}{lcl}
J_{\mathrm{mix}}
&=&\left(\begin{array}{ll}
  -2 J_{\D}& J_{\D}S_1\\
  J_{\D}S_1 &- 2 J_{\D}
  \end{array} \right),
  \end{array}
\end{equation}
which allows us to restate
(\ref{deftildeL}) as
\begin{equation}\label{deftildeL:short}
\begin{array}{lcl}
\tilde{\mathcal{L}}_{\epsilon,\lambda}&=&
\tilde{c}_\epsilon\frac{d}{d \xi}
- \mathcal{M}^1_{1/\epsilon^2} J_{\mathrm{mix}}
- DF (\tilde{U}_{\epsilon} ) + \lambda.
\end{array}
\end{equation}
Our two main results generalize the bounds in Propositions \ref{eigenschappenL0:a} and \ref{eigenschappenL0:b}
to the current setting. The scalings on the odd variables allow us to obtain certain key estimates
that are required by the spectral convergence approach.

\begin{proposition}\label{theorem4equivalent} Assume that (\asref{familyassumption}{\text{h}}),
(\asref{aannamesconstanten}{\text{H}}), (\asref{aannamesconstanten3}{\text{H}}),
(\asref{aannamespuls:even}{\text{H}}), (\asref{aannamespuls:odd}{\text{H}}) and (\asref{extraaannamespuls}{\text{H}}) are satisfied.
Then there exist positive constants $C_0>0$ and $\delta_0>0$ together with a strictly positive
function $\epsilon_0: (0, \delta_0) \rightarrow \R_{>0}$,
%depending only on the choice of $\tilde{U}_\epsilon$ and $\tilde{c}_\epsilon$, in such a way
so that for each $0<\delta<\delta_0$ and $0<\epsilon<\epsilon_0(\delta)$
% we have that
the operator $\tilde{\mathcal{L}}_{\epsilon,\delta}$ %^+:\mathbf{H}^1\rightarrow\mathbf{L}^2$
is invertible and satisfies the bound
\begin{equation}\label{normboundmetc0fam}
\begin{array}{lcl}
\nrm{\mathcal{M}_\epsilon^{1,2}\Phi}_{\mathbf{H}^1}
 &\leq & C_0\Big[\nrm{\mathcal{M}_\epsilon^{1,2} \Theta}_{\mathbf{L}^2}
   +\frac{1}{\delta}\big|\ip{  \Theta,(0,\overline{\Phi}_{e;0}^{\mathrm{adj}})}_{\mathbf{L}^2}
     \big|\Big]
\end{array}
\end{equation}
for any $\Phi\in\mathbf{H}^1$ and $\Theta = \tilde{\mathcal{L}}_{\epsilon, \delta} \Phi$.
\end{proposition}

\begin{proposition}\label{theorem4equivalentcompact}
Assume that (\asref{familyassumption}{\text{h}}), (\asref{aannamesconstanten}{\text{H}}),
(\asref{aannamesconstanten3}{\text{H}}), (\asref{aannamespuls:even}{\text{H}}),
(\asref{aannamespuls:odd}{\text{H}}), (\asref{extraaannamespuls}{\text{H}}) and (\asref{extraextraaannamespuls}{\text{H}})
are all satisfied and pick a sufficiently small constant $\lambda_{0} > 0$.
Then for any set $M \subset \C$ that satisfies %(\asrefrev{Massumption}{M_{\lambda_{0}}})
(h$M_{\lambda_{0}}$),
there exist positive constants $C_M>0$ and $\epsilon_M>0$ %, depending only on the choice of $\tilde{U}_\epsilon$ and $\tilde{c}_\epsilon$ and the choice of $M$,
so that %in such a way that
for each $\lambda\in M$ and $0<\epsilon<\epsilon_M$ %we have that
the operator $\tilde{\mathcal{L}}_{\epsilon,\lambda}$ %^+:\mathbf{H}^1\rightarrow\mathbf{L}^2$ is invertible, together with the bound
is invertible and satisfies the bound
\begin{equation}\label{normboundmetc0}
\begin{array}{lcl}
\nrm{\mathcal{M}_\epsilon^{1,2}\Phi}_{\mathbf{H}^1_{\C}}&\leq &
   C_M \nrm{\mathcal{M}_\epsilon^{1,2} \Theta}_{\mathbf{L}^2_{\C}}
\end{array}
\end{equation}
for any $\Phi\in\mathbf{H}^1_{\C}$ and $\Theta = \tilde{\mathcal{L}}_{\epsilon, \lambda} \Phi$.
\end{proposition}

By using bootstrapping techniques it is possible to obtain variants
of the estimate in Proposition \ref{theorem4equivalent}. Indeed,
it is possible to remove the scaling on the first component of $\Phi$ (but not on the first component of $\Phi'$).
\begin{corollary}\label{theorem4equivalentspecific1}
Consider the setting of Proposition \ref{theorem4equivalent}.
%Assume that (\asref{familyassumption}{\text{h}}), (\asref{aannamesconstanten}{\text{H}}),(\asref{aannamesconstanten3}{\text{H}}),
%(\asref{aannamespuls:even}{\text{H}}), (\asref{aannamespuls:odd}{\text{H}}) and (\asref{extraaannamespuls}{\text{H}}) are satisfied.
Then for each $0<\delta<\delta_0$ and $0<\epsilon<\epsilon_0(\delta)$,
the operator $\tilde{\mathcal{L}}_{\epsilon,\delta}$ satisfies the bound
\begin{equation}\label{normboundmetc0fam1}
\begin{array}{lcl}
\nrm{\mathcal{M}_\epsilon^{1,2}\Phi'}_{\mathbf{L}^2}
+ \nrm{\mathcal{M}_\epsilon^{2}\Phi}_{\mathbf{L}^2}
 &\leq & C_0\Big[\nrm{\mathcal{M}_\epsilon^{1,2} \Theta}_{\mathbf{L}^2}
   +\frac{1}{\delta}\big|\ip{  \Theta,(0,\overline{\Phi}_{e;0}^{\mathrm{adj}})}_{\mathbf{L}^2}
     \big|\Big]
\end{array}
\end{equation}
for any $\Phi\in\mathbf{H}^1$ and $\Theta = \tilde{\mathcal{L}}_{\epsilon, \delta} \Phi$,
possibly after increasing $C_0 > 0$.
\end{corollary}
\textit{Proof.} Write $\Phi=(\phi_o,\psi_o,\phi_e,\psi_e)$ and $\Theta=(\theta_o,\chi_o,\theta_e,\chi_e)$. Note that the first component of the equation $\Theta = \tilde{\mathcal{L}}_{\epsilon, \delta} \Phi$ yields
\begin{equation}
\begin{array}{lcl}
2\D\phi_o  &=&\D S_1\phi_e
-\epsilon^2 \tilde{c}_{\epsilon}\phi_o'
+\epsilon^2 D_1f_o(\tilde{U}_{o;\epsilon})\phi_o+\epsilon^2 D_2f_o(\tilde{U}_{o;\epsilon})\psi_o
-\delta\epsilon^2\phi_o+\epsilon^2\theta_o.
\end{array}
\end{equation}
Recall the constants $\tilde{K}_{\mathrm{fam}}$ and $\tilde{K}_{F}$ from (\ref{eq:trn:unif:bnd:in:h:fam}) and (\ref{eq:cnst:kf:def}) respectively
and write
\begin{equation}\begin{array}{lclcl}
d_{\min}&=&\min_{1\leq i\leq n}\D_{i,i},
\qquad \qquad
 d_{\max}&=&\max\limits_{1\leq i\leq n}\D_{i,i}.\end{array}
\end{equation}
We can now estimate
\begin{equation}
\begin{array}{lcl}
2 d_{\min}\nrm{\phi_o}_{L^2(\R;\R^n)}&\leq &
2 \nrm{\D\phi_o}_{L^2(\R;\R^n)}\\[0.2cm]
&\leq &
  \nrm{\D S_1\phi_e}_{L^2(\R;\R^n)}
   +\epsilon\abs{\tilde{c}_{\epsilon}}\nrm{\epsilon\phi_o'}_{L^2(\R;\R^n)}
 \\[0.2cm]
&&\qquad + \epsilon\nrm{D_1f_o(\overline{U}_{o;\epsilon})}_\infty\nrm{\epsilon \phi_o}_{L^2(\R;\R^n)}
\\[0.2cm]
&&\qquad
 +\epsilon\nrm{D_2f_o(\overline{U}_{o;\epsilon})}_\infty\nrm{\epsilon\psi_o}_{L^2(\R;\R^k)}
\\[0.2cm]
&&\qquad
  +\epsilon\delta\nrm{\epsilon \phi_o}_{L^2(\R;\R^n)}
   + \epsilon\nrm{\epsilon\theta_o}_{L^2(\R;\R^n)}
\\[0.2cm]
&\leq &
  \Big[ 2 d_{\max} + \epsilon(\tilde{K}_{\mathrm{fam}} + 2 \tilde{K}_{F} + \delta_0) \Big]
    \norm{\M_{\epsilon}^{1,2} \Phi}_{\mathbf{H}^1}
  + \epsilon \nrm{\M_{\epsilon}^{1,2} \Theta } .
\\[0.2cm]
\end{array}
\end{equation}
The desired bound hence follows directly
from Proposition \ref{theorem4equivalent}.
\qed\\

The scaling on the second components of $\Phi$ and $\Phi'$ can be removed in a similar fashion.
However, in this case one also needs to remove the corresponding scaling on  $\Theta$.
\begin{corollary}\label{theorem4equivalentspecific2}
Consider the setting of Proposition \ref{theorem4equivalent}.
%Assume that (\asref{familyassumption}{\text{h}}), (\asref{aannamesconstanten}{\text{H}}),(\asref{aannamesconstanten3}{\text{H}}),
%(\asref{aannamespuls:even}{\text{H}}), (\asref{aannamespuls:odd}{\text{H}}) and (\asref{extraaannamespuls}{\text{H}}) are satisfied.
Then for each $0<\delta<\delta_0$ and $0<\epsilon<\epsilon_0(\delta)$,
the operator $\tilde{\mathcal{L}}_{\epsilon,\delta}$ satisfies the bound
\begin{equation}\label{normboundmetc0spc3}
\begin{array}{lcl}
\nrm{\M_{\epsilon}^{1}\Phi'}_{\mathbf{L}^2}
 + \nrm{\Phi}_{\mathbf{L}^2}
 &\leq &
C_0\Big[\nrm{\M_{\epsilon}^{1}\Theta}_{\mathbf{L}^2}+\frac{1}{\delta}\big|\ip{ \Theta,(0,\overline{\Phi}_{e;0}^{\mathrm{adj}})}_{\mathbf{L}^2}\big|\Big]
\end{array}
\end{equation}
for any $\Phi\in\mathbf{H}^1$ and $\Theta = \tilde{\mathcal{L}}_{\epsilon, \delta} \Phi$,
possibly after increasing $C_0 > 0$.
\end{corollary}
\textit{Proof.}
Writing $\Phi_o = (\phi_o, \psi_o)$ and
$\Theta_o = (\theta_o, \chi_o)$,
we can inspect the definitions
(\ref{deftildeL}) and
(\ref{defL0}) to obtain
\begin{equation}
\begin{array}{lcl}
(\overline{L}_{o}+\delta)\psi_o&=&
D_1g_o(\tilde{U}_{o;\epsilon})\phi_o+\chi_o .
\end{array}
\end{equation}
Using Lemma \ref{eigenschappen2ecompL0}
we hence obtain the estimate
\begin{equation}
\begin{array}{lcl}
\nrm{\psi_o}_{H^1(\R;\R^k)}&\leq &
%K\nrm{D_1g_o(\overline{U}_{o;0})\phi_o+\chi_o}_{L^2(k)}\\
%&\leq &
C_1' \Big[ \nrm{D_1g_o(\tilde{U}_{o;\epsilon})}_\infty\nrm{\phi_o}_{L^2(\R;\R^n)}+\nrm{\chi_o}_{L^2(\R;\R^k)}\Big]
\end{array}
\end{equation}
for some $C_1' > 0$.
Combining this with (\ref{normboundmetc0fam1})
yields the desired bound
(\ref{normboundmetc0spc3}).
\qed\\

Our final result here provides information on the second derivatives of $\Phi$,
in the setting where $\Theta$ is differentiable.
In particular, we introduce the spaces
\begin{equation}\begin{array}{lclclcl}
\mathbf{H}^2_o &=& \mathbf{H}^2_e &=&
H^2(\Real; \Real^n) \times H^2(\Real;\Real^k),
\qquad
\qquad
\mathbf{H}^2 &=& \mathbf{H}^2_o \times \mathbf{H}^2_e .\end{array}
\end{equation}
We remark here that we have chosen to keep the scalings on the second
components of $\Phi''$ and $\Theta'$ because this will
be convenient in {\S}\ref{sectionexistence}.
Note also that the stated bound on $\norm{\Phi}_{\mathbf{H^1}}$ can actually be obtained by treating $\tilde{\mathcal{L}}_{\epsilon,\delta}$ as a regular
perturbation of $\L_{\diamond,\delta}$. The point here is that we gain an order of regularity, which is crucial for the nonlinear estimates.

\begin{corollary}\label{theorem4H2versie2}
Consider the setting of Proposition \ref{theorem4equivalent}
and assume furthermore
that $\nrm{ \tilde{U}_{\epsilon}'}_\infty$ is uniformly bounded for $\epsilon > 0$.
Then for each $0<\delta<\delta_0$ and any $0<\epsilon<\epsilon_0(\delta)$,
the operator $\tilde{\mathcal{L}}_{\epsilon,\delta}: \mathbf{H}^2 \to \mathbf{H}^1$
is invertible and satisfies the bound
\begin{equation}\label{normboundmetc0inH22}
\begin{array}{lcl}
\nrm{\M_{\epsilon}^{1,2}\Phi''}_{\mathbf{L}^2}
+\nrm{\Phi}_{\mathbf{H}^1}
&\leq &
C_0 \Big[\nrm{\M_{\epsilon}^{1}\Theta}_{\mathbf{L}^2}+\nrm{\M_{\epsilon}^{1,2}\Theta'}_{\mathbf{L}^2}
  +\frac{1}{\delta}\big|\ip{\Theta,(0,\overline{\Phi}_{e;0}^{\mathrm{adj}})}_{\mathbf{L}^2}\big|\Big]
\end{array}
\end{equation}
for any $\Phi\in\mathbf{H}^2$ and $\Theta=\tilde{\mathcal{L}}_{\epsilon,\delta}\Phi$,
possibly after increasing $C_0 > 0$.
\end{corollary}
\textit{Proof.}
Pick two constants  $0<\delta<\delta_0$ and $0<\epsilon<\epsilon_0(\delta)$
together with a function
$\Phi=(\Phi_o,\Phi_e)\in\mathbf{H}^1$ and write $\Theta = \tilde{\mathcal{L}}_{\epsilon,\delta}\Phi \in \mathbf{L}^2$.
If in fact $\Phi \in \mathbf{H}^2$,
then a direct differentiation shows that
%By differentiating the definition of $\Theta$
%we obtain
\begin{equation}\label{gedifferentieerd}
\begin{array}{lcl}
\Theta'&=&\tilde{\mathcal{L}}_{\epsilon,\delta}\Phi'
- D^2 F\big(\tilde{U}_\epsilon \big)
  \big[ \tilde{U}_\epsilon' , \Phi \big] ,
%
%
%+\left(\begin{array}{l}D^2F_o(\overline{U}_{o;0})[\overline{U}_{o;0}',\Phi_o]\\
%
%D^2F_e(\overline{U}_{e;0})[\overline{U}_{e;0}',\Phi_e]\end{array}\right)\\[0.4cm]
%&:=&\Theta'+ \mathbf{D}_2[\Phi].
\end{array}
\end{equation}
which due to the boundedness of $\Phi$ implies that $\Theta \in \mathbf{H}^1$.
In particular, $\tilde{\mathcal{L}}_{\epsilon,\delta}$ maps $\mathbf{H}^2$
into $\mathbf{H}^1$. Reversely, suppose that we know that
$\Theta\in\mathbf{H}^1$. Rewriting
(\ref{gedifferentieerd}) yields
\begin{equation}
\begin{array}{lcl}
\tilde{c}_\epsilon\Phi''&=& \Theta' - \delta \Phi'
+ \M_{1/\epsilon^{2}}^{1} J_{\mathrm{mix}} \Phi'
+ DF(\tilde{U}_{\epsilon} ) \Phi'
+ D^2 F(\tilde{U}_{\epsilon} )
  \big[ \tilde{U}_{\epsilon}' , \Phi \big].
\end{array}
\end{equation}
Since $\Phi$ is bounded, this allows us to conclude that $\Phi\in\mathbf{H}^2$.
On account of
Proposition \ref{theorem4equivalent} we hence see that $\tilde{\mathcal{L}}_{\epsilon,\delta}$ is invertible
as a map from $\mathbf{H}^2$ to $\mathbf{H}^1$.\\

%Again, fix $\Phi=(\Phi_o,\Phi_e)\in\mathbf{H}^2$ and set $\Theta=\mathcal{L}_{\epsilon,\delta}\Phi$.
% and
%note that $\Theta = \mathcal{L}_{\epsilon, \delta } \Phi$
%implies that
Fixing $\delta_{\mathrm{ref}} = \frac{1}{2} \delta_0$,
a short computation shows that
\begin{equation}\begin{array}{lcl}
\tilde{\mathcal{L}}_{\epsilon,\delta_{\mathrm{ref}}} \Phi'
&=& \Theta' + D^2F[\tilde{U}_{\epsilon}' ,\Phi]
+ (\delta_{\mathrm{ref}} - \delta) \Phi' .
\end{array}
\end{equation}
By (\ref{normboundmetc0spc3}) we obtain the bound
\begin{equation}\label{H2bound1}
\begin{array}{lcl}
\nrm{\M_{\epsilon}^{1}\Phi'}_{\mathbf{L}^2}
+ \nrm{\Phi}_{\mathbf{L}^2}&\leq &
C_0\Big[\nrm{\M_{\epsilon}^{1}\Theta}_{\mathbf{L}^2}
  +\frac{1}{\delta}\big|\ip{\Theta,(0,\overline{\Phi}_{e;0}^{\mathrm{adj}})}_{\mathbf{L}^2}\big|\Big].\\[0.2cm]
\end{array}
\end{equation}
On the other hand,
\sref{normboundmetc0fam1}
yields the estimate
\begin{equation}\label{H2bound2:a}
\begin{array}{lcl}
\nrm{\M_{\epsilon}^{1,2}\Phi''}_{\mathbf{L}^2}
+ \nrm{\M_{\epsilon}^2 \Phi'}_{\mathbf{L}^2}&\leq &
C_0\Big[\nrm{\M_{\epsilon}^{1,2}\Theta'}_{\mathbf{L}^2}
 +\nrm{\M_{\epsilon}^{1,2}
   D^2F[\tilde{U}_{\epsilon}',\Phi]}_{\mathbf{L}^2}
 +\nrm{\M_{\epsilon}^{1,2}(\delta_{\mathrm{ref}}-\delta)\Phi'}_{\mathbf{L}^2} \Big]\\[0.2cm]
&&\qquad +\frac{C_0}{\delta_{\mathrm{ref}}}\big|\ip{
   \Theta'-D^2 F(\tilde{U}_{\epsilon})
       [\tilde{U}_{\epsilon}',\Phi]
         -(\delta_{\mathrm{ref}}-\delta)\Phi',(0,\overline{\Phi}_{e;0}^{\mathrm{adj}})}_{\mathbf{L}^2}\big| .
\\[0.2cm]
\end{array}
\end{equation}
Since $\tilde{U}_{\epsilon}$ and $\tilde{U}_{\epsilon}'$
 are uniformly bounded by assumption, we readily see that
\begin{equation}\label{H2bound3}
\begin{array}{lclcl}
\nrm{\mathcal{M}^{1,2}_{\epsilon} D^2 F(\tilde{U}_{\epsilon})
  [\tilde{U}_{\epsilon}',\Phi]}_{\mathbf{L^2}}
&\le & \nrm{ D^2 F(\tilde{U}_{\epsilon})
  [\tilde{U}_{\epsilon}',\Phi]}_{\mathbf{L^2}}
%\\[0.2cm]
 &\le & C_1' \nrm{\Phi}_{\mathbf{L}^2}

\end{array}
\end{equation}
for some $C_1'>0$.
In particular, we find
\begin{equation}\label{H2bound2:b}
\begin{array}{lcl}
\nrm{\M_{\epsilon}^{1,2}\Phi''}_{\mathbf{L}^2}
+ \nrm{\M_{\epsilon}^2 \Phi'}_{\mathbf{L}^2}
& \le &
C_2' \Big[\nrm{\M_{\epsilon}^{1,2}\Theta'}_{\mathbf{L}^2}
+ \norm{\Phi}_{\mathbf{L}^2}
+ \nrm{\M_{\epsilon}^{1,2}\Phi'}_{\mathbf{L}^2}
\\[0.2cm]
& & \qquad \qquad
 + \nrm{\Theta_e'}_{\mathbf{L}^2_e}
 + \nrm{\Phi_e'}_{\mathbf{L}^2_e}
\Big]

\end{array}
\end{equation}
for some $C_2' > 0$.
 Exploiting the estimates
\begin{equation}
%\begin{array}{lclcl}
 \nrm{ \Phi_e'}_{\mathbf{L}^2_e} %&\le &
 \le \nrm{ \mathcal{M}_{\epsilon}^{1,2} \Phi'}_{\mathbf{L}^2}
 \le
 \nrm{ \mathcal{M}_{\epsilon}^{1} \Phi'}_{\mathbf{L}^2},
\qquad \qquad
\nrm{ \Theta_e'}_{\mathbf{L}^2_e} %&\le &
\le
 \nrm{ \mathcal{M}_{\epsilon}^{1,2} \Theta'}_{\mathbf{L}^2},
%\end{array}
\end{equation}
together with
\begin{equation}\begin{array}{lcl}
\norm{\Phi'}_{\mathbf{L}^2} &\le &\norm{\mathcal{M}_{\epsilon}^1 \Phi'}_{\mathbf{L}^2}
+ \norm{\mathcal{M}_{\epsilon}^2 \Phi'}_{\mathbf{L}^2} ,\end{array}
\end{equation}
the bounds (\ref{H2bound1}) and (\ref{H2bound2:b}) %and (\ref{H2bound3})
can be combined to arrive at the desired inequality (\ref{normboundmetc0inH22}).
\qed\\

\subsection{Strategy}

In this subsection we outline our broad strategy to establish Propositions \ref{theorem4equivalent} and \ref{theorem4equivalentcompact}.
As a first step, we compute the Fredholm index of the operators $\tilde{\mathcal{L}}_{\epsilon,\lambda}$
for $\lambda$ in a right half-plane that includes the imaginary axis.

\begin{lemma}\label{Lepsfredholm} Assume that (\asref{familyassumption}{\text{h}}),
(\asref{aannamesconstanten}{\text{H}}), (\asref{aannamesconstanten3}{\text{H}}),
(\asref{aannamespuls:even}{\text{H}}) and (\asref{aannamespuls:odd}{\text{H}}) % and (\asref{extraaannamespuls}{\text{H}})
are satisfied.
Then there exists a constant $\lambda_0 > 0$ %for which %in such a way that
so that
the operators $\tilde{\mathcal{L}}_{\epsilon,\lambda}$
are Fredholm with index zero whenever $\Re \lambda \ge - \lambda_0$ and $\epsilon > 0$.
\end{lemma}
\textit{Proof.}
%Fix a small constant $\beta_*>0$ together with a $\lambda\in\C$
%that has $\Re\lambda\geq-\beta_*$.
Upon writing
\begin{equation}
\begin{array}{lcl}
F^{(1)}_{o;\rho}=\rho DF_o(U_o^{-})+(1-\rho)DF_o(U_o^{+}),\\[0.2cm]
F^{(1)}_{e;\rho}=\rho DF_e(U_e^{-})+(1-\rho)DF_e(U_e^{+})
\end{array}
\end{equation}
for any $0 \le \rho \le 1$,
we introduce the constant coefficient operator
$L_{\rho;\epsilon,\lambda}: \mathbf{H}^1_{\C} \to \mathbf{L}^2_{\C}$ that acts as
\begin{equation}
\begin{array}{lcl}
L_{\rho;\epsilon,\lambda}
&=&\left(\begin{array}{ll}\tilde{c}_\epsilon\frac{d}{d\xi}+\frac{2}{\epsilon^2}J_{\D}- F^{(1)}_{o;\rho}+\lambda &
-\frac{1}{\epsilon^2}J_{\D} S_1\\
-J_{\D} S_1 &\tilde{c}_\epsilon\frac{d}{d\xi}+2J_{\D} -F^{(1)}_{e;\rho} +\lambda
\end{array}\right)
\end{array}
\end{equation}
and has the associated characteristic function
\begin{equation}
\begin{array}{lcl}
\Delta_{L_{\rho;\epsilon,\lambda} }(z)
%&:=&\Big[
%\mathcal{L}_{\epsilon,\lambda,\rho}^+e^{z\xi}\Big](0)\\[0.4cm]
&=&
\left(\begin{array}{ll}\tilde{c}_\epsilon z+\frac{2}{\epsilon^2}J_{\D}-F^{(1)}_{o;\rho}+\lambda &
  -\frac{1}{\epsilon^2}J_{\D}\Big[e^z+e^{-z}\Big]\\
-J_{\D}\Big[e^z+e^{-z}\Big] &\tilde{c}_\epsilon z+2J_{\D} -F^{(1)}_{e;\rho}+\lambda
\end{array}\right).
\end{array}
\end{equation}

Upon writing
\begin{equation}
\begin{array}{lcl}
F^{(1)}_{\rho} &=&
\left(\begin{array}{ll} F^{(1)}_{o;\rho}  &   0 \\ 0 & F^{(1)}_{e;\rho}
\end{array}\right)\end{array}
\end{equation}
together with
\begin{equation}\begin{array}{lcl}
A(y) &=& \left(\begin{array}{ll} J_\D  &   -J_\D\cos(y) \\ -J_\D\cos(y) & J_\D
\end{array}\right),\end{array}
\end{equation}
we see that
\begin{equation}\begin{array}{lcl}
\mathcal{M}_{\epsilon^2}^{1,2} \Delta_{L_{\rho;\epsilon,\lambda}}(iy)
&=& (\tilde{c}_{\epsilon} i y  + \lambda) \mathcal{M}_{\epsilon^2}^{1,2}
+ 2 A(y)
- \mathcal{M}_{\epsilon^2}^{1,2} F^{(1)}_{\rho}.\end{array}
\end{equation}
For any $y \in \Real$ and $V \in \mathbb{C}^{2(n + k)}$
we have
\begin{equation}\begin{array}{lcl}
\Re V^{\dagger}
\tilde{c}_{\epsilon} i y  \mathcal{M}_{\epsilon^2}^{1,2} V
&=& 0,
\end{array}
\end{equation}
together with
\begin{equation}\begin{array}{lcl}
\Re V^{\dagger}
 A(y) V
&\ge& 0 .
\end{array}
\end{equation}
In particular, we see that
\begin{equation}
\begin{array}{lcl}
\Re V^\dagger \mathcal{M}_{\epsilon^2}^{1,2}
 \Delta_{L_{\rho;\epsilon,\lambda} }(iy) V
& \ge &
  -\epsilon^2 \Re \big[ V^\dagger_o (F_{o;\rho}^{(1)}  - \lambda) V_o \big]
  - \Re \big[ V^\dagger_e (F_{e;\rho}^{(1)}  - \lambda) V_e \big].

\end{array}
\end{equation}

Let us pick an arbitrary $\lambda_0 > 0$ and suppose that $\Delta_{L_{\rho;\epsilon,\lambda} }(iy) V = 0$ holds
for some $V \in \mathbb{C}^{2(n+k)} \setminus \{ 0 \}$ and $\Re \lambda \ge - \lambda_0$.
We claim that
there exist constants $\vartheta_1 > 0$ and $\vartheta_2 > 0$, that do not depend on $\lambda_0$,
so that
\begin{equation}
\label{eq:trnsf:fredholm:claim}\begin{array}{lcl}
- \Re V^\dagger_{\#} (F_{\#;\rho}^{(1)}  - \lambda) V_{\#} &\ge & (\vartheta_2 - \vartheta_1 \lambda_0 ) \abs{V_{\#}}^2\end{array}
\end{equation}
for $\# \in \{o,e\}$. Assuming that this is indeed the case, we pick $\lambda_0 = \frac{ \vartheta_2}{2 \vartheta_1}$
and obtain the contradiction
\begin{equation}
\begin{array}{lcl}
0 & = &
\Re V^\dagger \mathcal{M}_{\epsilon^2}^{1,2}
 \Delta_{L_{\rho;\epsilon,\lambda} }(iy) V
\\[0.2cm]
& \ge &
  \frac{1}{2} \vartheta_2 \big[ \epsilon^2 \abs{V_o}^2 + \abs{V_e}^2 \big]
\\[0.2cm]
& > & 0.
\end{array}
\end{equation}
The desired Fredholm properties then follow directly from \cite[Thm. C]{MPA}.\\

In order to establish the claim \sref{eq:trnsf:fredholm:claim}, we first assume that $F_\#$
satisfies (h$\alpha$).
The negative-definiteness of
$F_{\#;\rho}^{(1)}$ then directly yields the bound
\begin{equation}
\begin{array}{lcl}
\Re V_\#^\dagger (F_{\#;\rho}^{(1)}  - \lambda) V_\#
&\le & (\lambda_0 - \vartheta_2)   \abs{V_\#}^2
 \end{array}
\end{equation}
for some $\vartheta_2 > 0$.\\

On the other hand, if $F_\#$
satisfies (h$\beta$),
then we can use the identity
\begin{equation}
\begin{array}{lcl}
(\tilde{c}_{\epsilon} i y + \lambda) w_\#
- [F^{(1)}_{\#;\rho}]_{2,2} w_\#
&=&  [F^{(1)}_{\#;\rho}]_{2,1} v_\#
\end{array}
\end{equation}
to compute
\begin{equation}
\begin{array}{lcl}
\Re V_{\#}^\dagger
\left(\begin{array}{ll} 0 & [F^{(1)}_{\#;\rho}]_{1,2} \\
    {[}F^{(1)}_{\#;\rho}]_{2,1} & 0
\end{array}\right) V_{\#}
& = &
\Re V_\#^\dagger
\left(\begin{array}{ll} 0 & -\Gamma [F^{(1)}_{\#;\rho}]_{2,1}^\dagger \\
    {[}F^{(1)}_{\#;\rho}]_{2,1}& 0
\end{array}\right) V_\#
\\[0.4cm]
& = &
\Re \Big[ -\Gamma v_\#^\dagger [F^{(1)}_{\#;\rho}]_{2,1}^\dagger w_\#
+  w_\#^\dagger [F^{(1)}_{\#;\rho}]_{2,1} v_\# \Big]
\\[0.2cm]
& = & (1 - \Gamma) \Re w_\#^\dagger  [F^{(1)}_{\#;\rho}]_{2,1} v_\#
\\[0.2cm]
& = &
(1 - \Gamma) \Re w_\#^\dagger
\big[ \tilde{c}_{\epsilon} i y + \lambda \big] w_\#
- (1 - \Gamma) \Re w_\#^\dagger [F_{\#;\rho}^{(1)}]_{2,2} w_\#
\\[0.2cm]
& = &
(1 - \Gamma) \Re \lambda \abs{w_\#}^2
- (1 - \Gamma) \Re w_\#^\dagger [F_{\#;\rho}^{(1)}]_{2,2} w_\# .
\\[0.2cm]
%& \le &
% - (\Gamma - 1) [\Re \lambda + \vartheta_2] \abs{w_\#}^2 .
\end{array}
\end{equation}
In particular, Lemma \ref{limitnegdef}
allows us to obtain the estimate
\begin{equation}\begin{array}{lcl}
\Re V^\dagger_{\#} (F^{(1)}_{\#;\rho} -\lambda ) V_{\#}
&= &
- \Gamma \Re \lambda \abs{w_\#}^2
+ \Gamma \Re w_\#^\dagger [F_{\#;\rho}^{(1)}]_{2,2} w_\#
\\[0.2cm]
& & \qquad
- \Re \lambda \abs{v_{\#}}^2
+ \Re v_{\#}^\dagger [F_{\#;\rho}^{(1)}]_{2,2} v_{\#}
\\[0.2cm]
& \le &
(\Gamma + 1) \lambda_0 \abs{V_{\#} }^2
- \vartheta_2 \abs{V_{\#} }^2
  \end{array}
\end{equation}
for some $\vartheta_2 > 0$, as desired.
\qed\\

For any $\epsilon > 0$ and $0 < \delta < \delta_{\diamond}$
we introduce the quantity
\begin{equation}\label{defLambdaepsdelta}
\begin{array}{lcl}
\Lambda(\epsilon,\delta)&=&\inf\limits_{\Phi\in\mathbf{H}^1,\nrm{\mathcal{M}_\epsilon^{1,2}\Phi}_{\mathbf{H}^1 }=1}
  \Big[\nrm{\mathcal{M}_\epsilon^{1,2}\tilde{\mathcal{L}}_{\epsilon,\delta}\Phi}_{\mathbf{L}^2}
    +\frac{1}{\delta}\big|\ip{\tilde{\mathcal{L}}_{\epsilon,\delta}\Phi,(0,\overline{\Phi}_{e;0}^{\mathrm{adj}})}_{\mathbf{L}^2}\big|\Big] ,
\end{array}
\end{equation}
which allows us to define
\begin{equation}
\begin{array}{lcl}
\Lambda(\delta)&=&\liminf\limits_{\epsilon\downarrow 0}\Lambda(\epsilon,\delta).
\end{array}
\end{equation}
Similarly, for any $\epsilon > 0$ and any subset $M\subset\C$
we write
\begin{equation}
\begin{array}{lcl}
\Lambda(\epsilon,M)&=&\inf\limits_{\Phi\in\mathbf{H}^1,\lambda\in M,\nrm{\mathcal{M}_\epsilon^{1,2}\Phi}_{\mathbf{H}^1}=1
 }\nrm{\mathcal{M}_\epsilon^{1,2}\tilde{\mathcal{L}}_{\epsilon,\lambda}\Phi}_{\mathbf{L}^2} ,
\end{array}
\end{equation}
together with
\begin{equation}\label{defLambdaM}
\begin{array}{lcl}
\Lambda(M)&=&\liminf\limits_{\epsilon\downarrow 0}\Lambda(\epsilon,M).
\end{array}
\end{equation}

The following proposition forms the key ingredient for proving Proposition \ref{theorem4equivalent} and \ref{theorem4equivalentcompact}.
It is the analogue of \cite[Lem. 6]{BatesInfRange}.

\begin{proposition}\label{lemma6equivalent} Assume that (\asref{familyassumption}{\text{h}}), (\asref{aannamesconstanten}{\text{H}}),
(\asref{aannamesconstanten3}{\text{H}}), (\asref{aannamespuls:even}{\text{H}}), (\asref{aannamespuls:odd}{\text{H}}) and
(\asref{extraaannamespuls}{\text{H}}) are satisfied. Then there exist constants $\delta_0>0$ and $C_0>0$ so that
% depending only on the choice of $\tilde{U}_\epsilon$ and $c_\epsilon$, in such a way that
\begin{equation}
\begin{array}{lcl}
\Lambda(\delta)&\geq &\frac{2}{C_0}
\end{array}
\end{equation}
holds for all $0<\delta<\delta_0$.

Assume furthermore that (\asref{extraextraaannamespuls}{\text{H}}) holds
and pick a sufficiently small $\lambda_0 > 0$.
Then for any subset $M\subset\C$  that satisfies %(\asrefrev{Massumption}{M_{\lambda_{0}}}),
(h$M_{\lambda_{0}}$), there exists a
constant $C_M$ %, depending only on the choice of $M$, $\tilde{U}_\epsilon$ and $c_\epsilon$, in such a way that
so that
\begin{equation}
\begin{array}{lcl}
\Lambda(M)&\geq &\frac{2}{C_M}.
\end{array}
\end{equation} \end{proposition}

%Using this proposition our main results from this section can be readily established.\\

\textit{Proof of Proposition \ref{theorem4equivalent}.} Fix $0 < \delta < \delta_0$.
Proposition
\ref{lemma6equivalent} implies that we can pick $\epsilon_0(\delta)>0$ in such a way that
$\Lambda(\epsilon,\delta)\geq\frac{1}{C_0}$ for each $0<\epsilon<\epsilon_0(\delta)$.
This means that $\tilde{\mathcal{L}}_{\epsilon,\delta}$ is injective for each such
$\epsilon$ and that the bound
(\ref{normboundmetc0fam}) holds for any $\Phi\in\mathbf{H}^1$.
Since $\tilde{\mathcal{L}}_{\epsilon,\delta}$ is also a Fredholm operator with index zero by
Lemma \ref{Lepsfredholm},
it must be invertible.
%
%By the definition of
%$\Lambda(\epsilon,\delta)$ this means that  On account of Lemma \ref{Lepsfredholm} we have that $\tilde{\mathcal{L}}_{\epsilon,\delta}$
%has Fredholm index $0$ and since it is injective, it follows that it is invertible.
\qed\\

\textit{Proof of Proposition \ref{theorem4equivalentcompact}.}
The result can be established by repeating the arguments used in the proof of Proposition \ref{theorem4equivalent}.
\qed\\

%The proof of Proposition \ref{theorem4equivalentcompact} is identical and will be omitted.

%Existentie van de sequence

\subsection{Proof of Proposition \ref{lemma6equivalent}}

We now set out to prove Proposition \ref{lemma6equivalent}. In Lemma's \ref{lemma6bewijs1} and \ref{lemma6bewijs2} we construct weakly converging
sequences that realize the infima in (\ref{defLambdaepsdelta})-(\ref{defLambdaM}). In Lemma's \ref{lemma6bewijs2.5}-\ref{lemma6bewijs4} we exploit the structure of our
operator (\ref{deftildeL:short}) to recover lower bounds on the norms of the derivatives of these sequences that are typically lost when taking weak limits.
First recall the constant $\delta_{\diamond}$ from Proposition \ref{eigenschappenL0:a}.

\begin{lemma}\label{lemma6bewijs1} %Assume that (\asref{familyassumption}{\text{h}}), (\asref{aannamesconstanten}{\text{H}}),
%(\asref{aannamesconstanten3}{\text{H}}), (\asref{aannamespuls:even}{\text{H}}), (\asref{aannamespuls:odd}{\text{H}}) and (\asref{extraaannamespuls}{\text{H}}) are satisfied.
Consider the setting of Proposition \ref{lemma6equivalent}
and pick $0 < \delta < \delta_{\diamond}$.
% pick a sufficiently small $0<\delta_0\leq \delta_{\diamond}$ and fix $0<\delta<\delta_0$.
Then there exists a sequence
\begin{equation}\begin{array}{lcl}
\{(\epsilon_j,\Phi_j, \Theta_j)\}_{j\geq 1} &\subset &(0,1)\times \mathbf{H}^1 \times \mathbf{L^2}\end{array}
\end{equation}
together with a pair of functions
\begin{equation}
%\Phi \in \mathbf{H}^1 \cap \mathbf{R}_0,
%\qquad
\Phi\in  \mathbf{H}^1,
\qquad
\Theta \in \mathbf{L}^2
\end{equation}
that satisfy the following properties.

\begin{enumerate}[label=(\roman*)]

\item We have $\lim\limits_{j\rightarrow\infty}\epsilon_j=0$ together with
\begin{equation}
\label{eq:trnsf:fred:lim}
\begin{array}{lcl}
\lim\limits_{j\rightarrow\infty}\Big[\nrm{\mathcal{M}_{\epsilon_j}^{1,2}\Theta_j}_{\mathbf{L}^2}
  +\frac{1}{\delta}\big|\ip{\Theta_j,(0,\overline{\Phi}_{e;0}^{\mathrm{adj}})}_{\mathbf{L}^2}\big|\Big]&=&\Lambda(\delta).
\end{array}
\end{equation}

\item For every $j \geq 1$ we have the identity
\begin{equation}\begin{array}{lcl}
\label{eq:transf:fred:id:for:theta:j}
\tilde{\mathcal{L}}_{\epsilon_j,\delta}\Phi_j &=& \Theta_j\end{array}
\end{equation}
together with the normalization
\begin{equation}\begin{array}{lcl}
\label{eq:trnsf:fred:norm:cond}
\nrm{\mathcal{M}_{\epsilon_j}^{1,2}\Phi_j}_{\mathbf{H}^1}&=&1.\end{array}
\end{equation}
\item Writing $\Phi = (\phi_o, \psi_o, \phi_e, \psi_e)$, we have $\phi_o = 0$.
\item The sequence $\mathcal{M}_{\epsilon_j}^{1,2} \Phi_j$ converges to $ \Phi$ strongly in $\mathbf{L}_{\mathrm{loc}}^2$ and weakly in $\mathbf{H}^1$. In addition,
the sequence $\M_{\epsilon_j}^{1,2}\Theta_j$ converges weakly to $\Theta$ in $\mathbf{L}^2$.

\end{enumerate}
\end{lemma}
\textit{Proof.}
Items (i) and (ii) follow directly
from the definition of $\Lambda(\delta)$.
The normalization (\ref{eq:trnsf:fred:norm:cond})
and the limit (\ref{eq:trnsf:fred:lim}) ensure that $\nrm{\M_{\epsilon_j}^{1,2}\Phi_j}_{\mathbf{H}^1}$
and $\nrm{\M_{\epsilon_j}^{1,2}\Theta_j}_{\mathbf{L}^2}$ are bounded, which allows us to
obtain the weak limits (iv) after passing to a subsequence.\\

 In order to obtain (iii),
we write $\Phi_j=(\phi_{o,j},\psi_{o,j},\phi_{e,j},\psi_{e,j})$ 
together with $\Theta_j= (\theta_{o,j},\chi_{o,j},\theta_{e,j},\chi_{e,j})$
and %we 
note that the first component of (\ref{eq:transf:fred:id:for:theta:j}) yields
\begin{equation}\label{eq:first:comp}
\begin{array}{lcl}
2\D\phi_{o,j} - \D S_1\phi_{e,j}&=&
-\epsilon_j^2 \tilde{c}_{\epsilon_j}\phi_{o,j}'
+\epsilon_j^2 D_1f_o(\tilde{U}_{o;\epsilon_j})\phi_{o,j}+\epsilon_j^2 D_2f_o(\tilde{U}_{o;\epsilon_j})\psi_{o,j}
-\delta\epsilon_j^2\phi_{o,j}+\epsilon_j^2\theta_{o,j}.
\end{array}
\end{equation}
The normalization condition (\ref{eq:trnsf:fred:norm:cond})
and the limit (\ref{eq:trnsf:fred:lim}) hence imply that
\begin{equation}\begin{array}{lcl}
\lim_{j \to \infty} \nrm{2 \mathcal{D} \phi_{o;j} -  \mathcal{D} S_1 \phi_{e,j}}_{ L^2(\R;\R^n)} &= &0.\end{array}
\end{equation}
In particular, we see that $\{\phi_{o;j}\}_{j\geq 1}$ is a bounded sequence. This yields the desired identity \\
$\phi_o=\lim\limits_{j\rightarrow\infty}\epsilon_j\phi_{o,j}=0$.

\qed\\

\begin{lemma}\label{lemma6bewijs2} %Assume that (\asref{familyassumption}{\text{h}}), (\asref{aannamesconstanten}{\text{H}}),
%(\asref{aannamesconstanten3}{\text{H}}), (\asref{aannamespuls:even}{\text{H}}), (\asref{aannamespuls:odd}{\text{H}}) and (\asref{extraaannamespuls}{\text{H}}) are satisfied.
Consider the setting of Proposition \ref{lemma6equivalent} and pick a sufficiently small $\lambda_0 > 0$.
Then for any $M\subset \C$ that satisfies %(\asrefrev{Massumption}{M_{\lambda_{0}}}),
(h$M_{\lambda_{0}}$),  there exists a sequence
\begin{equation}\begin{array}{lcl}
\{(\lambda_j,\epsilon_j,\Phi_j, \Theta_j)\}_{j\geq 1} &\subset  &M\times (0,1)\times \mathbf{H}^1 \times \mathbf{L^2}\end{array}
\end{equation}
together with a triplet
\begin{equation}
%\Phi \in \mathbf{H}^1 \cap \mathbf{R}_0,
%\qquad
\Phi\in  \mathbf{H}^1,
\qquad
\Theta \in \mathbf{L}^2,
\qquad
\lambda \in M
\end{equation}
%and a constant $\lambda\in M$
that satisfy the limits
\begin{equation}
\epsilon_j \to 0,
\qquad \lambda_j \to \lambda,
\qquad
\nrm{\mathcal{M}_{\epsilon_j}^{1,2}\Theta_j}_{\mathbf{L}^2}
  \to \Lambda(M)
\end{equation}
as $j \to \infty$,
together with the properties (ii) - (iv) from Lemma \ref{lemma6bewijs1}.
\end{lemma}
\textit{Proof.}
These properties can be obtained by following the proof of Lemma \ref{lemma6bewijs1} in an almost identical fashion.
\qed\\
%The proof of this lemma is almost identical to that of Lemma \ref{lemma6bewijs1} and as such will be omitted.\\

In the remainder of this section we will often treat
the settings of Lemma \ref{lemma6bewijs1} and Lemma \ref{lemma6bewijs2}
in a parallel fashion. In order to streamline our notation, we
use the value $\lambda_0$ stated in Lemma \ref{Lepsfredholm}
and
interpret $\{\lambda_j\}_{j\geq 1}$
as the constant sequence $\lambda_j=\delta$ when working in the context of Lemma \ref{lemma6bewijs1}.
In addition, we write $\lambda_{\max}=\delta_{\diamond}$ in the setting of Lemma \ref{lemma6bewijs1}
or $\lambda_{\max}=\max\{|\lambda|:\lambda\in M\}$ in the setting of Lemma \ref{lemma6bewijs2}.

\begin{lemma}\label{lemma6bewijs2.5}
%There exists a constant $K > 1$ so that for any ...,
Consider the setting of %Proposition \ref{lemma6equivalent} and
Lemma \ref{lemma6bewijs1} or Lemma \ref{lemma6bewijs2}.
Then the function $\Phi$ from Lemma \ref{lemma6bewijs1} satisfies
\begin{equation}
\begin{array}{lcl}
\nrm{ \Phi}_{\mathbf{H}^1}&\leq & C_{\diamond}\Lambda(\delta),
\end{array}
\end{equation}
while the function $\Phi$ from Lemma \ref{lemma6bewijs2} satisfies
\begin{equation}
\begin{array}{lcl}
\nrm{ \Phi}_{\mathbf{H}^1}&\leq & C_{\diamond;M}\Lambda(M) .
\end{array}
\end{equation}
\end{lemma}
\textit{Proof.} In order to  take the $\epsilon\downarrow 0$ limit in a controlled fashion,
we introduce the operator
\begin{equation}
    \begin{array}{lcl}
         \tilde{L}_{0;\lambda}&=&\lim\limits_{j\rightarrow\infty}\mathcal{M}_{\epsilon_j^2}^1
  \tilde{\mathcal{L}}_{\epsilon_j,\lambda_j} .
    \end{array}
\end{equation}
Upon introducing the top-left block
\begin{equation}\begin{array}{lcl}
[\tilde{L}_{0;\lambda}]_{1,1}
&=&
\left(
\begin{array}{cc}
2 \mathcal{D} & 0 \\
- D_1 g_o(\overline{U}_{o;0})
 & \overline{L}_{o} + \lambda \\
\end{array}
\right) ,
\end{array}
\end{equation}
we can explicitly write
\begin{equation}
\begin{array}{lcl}
\tilde{L}_{0;\lambda}&=&\left(\begin{array}{ll}
   [\tilde{L}_{0;\lambda}]_{1,1}
& -J\D S_1\\[0.2cm]
-J\D S_1 & c_0\frac{d}{d\xi}+2J\D-DF_e(\overline{U}_{e;0})+\lambda
\end{array}\right) .
\end{array}
\end{equation}
Note that $\tilde{L}_{0;\lambda}$ and its adjoint
$\tilde{L}_{0;\lambda}^{\mathrm{adj}}$
are both bounded operators from $\mathbf{H}^1$ to $\mathbf{L}^2$.\\

In addition, we introduce the commutators
\begin{equation}\begin{array}{lcl}
B_j &=&
    \tilde{\mathcal{L}}_{\epsilon_j,\lambda_j} M_{\epsilon_j}^{1,2}
      - M_{\epsilon_j}^{1,2}\tilde{\mathcal{L}}_{\epsilon_j,\lambda_j} .
      \end{array}
\end{equation}
A short computation shows that
\begin{equation}
    \begin{array}{lcl}
       B_j &=& \left(\begin{array}{cc}
                         [B_j]_{1,1} & (\frac{1}{\epsilon_j}-\frac{1}{\epsilon_j^2})J_{\D} S_1
                         \\
                         (1-\epsilon_j)J_{\D} S_1&0
                \end{array}
                \right),
    \end{array}
\end{equation}
in which the top-left block is given by
\begin{equation}\begin{array}{lcl}
[B_j]_{1,1} &=&(1 - \epsilon_j) \left(\begin{array}{cc} 0 &  D_2 f_o(\tilde{U}_{o;\epsilon_j} ) \\
         -D_1 g_o( \tilde{U}_{o;\epsilon_j} ) & 0 \\
      \end{array} \right) .
      \end{array}
\end{equation}

Pick any test-function $Z \in C^\infty(\Real; \Real^{2n+2k})$
and write
\begin{equation}\begin{array}{lcl}
\mathcal{I}_j
& = &
 \langle \mathcal{M}_{\epsilon_j^2}^1
  \tilde{\mathcal{L}}_{\epsilon_j,\lambda_j}
  \mathcal{M}_{\epsilon_j}^{1,2} \Phi_j , Z \rangle_{\mathbf{L}^2}  .\end{array}
\end{equation}
Using the strong convergence
\begin{equation}\begin{array}{lcl}
\tilde{\mathcal{L}}_{\epsilon_j,\lambda_j}^{\mathrm{adj}}
        \mathcal{M}_{\epsilon_j^2}^1   Z
&\to& \tilde{L}_{0;\lambda}^{\mathrm{adj}}  Z \in \mathbf{L^2},
\end{array}
\end{equation}
we obtain the limit
\begin{equation}
\begin{array}{lcl}
\mathcal{I}_j
& = &
 \langle
   \mathcal{M}_{\epsilon_j}^{1,2} \Phi_j ,
     \tilde{\mathcal{L}}_{\epsilon_j,\lambda_j}^{\mathrm{adj}}
        \mathcal{M}_{\epsilon_j^2}^1   Z \rangle_{\mathbf{L}^2}
\\[0.2cm]
& \to &
  \langle \Phi , \tilde{L}_{0;\lambda}^{\mathrm{adj}}  Z \rangle_{\mathbf{L}^2}
\\[0.2cm]
& =&
  \langle  \tilde{L}_{0;\lambda}\Phi ,  Z \rangle_{\mathbf{L}^2}
\end{array}
\end{equation}
as $j \to \infty$.\\

In particular, we see that
\begin{equation}
\begin{array}{lcl}
\mathcal{I}_j
& = &
  \langle  \mathcal{M}_{\epsilon_j^2}^1
   \mathcal{M}_{\epsilon_j}^{1,2}\tilde{\mathcal{L}}_{\epsilon_j,\lambda_j}
      \Phi_j , Z \rangle_{\mathbf{L}^2}
  + \langle  \mathcal{M}_{\epsilon_j^2}^1 B_j \Phi_j , Z \rangle_{\mathbf{L}^2}
\\[0.2cm]
& = &
\langle \mathcal{M}_{\epsilon_j^2}^1
  \mathcal{M}_{\epsilon_j}^{1,2}
  \Theta_j , Z \rangle_{\mathbf{L}^2}
  + \langle  \mathcal{M}_{\epsilon_j^2}^1 B_j \Phi_j , Z \rangle_{\mathbf{L}^2}
\\[0.2cm]
& \to &  \langle \mathcal{M}_{0}^1 \Theta, Z \rangle_{\mathbf{L}^2}
 + \big\langle \big(-\D S_1\phi_e,- D_1 g_o( \overline{U}_{o;0}) \phi_o ,\D S_1\phi_o ,0\big), Z\big\rangle_{\mathbf{L}^2}.
\end{array}
\end{equation}
It hence follows that
\begin{equation}
\begin{array}{lcl}
 \tilde{L}_{0;\delta}\Phi
 &=& \mathcal{M}_0^1 \Theta  + \big(-\D S_1\phi_e,- D_1 g_o( \overline{U}_{o;0}) \phi_o ,\D S_1\phi_o ,0\big).\end{array}
\end{equation}
Introducing the functions
\begin{equation}\begin{array}{lclcl}
\Phi_\diamond&=&(\psi_0,\phi_e,\psi_e) ,
\qquad
\qquad
 \Theta_{\diamond}&=&(\chi_o,\theta_e,\chi_e),\end{array}
\end{equation}
the identity $\phi_o=0$
implies that
%
% and introducing $\Phi_\diamond=(\psi_0,\phi_e,\psi_e)$ and
%$$, we obtain the identity
\begin{equation}
    \begin{array}{lcl}
    \mathcal{L}_{\diamond, \lambda}\Phi_{\diamond}&=& \Theta_{\diamond}.
    \end{array}
\end{equation}

In the setting of Lemma \ref{lemma6bewijs1},
we may hence use Proposition \ref{eigenschappenL0:a}
to compute
\begin{equation}
\begin{array}{lcl}
%\nrm{\M_{0}^1\Phi}_{\mathbf{H}^1}&=&
\nrm{\Phi_{\diamond}}_{\mathbf{H}_\diamond^1}
%\\[0.2cm]
&\leq &
C_\diamond \Big[\nrm{\Theta_\diamond}_{\mathbf{L}_\diamond^2}
 +\frac{1}{\delta}\big|\ip{
     \Theta_\diamond,
     (0,\overline{\Phi}_{e;0}^{\mathrm{adj}})
   }_{\mathbf{L}^2_{\diamond}}\big|\Big]\\[0.2cm]
&\le & C_\diamond \Big[\nrm{\Theta}_{\mathbf{L}^2}
 +\frac{1}{\delta}\big|\ip{\Theta,(0,\overline{\Phi}_{e;0}^{\mathrm{adj}})}_{\mathbf{L}^2}\big|\Big] .
\end{array}
\end{equation}
The lower semi-continuity of the $L^2$-norm
and the convergence in (iv) of Lemma \ref{lemma6bewijs1}
imply that
\begin{equation}
\begin{array}{lcl}
\nrm{\Theta}_{\mathbf{L}^2}+\frac{1}{\delta}\big|\ip{\Theta,(0,\overline{\Phi}_{e;0}^{\mathrm{adj}})}_{\mathbf{L}^2}\big|&\leq &\Lambda(\delta).
\end{array}
\end{equation}
In particular, we find
\begin{equation}
\begin{array}{lclcl}
\nrm{\Phi}_{\mathbf{H}^1}&=&
  \nrm{\Phi_{\diamond}}_{\mathbf{H}_\diamond^1}
%\\[0.2cm]
& \le & C_\diamond \Lambda(\delta),
\end{array}
\end{equation}
as desired. In the setting of Lemma \ref{lemma6bewijs2} the bound  follows in a similar fashion.
\qed\\

We note that
\begin{equation}\label{decomptheta_j}
\begin{array}{lcl}
\mathcal{M}_{\epsilon_j^2}^{1,2} \Theta_j
& = &
 \tilde{c}_{\epsilon_j}
   \mathcal{M}_{\epsilon_j^2}^{1,2} \Phi_j'+\mathcal{M}_{\epsilon_j^2}^{1,2} \big(

   - DF(\tilde{U}_{\epsilon_j} )
      +\lambda_j\big)\Phi_j
   - J_{\mathrm{mix}} \Phi_j ,
\end{array}
\end{equation}
in which $J_{\mathrm{mix}}$ is given by (\ref{Jmix}) and in which
\begin{equation}\begin{array}{lcl}
             DF(\tilde{U}_{\epsilon} )&=&\left( \begin{array}{ll}
     DF_o(\tilde{U}_{o;\epsilon} ) & 0\\ 0 &
       DF_e(\tilde{U}_{e;\epsilon} ) \end{array} \right).\end{array}
\end{equation}

\begin{lemma}\label{lemma6bewijsbound1} Assume that (\asref{aannamesconstanten}{\text{H}}) is satisfied. Then the bounds

\begin{equation}
    \begin{array}{lcl}
    \Re \, \ip{-J_{\mathrm{mix}}\Phi,\Phi'}_{\mathbf{L}^2}&=&0,\\[0.2cm]
    \Re \, \ip{-J_{\mathrm{mix}}\Phi,\Phi}_{\mathbf{L}^2}&\geq &0
    \end{array}
    \end{equation}
    hold for all $\Phi\in\mathbf{H}^1_{\C}$.
\end{lemma}
\textit{Proof.} Pick $\Phi\in\mathbf{H}^1_{\C}$ and write $\Phi=(\Phi_o,\Phi_e)$. We can compute
\begin{equation}
    \begin{array}{lcl}
    \Re \, \ip{-J_{\mathrm{mix}}\Phi,\Phi'}_{\mathbf{L}^2}
    &=&
      \Re \, \ip{2J_{\D}\Phi_{o},\Phi_{o}'}_{\mathbf{L}^2_o}
      -\Re \, \ip{J_{\D} S_1 \Phi_e,\Phi_{o}'}_{\mathbf{L}^2_o}
      \\[0.2cm] &&\qquad
       -\Re \, \ip{J_{\D} S_1 \Phi_{o},\Phi_e'}_{\mathbf{L}^2_e}
       +2  \Re \, \ip{J_{\D}\Phi_e,\Phi_e'}_{\mathbf{L}^2_e}\\[0.2cm]
&=&0,
    \end{array}
\end{equation}
since we have
$\Re \, \ip{J_{\D} S_1 \Phi_e,\Phi_{o}'}_{\mathbf{L}^2_o}=-\Re \,\ip{J_{\D} S_1 \Phi_{o},\Phi_e'}_{\mathbf{L}^2_e}$. Moreover, we can estimate
\begin{equation}
    \begin{array}{lcl}
\Re \, \ip{-J_{\mathrm{mix}}\Phi,\Phi}_{\mathbf{L}^2}&=&
  \Re \,\ip{2J_{\D}\Phi_{o},\Phi_{o}}_{\mathbf{L}^2_o}
    -\Re \, \ip{J_{\D} S_1 \Phi_e,\Phi_{o}}_{\mathbf{L}^2_o}
    \\[0.2cm] &&\qquad
      -\Re \, \ip{J_{\D} S_1 \Phi_{o},\Phi_e}_{\mathbf{L}^2_e}
      +2\Re\, \ip{J_{\D}\Phi_e,\Phi_e}_{\mathbf{L}^2_e}\\[0.2cm]
    &\geq &2\nrm{\sqrt{J_{\D}}\Phi_{o}}_{\mathbf{L}^2_o}^2+2\nrm{\sqrt{J_{\D}}\Phi_e}_{\mathbf{L}^2_e}^2 -4\nrm{\sqrt{J_{\D}}\Phi_{o}}_{\mathbf{L}^2_o}\nrm{\sqrt{J_{\D}}\Phi_e}_{\mathbf{L}^2_e}\\[0.2cm]
&\geq &2\nrm{\sqrt{J_{\D}}\Phi_{o}}_{\mathbf{L}^2_o}^2+2\nrm{\sqrt{J_{\D}}\Phi_e}_{\mathbf{L}^2_e}^2-4\Big(\frac{1}{2}\nrm{\sqrt{J_{\D}}\Phi_{o}}_{\mathbf{L}^2_o}^2+\frac{1}{2}\nrm{\sqrt{J_{\D}}\Phi_e}_{\mathbf{L}^2_e}^2\Big)\\[0.2cm]
&= &0.
    \end{array}
\end{equation}
\qed\\

\begin{lemma}\label{lemma6bewijsbound2:phi:dphi} Consider the setting of Lemma \ref{lemma6bewijs1} or Lemma \ref{lemma6bewijs2}.
%Recall the constant $\tilde{K}_{F}$ from (\ref{eq:cnst:kf:def}). Then there exist constants $a>0,m>0,g>0$ and $\beta\geq 0$ in such a way that the bounds
Then the bound
\begin{equation}
    \begin{array}{lcl}
   \big| \Re\big\langle \mathcal{M}_{\epsilon_j^2}^{1,2} \big(-DF(\tilde{U}_{\epsilon_j} )+\lambda_j\big)\Phi_j,\Phi_j'\big\rangle_{\mathbf{L}^2} \big|
    &\leq &
        2 (\tilde{K}_{F} + \lambda_{\max} )\nrm{\M_{\epsilon_j}^{1,2}\Phi}_{\mathbf{L}^2}\nrm{\M_{\epsilon_j}^{1,2}\Phi_j'}_{\mathbf{L}^2}
       \\[0.2cm]
    \end{array}
\end{equation}
holds for all $j\geq 1$.
\end{lemma}
\textit{Proof.}
We first note that
\begin{equation}
    \begin{array}{lcl}
    \Re\big\langle \mathcal{M}_{\epsilon_j^2}^{1,2}
       \big(-DF(\tilde{U}_{\epsilon_j} )+\lambda_j\big)\Phi_j,\Phi_j'\big\rangle_{\mathbf{L}^2}
    &=& \Re\ip{\epsilon_j ( - DF_o(\tilde{U}_{o;\epsilon_j}) + \lambda_j ) \Phi_{o,j},\epsilon_j\Phi_{o,j}'}_{\mathbf{L}^2_o}
\\[0.2cm]
&&\qquad
  + \Re\ip{ ( - DF_e(\tilde{U}_{e;\epsilon_j}) + \lambda_j) \Phi_{e,j},\Phi_{e,j}'}_{\mathbf{L}^2_e}.
    \end{array}
\end{equation}
Using Cauchy-Schwarz we compute
\begin{equation}
    \begin{array}{lcl}
    \big|\Re\big\langle \mathcal{M}_{\epsilon_j^2}^{1,2} \big(-DF(\tilde{U}_{\epsilon_j} )+\lambda_j\big)\Phi_j,\Phi_j'\big\rangle_{\mathbf{L}^2}\big|
&\leq & \big(\tilde{K}_F + \lambda_{\max} \big) \nrm{\epsilon_j\Phi_{o,j}}_{\mathbf{L}^2_o}\nrm{\epsilon_j\Phi_{o,j}'}_{\mathbf{L}^2_o}\\[0.2cm]
&&\qquad  + \big(\tilde{K}_F + \lambda_{\max} \big) \nrm{\Phi_{e,j}}_{\mathbf{L}^2_e}\nrm{\Phi_{e,j}'}_{\mathbf{L}^2_e}\\[0.2cm]
&\leq & 2 \big( \tilde{K}_{F} + \lambda_{\max}\big) \nrm{\M_{\epsilon_j}^{1,2}\Phi_j}_{\mathbf{L}^2}\nrm{\M_{\epsilon_j}^{1,2}\Phi_j'}_{\mathbf{L}^2},
    \end{array}
\end{equation}
as desired.
\qed\\

\begin{lemma}\label{lemma6bewijsbound2} Consider the setting of Lemma \ref{lemma6bewijs1} or Lemma \ref{lemma6bewijs2},
possibly decreasing the size of $\lambda_0 > 0$.
%Recall the constant $\tilde{K}_{F}$ from (\ref{eq:cnst:kf:def}).
Then there exist strictly positive constants $(a,m,g)$
together with a constant $\beta\geq 0$ so that the bound
\begin{equation}
    \begin{array}{lcl}
    \Re\big\langle\mathcal{M}_{\epsilon_j^2}^{1,2} \big(-DF(\tilde{U}_{\epsilon_j} )+\lambda_j\big)\Phi_j,\Phi_j\big\rangle_{\mathbf{L}^2}
    %  &\geq & a\nrm{\M_{\epsilon_j}^{1,2}\Phi_j}_{\mathbf{L}^2}^2-g\int\limits_{|x|\leq m|}|\M_{\epsilon_j}^{1,2}\Phi_j|^2\\[0.4cm]
%&&\qquad -\beta\nrm{(\epsilon_j\chi_{o,j},\chi_{e,j})}_{L^2(\R;\R^{2k})}^2\\[0.2cm]
&\geq & a\nrm{\M_{\epsilon_j}^{1,2}\Phi_j}_{\mathbf{L}^2}^2
  -g\int\limits_{|x|\leq m}|\M_{\epsilon_j}^{1,2}\Phi_j|^2
  -\beta\nrm{\M_{\epsilon_j}^{1,2}\Theta_j}_{\mathbf{L}^2}^2
    \end{array}
\end{equation}
holds for all $j\geq 1$.
\end{lemma}
\textit{Proof.}
We first note that
\begin{equation}\label{opsplitsen}
\begin{array}{lcl}
 \Re\big\langle\mathcal{M}_{\epsilon_j^2}^{1,2} \big(-DF(\tilde{U}_{\epsilon_j} )+\lambda_j\big)\Phi_j,\Phi_j\big\rangle_{\mathbf{L}^2}
 & = &
   \epsilon^2 \mathcal{N}_{o;j} + \mathcal{N}_{e;j} , %+ \Re \lambda_j \norm{\M^{1,2} \Phi_j}_{\mathbf{L}^2},
\end{array}
\end{equation}
in which we have defined
\begin{equation}\begin{array}{lcl}
\mathcal{N}_{\#, j} &=&
     \Re\big\langle  \big( -  DF_{\#}(\tilde{U}_{\#;\epsilon_j} )  + \lambda_j \big)  \Phi_{\#,j} ,\Phi_{\#,j}\big\rangle_{\mathbf{L}^2_{\#}}
    \end{array}
\end{equation}
for $\# \in \{o, e\}$.
\\

Let us first suppose that $F_{\#}$ satisfies (h$\beta$) and let $\Gamma_{\#}$ be the proportionality constant from that assumption.
We start by studying the cross-term
\begin{equation}
\begin{array}{lcl}
\mathcal{C}_{\#,j} &=& - \Re\big\langle  D_2 f_{\#}\big(\tilde{U}_{\#;\epsilon_j}  \big)\psi_{\#,j},\phi_{\#,j}\big\rangle_{L^2(\R;\R^n)}
\\[0.2cm]
& & \qquad
- \Re\big\langle  D_1 g_{\#}\big(\tilde{U}_{\#;\epsilon_j}  \big)\phi_{\#,j},\psi_{\#,j}\big\rangle_{L^2(\R;\R^k)}.
\end{array}
\end{equation}
Recalling that
\begin{equation}
\begin{array}{lcl}
\chi_{{\#},j}&=&\tilde{c}_{\epsilon_j}\psi_{{\#},j}'-Dg_{{\#};1}(\tilde{U}_{{\#};\epsilon_j})\phi_{{\#},j}
-Dg_{{\#};2}(\tilde{U}_{{\#};\epsilon_j})\psi_{{\#},j}+\lambda_j\psi_{{\#},j},
\end{array}
\end{equation}
we obtain the identity
\begin{equation}\label{bistableest2}
\begin{array}{lcl}
\mathcal{C}_{\#,j}&=&(\Gamma_{\#}-1)\Re\ip{D_1g_{\#}(\tilde{U}_{{\#};\epsilon_j})\phi_{{\#},j},\psi_{{\#},j}}_{L^2(\R;\R^k)}\\[0.2cm]
&=&(\Gamma_{\#}-1)\Re\ip{\tilde{c}_{\epsilon_j}\psi_{{\#},j}'-D_2g_{\#}(\tilde{U}_{{\#};\epsilon_j})\psi_{{\#},j}
  +\lambda_j\psi_{{\#},j}-\chi_{{\#},j},\psi_{{\#},j}}_{L^2(\R;\R^k)}
  \\[0.2cm]
& = &\tilde{c}_{\epsilon_j}(\Gamma_{\#}-1)\Re\ip{\psi_{{\#},j}',\psi_{\#,j}}_{L^2(\R;\R^k)}
\\[0.2cm]
& & \qquad +(\Gamma_{\#}-1) \Re \ip{-D_2g_{\#}(\tilde{U}_{{\#};\epsilon_j})\psi_{{\#},j}+\lambda_j\psi_{{\#},j}-\chi_{{\#},j},\psi_{{\#},j}}_{L^2(\R;\R^k)}\\[0.2cm]
&= &
   (1 - \Gamma_{\#})\Re \ip{D_2g_{\#}(\tilde{U}_{{\#};\epsilon_j})\psi_{{\#},j},\psi_{{\#},j}}_{L^2(\R;\R^k)}
\\[0.2cm]
& & \qquad
  + (\Gamma_{\#} -1 ) \Big[ \Re \lambda  \norm{\psi_{{\#},j}}^2_{\mathbf{L}_{\#}} - \langle \chi_{{\#}, j} , \psi_{{\#}, j} \rangle_{L^2(\R;\R^k)}
  \Big] .
\\[0.2cm]
\end{array}
\end{equation}
In particular, we see that
\begin{equation}\begin{array}{lcl}
\mathcal{N}_{\#, j} &=&
           \Gamma_{\#} \Re \lambda   % \norm{\psi_{{\#}, j}}^2_{L^2(\Real;\Real^k)}
                                    \langle  \psi_{{\#}, j},\psi_{{\#}, j} \rangle_{L^2(\R;\R^k)}
             - \Gamma_{\#} \Re \ip{D_2g_{\#}(\tilde{U}_{{\#};\epsilon_j})\psi_{{\#},j},\psi_{{\#},j}}_{L^2(\R;\R^k)}
\\[0.2cm]
& & \qquad
  + \Re \lambda  %\norm{ \phi_{{\#},j} }^2_{L^2(\Real; \Real^n) }
           \langle  \phi_{{\#},j} ,\phi_{{\#},j}  \rangle_{L^2(\R;\R^n)}
     - \Re \ip{D_1f_{\#}(\tilde{U}_{{\#};\epsilon_j})\phi_{{\#},j},\phi_{{\#},j}}_{L^2(\R;\R^n)}
\\[0.2cm]
& & \qquad
   -  (\Gamma_{\#} -1 ) \langle \chi_{{\#}, j} , \psi_{{\#}, j} \rangle_{L^2(\R;\R^k)} .
    \end{array}
\end{equation}

Recall that $\tilde{U}_\epsilon\rightarrow \overline{U}_{0}$ in $L^\infty$, $\tilde{U}_{o;\epsilon_j}(\xi)\rightarrow U_o^\pm$ and
$\tilde{U}_{e;\epsilon_j}(\xi)\rightarrow U_e^\pm$ for $\xi\rightarrow\pm\infty$. Using  Lemma \ref{limitnegdef} and
%by
decreasing $\lambda_0$ if necessary, we see that there exist
$a > (\Gamma_{\#} + 1) \lambda_0 > 0$ and $m \gg 1$ so that
\begin{equation}\label{def:a}
\begin{array}{lcl}
3 a \abs{ \Phi_{\#, j}(\xi) }^2
& \le &
- \Re\big\langle
  D_1 f_{\#}\big(\tilde{U}_{\#;\epsilon_j}(\xi)  \big)\phi_{\#,j}(\xi),\phi_{\#,j}(\xi)
  \big\rangle_{\R^n}
\\[0.2cm]
& & \qquad
- \Gamma_{\#} \Re\big\langle
       D_2 g_{\#}\big(\tilde{U}_{\#;\epsilon_j}(\xi)  \big)\psi_{\#,j}(\xi),\psi_{\#,j}(\xi)
    \big\rangle_{\R^k}
\end{array}
\end{equation}
for all $\abs{\xi} \ge m$.
We hence obtain
\begin{equation}
\begin{array}{lcl}
\mathcal{N}_{\#,j}
&\geq &
   2a\int_{|\xi|\geq m} |\Phi_{\#,j}(\xi)|^2 \, d \xi  - (\Gamma_{\#} + 1)\big(\tilde{K}_{F}+\lambda_{\mathrm{max}}\big)\int_{|\xi|\leq m}
    |\Phi_{\#,j}(\xi)|^2 \, d \xi
    \\[0.2cm]
    & & \qquad
     - (\Gamma_{\#} +1 ) \nrm{\chi_{{\#},j}}_{L^2(\R;\R^k)}\nrm{\psi_{{\#},j}}_{L^2(\R;\R^k)}
   \\[0.2cm]
     &\geq & 2a\nrm{\Phi_{\#,j}}_{\mathbf{L}^2_{\#}}^2
            - (\Gamma_{\#} + 1)\big(2a + \tilde{K}_{F}+\lambda_{\mathrm{max}}\big)\int_{|\xi|\leq m}
                       |\Phi_{\#,j}(\xi)|^2 \, d \xi
    \\[0.2cm]
        & & \qquad
     - (\Gamma_{\#} +1 ) \nrm{\chi_{{\#},j}}_{L^2(\R;\R^k)}\nrm{\psi_{{\#},j}}_{L^2(\R;\R^k)} .
\end{array}
\end{equation}
Using the standard identity $xy\leq \frac{1}{4z}x^2+zy^2$ for $x,y\in\R$ and $z > 0$,
we now find
\begin{equation}\label{afschattingmetcrossterms}
\begin{array}{lcl}
\mathcal{N}_{\#,j}
&\geq &
      a\nrm{\Phi_{\#,j}}_{\mathbf{L}^2_{\#}}^2
            - (\Gamma_{\#} + 1)\big(2a + \tilde{K}_{F}+\lambda_{\mathrm{max}}\big)\int_{|\xi|\leq m}
                       |\Phi_{\#,j}(\xi)|^2 \, d \xi
    \\[0.2cm]
        & & \qquad
     -\frac{1}{4a} (\Gamma_{\#} +1 )^2 \nrm{\chi_{{\#},j}}^2_{L^2(\R;\R^k)},
\end{array}
\end{equation}
which has the desired form.

In the case where $F_{\#}$ satisfies (h$\alpha$), a similar bound can be obtained
in an analogous, but far easier fashion.
\qed\\

\begin{lemma}\label{lemma6bewijs3} Consider the setting of %Proposition \ref{lemma6equivalent} and
Lemma \ref{lemma6bewijs1} or Lemma \ref{lemma6bewijs2}.
%Recall the constant $\tilde{K}_{\mathrm{fam}}$ from (\ref{eq:trn:unif:bnd:in:h:fam}).
Then there exists a constant $\kappa>0$ so that the bound
\begin{equation}\label{eqster}
\begin{array}{lcl}
\kappa \nrm{\M_{\epsilon_j}^{1,2}\Phi_j}_{\mathbf{L}^2}^2
&\geq & \nrm{\M_{\epsilon_j}^{1,2}\Phi_j'}_{\mathbf{L}^2}^2
  -2 \tilde{K}_{\mathrm{fam}}^2 \nrm{\M_{\epsilon_j}^{1,2}\Theta_j}_{\mathbf{L}^2}^2
\end{array}
\end{equation}
holds for all $j\geq 1$.
\end{lemma}
\textit{Proof.}
For convenience, we assume that $\tilde{c}_{\epsilon_j} > 0$ for all $j \ge 1$.
Recalling the decomposition (\ref{decomptheta_j}), we can use Lemma's \ref{lemma6bewijsbound1} and
\ref{lemma6bewijsbound2:phi:dphi}
to compute
\begin{equation}
\begin{array}{lcl}
\Re\ip{\M_{\epsilon_j}^{1,2}\Theta_j,\M_{\epsilon_j}^{1,2}\Phi_j'}_{\mathbf{L}^2}&=&\tilde{c}_{\epsilon_j}\Re\ip{\M_{\epsilon_j}^{1,2}\Phi_j',\M_{\epsilon_j}^{1,2}\Phi_j'}_{\mathbf{L}^2}+\Re\ip{-J_{\mathrm{mix}}\Phi_j,\Phi_j'}_{\mathbf{L}^2}\\[0.2cm]
&&\qquad +\Re\big\langle \mathcal{M}_{\epsilon_j^2}^{1,2}\big(-DF(\tilde{U}_{\epsilon_j})+\lambda_j\big)\Phi_j,\Phi_j'\big\rangle_{\mathbf{L}^2}\\[0.2cm]
&\geq & -2  \big(\tilde{K}_{F} + \lambda_{\max} \big) \nrm{\M_{\epsilon_j}^{1,2}\Phi_j}_{\mathbf{L}^2}\nrm{\M_{\epsilon_j}^{1,2}\Phi_j'}_{\mathbf{L}^2}
+\tilde{c}_{\epsilon_j} \nrm{\M_{\epsilon_j}^{1,2}\Phi_j'}_{\mathbf{L}^2}^2 .
\\[0.2cm]
%&:=&-\kappa\nrm{\M_{\epsilon_j}^{1,2}\Phi_j}_{\mathbf{L}^2}\nrm{\M_{\epsilon_j}^{1,2}\Phi_j'}_{\mathbf{L}^2}+\tilde{K}_{\mathrm{fam}}\nrm{\M_{\epsilon_j}^{1,2}\Phi_j'}_{\mathbf{L}^2}^2.
\end{array}
\end{equation}
We hence see that
\begin{equation}
\begin{array}{lcl}
\tilde{c}_{\epsilon_j} \nrm{\M_{\epsilon_j}^{1,2}\Phi_j'}_{\mathbf{L}^2}^2
 & \le &  2  \big(\tilde{K}_{F} + \lambda_{\max} \big) \nrm{\M_{\epsilon_j}^{1,2}\Phi_j}_{\mathbf{L}^2}\nrm{\M_{\epsilon_j}^{1,2}\Phi_j'}_{\mathbf{L}^2}
       + \nrm{\M_{\epsilon_j}^{1,2}\Theta_j}_{\mathbf{L}^2}\nrm{\M_{\epsilon_j}^{1,2}\Phi_j'}_{\mathbf{L}^2} .
\end{array}
\end{equation}
Dividing by $\nrm{\M_{\epsilon_j}^{1,2}\Phi_j'}_{\mathbf{L}^2}$ and squaring,
we find
\begin{equation}
\begin{array}{lcl}
\tilde{c}_{\epsilon_j}^2 \nrm{\M_{\epsilon_j}^{1,2}\Phi_j'}_{\mathbf{L}^2}^2
 & \le &  8  \big(\tilde{K}_{F} + \lambda_{\max} \big)^2 \nrm{\M_{\epsilon_j}^{1,2}\Phi_j}^2_{\mathbf{L}^2}
       + 2 \nrm{\M_{\epsilon_j}^{1,2}\Theta_j}^2_{\mathbf{L}^2}  ,
\end{array}
\end{equation}
as desired.
\qed\\

Recall the constants $(g,m, a, \beta)$ introduced in Lemma \ref{lemma6bewijsbound2}.
Throughout the remainder of this section, we set out to obtain a lower bound for the integral
\begin{equation}\begin{array}{lcl}
\mathcal{I}_j &=&
g\int\limits_{|\xi|\leq m}|\mathcal{M}_{\epsilon_j}^{1,2}\Phi_j(\xi)|^2 \, d \xi .\end{array}
\end{equation}

\begin{lemma}\label{lemma6bewijs4}
Consider the setting of %Proposition \ref{lemma6equivalent} and
Lemma \ref{lemma6bewijs1} or Lemma \ref{lemma6bewijs2}. Then the bound
\begin{equation}\label{eqsterster}
\begin{array}{lcl}
\mathcal{I}_j
&\geq &\frac{a}{2}\nrm{\M_{\epsilon_j}^{1,2}\Phi_j}_{\mathbf{L}^2}^2-\Big(\frac{1}{2a}+\beta\Big)\nrm{\M_{\epsilon_j}^{1,2}\Theta_j}_{\mathbf{L}^2}^2
\end{array}
\end{equation}
holds for all $j\geq 1$.
\end{lemma}

\textit{Proof.} Recall the decomposition \sref{decomptheta_j}.
Combining the estimates in Lemma's \ref{lemma6bewijsbound1} and
\ref{lemma6bewijsbound2} and remembering that $\Re\ip{\M_{\epsilon_j}^{1,2}\Phi_j',\M_{\epsilon_j}^{1,2}\Phi_j}_{\mathbf{L}^2}=0$,
we find
\begin{equation}\label{eq:afsch:mathcalI}
\begin{array}{lcl}
\mathcal{I}_j
%&:=&
%g\int\limits_{|x|\leq m}|\mathcal{M}_{\epsilon_j}^{1,2}\Phi_j|^2\\[0.2cm]
&\geq &
 a\nrm{\M_{\epsilon_j}^{1,2}\Phi_j}_{\mathbf{L}^2}^2-\Re\ip{\M_{\epsilon_j}^{1,2}\Theta_j,\M_{\epsilon_j}^{1,2}\Phi_j}_{\mathbf{L}^2}
    -\beta\nrm{\M_{\epsilon_j}^{1,2}\Theta_j}_{\mathbf{L}^2}^2
  \\[0.2cm]
&\geq &a\nrm{\M_{\epsilon_j}^{1,2}\Phi_j}_{\mathbf{L}^2}^2-\nrm{\M_{\epsilon_j}^{1,2}\Theta_j}_{\mathbf{L}^2}\nrm{\M_{\epsilon_j}^{1,2}\Phi_j}_{\mathbf{L}^2}-\beta\nrm{\M_{\epsilon_j}^{1,2}\Theta_j}_{\mathbf{L}^2}^2 .
\end{array}
\end{equation}
Using the standard identity $xy\leq \frac{z}{2}x^2+\frac{1}{2z}y^2$ for $x,y\in\R$ and $z > 0 $ we can estimate
\begin{equation}
\begin{array}{lcl}
\mathcal{I}_j
%&\geq &a\nrm{\M_{\epsilon_j}^{1,2}\Phi_j}_{\mathbf{L}^2}^2-\nrm{\M_{\epsilon_j}^{1,2}\Theta_j}_{\mathbf{L}^2}\nrm{\M_{\epsilon_j}^2\Phi_j}_{\mathbf{L}^2}-\beta\nrm{\M_{\epsilon_j}^{1,2}\Theta_j}_{\mathbf{L}^2}^2\\[0.2cm]
&\geq &\frac{a}{2}\nrm{\M_{\epsilon_j}^{1,2}\Phi_j}_{\mathbf{L}^2}^2-\Big(\frac{1}{2a}+\beta\Big)\nrm{\M_{\epsilon_j}^{1,2}\Theta_j}_{\mathbf{L}^2}^2 ,
\end{array}
\end{equation}
as desired.
\qed\\

\textit{Proof of Proposition \ref{lemma6equivalent}.}
Introducing the constant $\gamma=\frac{a}{2(\kappa+1)}$,
we add $\gamma$ times  (\ref{eqster}) to (\ref{eqsterster}) and find
\begin{equation}
\begin{array}{lcl}
\mathcal{I}_j+\frac{a\kappa}{2(\kappa+1)}\nrm{\M_{\epsilon_j}^{1,2}\Phi_j}_{\mathbf{L}^2}^2
&\geq &
\frac{a}{2}\nrm{\M_{\epsilon_j}^{1,2}\Phi_j}_{\mathbf{L}^2}^2-\Big(\frac{1}{2a}+\beta\Big)\nrm{\M_{\epsilon_j}^{1,2}\Theta_j}_{\mathbf{L}^2}^2\\[0.2cm]
&&\qquad +\frac{a}{2(\kappa+1)}\nrm{\M_{\epsilon_j}^{1,2}\Phi_j'}_{\mathbf{L}^2}-\frac{a\tilde{K}_{\mathrm{fam}}^2}{2(\kappa+1)}\nrm{\M_{\epsilon_j}^{1,2}\Theta_j}_{\mathbf{L}^2}^2.
\end{array}\end{equation}
We hence obtain
\begin{equation}
\begin{array}{lcl}
\mathcal{I}_j&\geq &
  \frac{a}{2(\kappa+1)}\nrm{\M_{\epsilon_j}^{1,2}\Phi_j}_{\mathbf{H}^1} -\Big(\frac{1}{2a}+\beta+\frac{a\tilde{K}_{\mathrm{fam}}^2}{2(\kappa+1)}\Big)\nrm{\M_{\epsilon_j}^{1,2}\Theta_j}_{\mathbf{L}^2}^2\\[0.2cm]
&:=&C_3-C_4\nrm{\M_{\epsilon_j}^{1,2}\Theta_j}_{\mathbf{L}^2}^2 .
\end{array}\end{equation}
Letting $j\rightarrow\infty$ in the setting of Lemma \ref{lemma6bewijs1} yields
\begin{equation}
\begin{array}{lclcl}
C_3-C_4\Lambda(\delta)&\leq &g\int\limits_{|\xi|\leq m} |\Phi(\xi)|^2  \, d \xi
%\\[0.2cm]
&\leq & gC_{\diamond} \Lambda(\delta).
\end{array}
\end{equation}
As such, we can conclude that
\begin{equation}
\begin{array}{lcl}
\Lambda(\delta)&\geq &\frac{2}{C_0}
\end{array}
\end{equation}
for some $C_0>0$, as required.
An analogous computation can be used
%The analogous computation
for the setting of Lemma \ref{lemma6bewijs2}.
\qed\\

\section{Existence of travelling waves}\label{sectionexistence}
In this section we follow the spirit of \cite[Thm. 1]{BatesInfRange}
and develop a fixed point argument to show that
(\ref{ditishetprobleem}) admits travelling wave solutions
of the form \sref{eq:mr:trv:wave:ansatz}. The main complication
is that we need $\epsilon$-uniform bounds on the supremum norm of the waveprofiles
in order to control the nonlinear terms. This can be achieved by bounding
the $\mathbf{H}^1$-norm of the perturbation, but the estimates in Proposition
\ref{theorem4equivalent} feature a problematic scaling factor on the odd component.
Fortunately Corollary \ref{theorem4H2versie2} does provide
uniform $\mathbf{H}^1$-bounds, but it requires us to take a derivative of the travelling wave system.\\

Throughout this section we will apply the results
from {\S}\ref{singularoperator} to the constant family
\begin{equation}\begin{array}{lcl}
\big( \tilde{U}_\epsilon, \tilde{c}_{\epsilon} \big) &=& \big(\overline{U}_{0}, c_0 \big),\end{array}
\end{equation}
which clearly satisfies (\asref{familyassumption}{\text{h}}). In particular,
we fix a small constant $\delta > 0$ and write
$\mathcal{L}_{\epsilon,\delta}$ for the operators given by (\ref{deftildeL}) in this setting.
We set out to construct a branch of wavespeeds $c_{\epsilon}$ and small functions
\begin{equation}\begin{array}{lcl}
  \Phi_{\epsilon} &=& (\Phi_{o;\epsilon} , \Phi_{e;\epsilon}) \in \mathbf{H}^2\end{array}
\end{equation}
in such a way that
$\overline{U}_{0}+\Phi_\epsilon$ is a solution to (\ref{nieuwetravellingwaveeq}).
A short computation shows that this is equivalent to the system
\begin{equation}
\label{eq:ex:eqv:prob:easy:form}
\begin{array}{lcl}
\mathcal{L}_{\epsilon,\delta}(\Phi_\epsilon)&=&
  \mathcal{F}_{\delta}( c_{\epsilon} , \Phi_{\epsilon}),
\end{array}
\end{equation}
which we split up by introducing the expressions
\begin{equation}
\begin{array}{lcl}
\mathcal{R}(c,\Phi)&=&
(c_0-c) \partial_\xi\big(\overline{U}_{0}+\Phi\big)
 ,
\\[0.2cm]
\mathcal{E}_0 &=&
%\left(\begin{array}{l}
\Big(-J c_0\overline{U}_{o;0}'+J F_o(\overline{U}_{o;0}), 0 \Big) ,
%\\
%0 \end{array} \right),
\\[0.2cm]
\mathcal{N}_{\#}(\Phi_{\#})&=&F_{\#}(\overline{U}_{\#;0}+\Phi_{\#})
  -DF_{\#}(\overline{U}_{\#;0})\Phi_{\#}-F_{\#}(\overline{U}_{\#;0})\\[0.2cm]
\end{array}
\end{equation}
for $\# \in \{o, e\}$
and writing
\begin{equation}\label{opsplitsennonlinearpart}
\begin{array}{lcl}
\mathcal{F}_{\delta}( c_{\epsilon} , \Phi_{\epsilon})
& = &  \mathcal{R}(c_\epsilon,\Phi_\epsilon)
+ \mathcal{E}_0
+ \big( \mathcal{N}_o(\Phi_{o;\epsilon}) , \mathcal{N}_e(\Phi_{e;\epsilon}) \big)
+ \delta \Phi .
\end{array}
\end{equation}
Notice that $\mathcal{R}$ contains a derivative of $\Phi$. It is hence crucial that $\mathcal{L}_{\epsilon,\delta}^{-1}$ gains an order of regularity, which we obtained by the framework developed in {\S}\ref{singularoperator}.

For any $\epsilon > 0$ and $\Phi \in \mathbf{H}^2$ we introduce the norm
\begin{equation}\begin{array}{lcl}
\norm{\Phi}_{\mathbf{X}_{\epsilon} }^2 &=& \norm{ \mathcal{M}^{1,2}_{\epsilon}\partial_\xi^2 \Phi }_{\mathbf{L}^2}^2
%+ \norm{ \partial_\xi\Phi}_{L^2}^2
+ \norm{ \Phi}_{\mathbf{H}^1}^2 ,\end{array}
\end{equation}
which is equivalent to the standard norm on $\mathbf{H}^2$. For any $\eta > 0$,
this allows us to introduce the set
\begin{equation}
\begin{array}{lcl}
\mathbf{X}_{\eta;\epsilon}&= &\{\Phi\in\mathbf{H}^2:\nrm{\Phi}_{\mathbf{X}_{\epsilon} } \leq \eta\}.
\end{array}
\end{equation}
For convenience, we introduce the constant
$\eta_* = \big[ 2 \nrm{\Phi_{e;0}^{\mathrm{adj}}}_{\mathbf{L}^2_e} \big]^{-1}$,
together with the formal expression
\begin{equation}
\label{defceps}
\begin{array}{lcl}
c_\delta(\Phi_e)&=&c_0+
\big[ 1 + \langle \partial_\xi\Phi_e , \overline{\Phi}^{\mathrm{adj}}_{e;0} \rangle_{\mathbf{L}^2_e} \big]^{-1}
\Big[
  \delta \langle \Phi_e, \overline{\Phi}_{e;0}^{\mathrm{adj}} \rangle_{\mathbf{L}^2_e}
  + \langle \mathcal{N}_e(\Phi_e), \overline{\Phi}_{e;0}^{\mathrm{adj}} \rangle_{\mathbf{L}^2_e}
\Big].
\end{array}
\end{equation}

\begin{lemma}
Assume that (\asref{aannamesconstanten}{\text{H}}),
(\asref{aannamesconstanten3}{\text{H}}),
(\asref{aannamespuls:even}{\text{H}}), (\asref{aannamespuls:odd}{\text{H}}) and
(\asref{extraaannamespuls}{\text{H}}) are satisfied
and pick a constant $0 < \eta \le \eta_*$.
Then  the
expression \sref{defceps} is well-defined for any $\epsilon> 0$ and any
$\Phi = (\Phi_o, \Phi_e) \in \mathbf{X}_{\eta;\epsilon}$. In addition,
the equation
\begin{equation}\begin{array}{lcl}
\big\langle \mathcal{F}_{\delta}(c,\Phi),(0,\overline{\Phi}_{e;0}^{\mathrm{adj}}) \big\rangle_{\mathbf{L}^2}&=&0\end{array}
\end{equation}
has the unique solution $c = c_\delta(\Phi_e)$.
\end{lemma}
\begin{proof}
We first note that
\begin{equation}\begin{array}{lclcl}
\langle \partial_\xi\Phi_e , \overline{\Phi}^{\mathrm{adj}}_{e;0} \rangle_{\mathbf{L}^2_e}
&\ge &- \norm{\partial_\xi\Phi_e}_{\mathbf{L}^2_e}
 \norm{ \overline{\Phi}^{\mathrm{adj}}_{e;0} }_{\mathbf{L}^2_e}
&\ge &-\frac{1}{2},\end{array}
\end{equation}
which implies that \sref{defceps} is well-defined.
The result now follows by noting that $\ip{\mathcal{E}_0,(0,\overline{\Phi}_{e;0}^{\mathrm{adj}})}_{\mathbf{L}^2}=0$ and that
\begin{equation}
\begin{array}{lcl}
\big\langle \mathcal{R}(c,\Phi),(0,\overline{\Phi}_{e;0}^{\mathrm{adj}}) \big\rangle_{\mathbf{L}^2}
& = & (c_0 - c) \Big(  \langle \overline{U}_{0;e}', \overline{\Phi}_{e;0}^{\mathrm{adj}} \rangle_{\mathbf{L}^2_e}
  + \langle  \partial_\xi\Phi_e ,\overline{\Phi}_{e;0}^{\mathrm{adj}} \rangle_{\mathbf{L}^2_e} \Big)
\\[0.2cm]
& = & (c_0 - c) \Big( 1 + \langle  \partial_\xi\Phi_e ,\overline{\Phi}_{e;0}^{\mathrm{adj}} \rangle_{\mathbf{L}^2_e} \Big),
\end{array}
\end{equation}
which implies that
\begin{equation}\begin{array}{lcl}
\big\langle \mathcal{F}_{\delta}(c,\Phi),(0,\overline{\Phi}_{e;0}^{\mathrm{adj}}) \big\rangle_{\mathbf{L}^2}
&=&   (c_0 - c) \Big( 1 + \langle  \partial_\xi\Phi_e ,\overline{\Phi}_{e;0}^{\mathrm{adj}} \rangle_{\mathbf{L}^2_e} \Big)
+ \delta \langle \Phi_e, \overline{\Phi}_{e;0}^{\mathrm{adj}} \rangle_{\mathbf{L}^2_e}
  + \langle \mathcal{N}_e(\Phi_e), \overline{\Phi}_{e;0}^{\mathrm{adj}} \rangle_{\mathbf{L}^2_e}.\end{array}
\end{equation}
\end{proof}

Consider the setting of Corollary \ref{theorem4H2versie2}
and pick $0 < \delta < \delta_0$ and $0 < \epsilon < \epsilon_0(\delta)$.
Our goal here is to find solutions to \sref{eq:ex:eqv:prob:easy:form}
by showing that the map $T_{\epsilon,\delta}:\mathbf{X}_{\eta;\epsilon}\rightarrow \mathbf{H}^2$
that acts as
\begin{equation}
\begin{array}{lcl}
T_{\epsilon,\delta}(\Phi)&=&(\mathcal{L}_{\epsilon,\delta})^{-1}\mathcal{F}_{\delta}\big(c_\delta(\Phi_e),\Phi\big)
\end{array}
\end{equation}
admits a fixed point. For any triplet $(\Phi, \Phi^A, \Phi^B) \in \mathbf{X}_{\eta;\epsilon}^3$,
the bounds in Corollary \ref{theorem4H2versie2} imply that
\begin{equation}\label{Tbound1}
\begin{array}{lcl}
\norm{T_{\epsilon,\delta}(\Phi)}_{\mathbf{X}_{\epsilon} } & \le & C_1
\big[
  \norm{\M^1_{\epsilon} \mathcal{F}_{\delta}\big(c_\delta(\Phi_e),\Phi\big)}_{\mathbf{L}^2}
  +
  \norm{\M^{1,2}_{\epsilon} \partial_\xi \mathcal{F}_{\delta}\big(c_\delta(\Phi_e),\Phi\big)}_{\mathbf{L}^2}
\big] ,
\end{array}
\end{equation}
together with
\begin{equation}\label{Tbound2}
\begin{array}{lcl}
\norm{T_{\epsilon,\delta}(\Phi^A) - T_{\epsilon,\delta}(\Phi^B)}_{\mathbf{X}_{\epsilon} } & \le &
C_1
  \norm{\M^1_{\epsilon} \Big(
    \mathcal{F}_{\delta}\big(c_\delta(\Phi^A_e),\Phi^A\big)
     - \mathcal{F}_{\delta}\big(c_\delta(\Phi^B_e),\Phi^B\big)
  \Big)}_{\mathbf{L}^2}
\\[0.2cm]
& & \qquad
  + C_1
  \norm{\M^{1,2}_{\epsilon} \partial_\xi
  \Big(
    \mathcal{F}_{\delta}\big(c_\delta(\Phi^A_e),\Phi^A\big)
     - \mathcal{F}_{\delta}\big(c_\delta(\Phi^B_e),\Phi^B\big)
  \Big) }_{\mathbf{L^2} } .
\end{array}
\end{equation}
In order to show that $T_{\epsilon,\delta}$ is a contraction mapping,
it hence suffices to obtain suitable bounds
for the terms appearing on the right-hand side  of these estimates.\\

We start by obtaining pointwise bounds on the nonlinear terms.
To this end, we compute
\begin{equation}\label{eq:Noprime}
\begin{array}{lcl}
\partial_\xi\mathcal{N}_o(\Phi_{o}) &=&
\Big( DF_o(\overline{U}_{o;0} + \Phi_o ) - DF_o(\overline{U}_{o;0}) - D^2 F_o( \overline{U}_{o;0}) \Phi_o \Big)
  \overline{U}_{o;0}'
\\[0.2cm]
& & \qquad
  + \Big( DF_o( \overline{U}_{o;0} + \Phi_o ) - DF_o(\overline{U}_{o;0} ) \Big) \partial_\xi\Phi_o
\end{array}
\end{equation}
and note that a similar identity holds for $\partial_\xi\mathcal{N}_e(\Phi_e)$.
In addition, we remark that there is a constant $K_F>0$ for which the bounds
\begin{equation}
\label{eq:ex:cnst:kf:def}\begin{array}{lcl}
\nrm{DF_{\#}(\overline{U}_{\#;0} + \Phi_{\#})}_{\infty} + \nrm{D^2F_{\#}(\overline{U}_{\#;0} + \Phi_{\#})}_{\infty}
+ \nrm{D^3 F_{\#}(\overline{U}_{\#;0} + \Phi_{\#})}_{\infty}
&<& K_F\end{array}
%\nrm{DF_o(\overline{U}_{o;0} + \Phi_o)}_{\infty} + \nrm{D^2F_o(\overline{U}_{o;0} + \Phi_o))}_{\infty}
%+ \nrm{DF_e(\overline{U}_{e;0} + \Phi_e))}_{\infty} + \nrm{D^2F_e(\overline{U}_{e;0} + \Phi_e)}_{\infty} < K_F
\end{equation}
hold for $\# \in \{o, e\}$ and all $\Phi = (\Phi_o, \Phi_e)$ that have $\norm{\Phi}_{\mathbf{H}^1} \le \eta_*$.
\\

\begin{lemma}\label{theorem1bewijs1} Assume that (\asref{aannamesconstanten}{\text{H}}),
(\asref{aannamesconstanten3}{\text{H}}), (\asref{aannamespuls:even}{\text{H}})
and
(\asref{aannamespuls:odd}{\text{H}}) are satisfied.
There exists a constant $M>0$ so that for each $\Phi=(\Phi_o,\Phi_e)\in\mathbf{H}^1$ with
$\nrm{\Phi}_{\mathbf{H}^1} \leq \eta_* $, we have the pointwise estimates
\begin{equation}
\begin{array}{lcl}
|\mathcal{N}_o(\Phi_o)|&\leq & M|\Phi_o|^2,\\[0.2cm]
|\mathcal{N}_e(\Phi_e)|&\leq & M|\Phi_e|^2.
\end{array}
\end{equation} \end{lemma}
\textit{Proof.}
%Since $F_o$ is $C^3$-smooth,
Using \cite[Thm. 2.8.3]{DUI2004} we obtain
\begin{equation}
\begin{array}{lcl}
|\mathcal{N}_o(\Phi_o)|&\leq & \frac{1}{2} K_F |\Phi_o|^2.
\end{array}
\end{equation}
The estimate for $\mathcal{N}_e$ follows similarly.\qed\\

\begin{lemma}\label{theorem1bewijs2}
Assume that (\asref{aannamesconstanten}{\text{H}}),
(\asref{aannamesconstanten3}{\text{H}}), (\asref{aannamespuls:even}{\text{H}})
and
(\asref{aannamespuls:odd}{\text{H}}) are satisfied.
There exists a constant $M>0$ so that for each $\Phi=(\Phi_o,\Phi_e)\in\mathbf{H}^1$ with
$\nrm{\Phi}_{\mathbf{H}^1} \leq \eta_* $, we have the pointwise estimates
\begin{equation}
\begin{array}{lcl}
|\partial_\xi\mathcal{N}_o(\Phi_o)|&\leq & M\big(|\partial_\xi\Phi_o||\Phi_o|+|\Phi_o|^2\big),\\[0.2cm]
|\partial_\xi\mathcal{N}_e(\Phi_e)|&\leq & M\big(|\partial_\xi\Phi_e||\Phi_e|+|\Phi_e|^2\big).
\end{array}
\end{equation} \end{lemma}
\textit{Proof.}
We rewrite (\ref{eq:Noprime}) to obtain
\begin{equation}
\begin{array}{lcl}
\partial_\xi\mathcal{N}_o(\Phi_o)&=&DF_o(\overline{U}_{o;0}+\Phi_o)\partial_\xi(\overline{U}_{o;0}+\Phi_o)-DF_o(\overline{U}_{o;0})\partial_\xi(\overline{U}_{o;0}+\Phi_o)\\[0.2cm]
&&\qquad -D^2F_o(\overline{U}_{o;0})[\Phi_o, \partial_\xi(\overline{U}_{o;0}+\Phi_o)]+D^2F_o(\overline{U}_{o;0})[\Phi_o,\partial_\xi\Phi_o].
\end{array}
\end{equation}
This allows us to use \cite[Thm. 2.8.3]{DUI2004} and obtain the pointwise estimate
\begin{equation}
\begin{array}{lcl}
|\partial_\xi\mathcal{N}_o(\Phi_o)|&\leq &
\frac{1}{2} K_F |\Phi_o|^2\big(\abs{\overline{U}_{o;0}'} +|\partial_\xi\Phi_o|\big)
+ K_F |\Phi_o||\partial_\xi\Phi_o|
\\[0.2cm]
&\leq & M\big(|\partial_\xi\Phi_o||\Phi_o|+|\Phi_o|^2\big).
\end{array}
\end{equation}
% and possibly after increasing $M$.
The estimate for $\mathcal{N}_e$ follows similarly.
\qed\\

\begin{lemma}\label{theorem1bewijs3}
Assume that (\asref{aannamesconstanten}{\text{H}}),
(\asref{aannamesconstanten3}{\text{H}}), (\asref{aannamespuls:even}{\text{H}})
and
(\asref{aannamespuls:odd}{\text{H}}) are satisfied.
There exists a constant $M>0$ so that for each
pair
\begin{equation}
\Phi^A=(\Phi^A_{o},\Phi^A_{e}) \in \mathbf{H}^1,
\qquad
\qquad
\Phi^B=(\Phi^B_{o},\Phi^B_{e})\in\mathbf{H}^1
\end{equation}
that satisfies $\nrm{\Phi^A}_{\mathbf{H}^1} \leq \eta_*$
and $\nrm{\Phi^B}_{\mathbf{H}^1} \leq \eta_* $, we have the pointwise estimates
\begin{equation}
\begin{array}{lcl}
|\mathcal{N}_o(\Phi^A_{o})-\mathcal{N}_o(\Phi^B_{o})|&\leq &
   M\big[|\Phi^A_{o}|+|\Phi^B_{o}|\big]|\Phi^A_{o}-\Phi^B_{o}|,\\[0.2cm]
|\mathcal{N}_e(\Phi^A_{e})-\mathcal{N}_e(\Phi^B_{e})|&\leq &
   M\big[|\Phi^A_{e}|+|\Phi^B_{e}|\big]|\Phi^A_{e}-\Phi^B_{e}|.
\end{array}
\end{equation} \end{lemma}
\textit{Proof.}
%Fix $\Phi^A=(\Phi^A_{o},\Phi^A_{e}) ,\Phi^B=(\Phi^B_{o},\Phi^B_{e})\in \mathbf{H}^1$ with $\nrm{\Phi^A}_{\mathbf{H}^1} \leq \eta_*$
%and $\nrm{\Phi^B}_{\mathbf{H}^1} \leq \eta_* $. Then we can write
We first compute
\begin{equation}\label{firstcontractionestimate}
\begin{array}{lcl}
\mathcal{N}_o(\Phi^A_{o})-\mathcal{N}_o(\Phi^B_{o})&= & F_o\big(\overline{U}_{o;0}+\Phi^B_{o}+(\Phi^A_{o}-\Phi^B_{o})\big)-F_o\big(\overline{U}_{o;0}+\Phi^B_{o}\big)\\[0.2cm]
&&\qquad-DF_o\big(\overline{U}_{o;0}+\Phi^B_{o}\big)\big(\Phi^A_{o}-\Phi^B_{o}\big)\\[0.2cm]
&&\qquad +\big[DF_o\big(\overline{U}_{o;0}+\Phi^B_{o}\big)-DF_o(\overline{U}_{o;0})\big]\big(\Phi^A_{o}-\Phi^B_{o}\big).\\[0.2cm]
\end{array}
\end{equation}
Applying \cite[Thm. 2.8.3]{DUI2004} twice yields the pointwise estimate
\begin{equation}
\begin{array}{lcl}
|\mathcal{N}_o(\Phi^A_{o})-\mathcal{N}_o(\Phi^B_{o})|&\leq&
K_F \Big[ \frac{1}{2} |\Phi^A_{o}-\Phi^B_{o}|^2 + \abs{\Phi^B_o} |\Phi^A_{o}-\Phi^B_{o}| \Big]
\\[0.2cm]
& \le &
2 K_F \big[|\Phi^A_{o}|+|\Phi^B_{o}|\big]|\Phi^A_{o}-\Phi^B_{o}|.
\end{array}
\end{equation}
The estimate for $\mathcal{N}_e$ follows similarly.
\qed\\

\begin{lemma}\label{theorem1bewijs4} Assume that (\asref{aannamesconstanten}{\text{H}}),
(\asref{aannamesconstanten3}{\text{H}}), (\asref{aannamespuls:even}{\text{H}}) and
(\asref{aannamespuls:odd}{\text{H}}) are satisfied. There exists a constant $M>0$
so that for each
pair
\begin{equation}
\Phi^A=(\Phi^A_{o},\Phi^A_{e}) \in \mathbf{H}^1,
\qquad
\qquad
\Phi^B=(\Phi^B_{o},\Phi^B_{e})\in\mathbf{H}^1
\end{equation}
that satisfies $\nrm{\Phi^A}_{\mathbf{H}^1} \leq \eta_*$
and $\nrm{\Phi^B}_{\mathbf{H}^1} \leq \eta_* $ we have the pointwise estimates
\begin{equation}
\label{eq:ex:est:delta:n:deriv}
\begin{array}{lcl}
|\partial_\xi\mathcal{N}_{\#}(\Phi^A_{\#})-\partial_\xi\mathcal{N}_{\#}(\Phi^B_{\#})|&\leq &
 M\Big[|\partial_\xi\Phi^A_{\#}|+|\Phi^A_{\#}|+ \abs{\partial_\xi\Phi^B_{\#}}+|\Phi^B_{\#}| \Big]
   |\Phi^A_{\#}-\Phi^B_{\#}|
\\[0.2cm]
& & \qquad
  +M\Big[ \abs{\Phi^A_{\#}} +  |\Phi^B_{\#}| \Big]|\partial_\xi(\Phi^A_{\#}-\Phi^B_{\#})|,
\end{array}
\end{equation}
for $\# \in \{o, e\}$.
\end{lemma}
\textit{Proof.}
Differentiating \sref{firstcontractionestimate} line by line, we obtain
\begin{equation}
\begin{array}{lcl}
\partial_\xi\mathcal{N}_o(\Phi^A_{o})-\partial_\xi\mathcal{N}_o(\Phi^B_{o})&=&
d_1+d_2+d_3 ,
\end{array}
\end{equation}
with
\begin{equation}\begin{array}{lcl}
d_1&=&DF_o\big(\overline{U}_{o;0}+\Phi^B_{o}+(\Phi^A_{o}-\Phi^B_{o}) \big)
  \big(\overline{U}_{o;0}'+\partial_\xi\Phi^B_{o} + \partial_\xi(\Phi^A_o - \Phi^B_o) \big)
\\[0.2cm]
& & \qquad
  -DF_o\big(\overline{U}_{o;0}+\Phi^B_{o}\big)\partial_\xi\big(\overline{U}_{o;0}+\Phi^B_{o}\big) ,
 \\[0.2cm]
d_2 & = &
 -D^2F_o\big(\overline{U}_{o;0}+\Phi^B_{o}\big)\big[\Phi^A_{o}-\Phi^B_{o}, \partial_\xi(\overline{U}_{o;0}+\Phi^B_{o})\big]
 - DF_o\big(\overline{U}_{o;0}+\Phi^B_{o}\big)  \partial_\xi(\Phi^A_o - \Phi^B_o) ,
\\[0.2cm]
d_3&=&\big[DF_o\big(\overline{U}_{o;0}+\Phi^B_{o}\big)
  -DF_o(\overline{U}_{o;0})\big] \partial_\xi \big(\Phi^A_{o}-\Phi^B_{o}\big)
\\[0.2cm]
& & \qquad
  +  D^2F_o\big(\overline{U}_{o;0}+\Phi^B_{o}\big)[\partial_\xi(\overline{U}_{o;0} + \Phi^B_o), \Phi^A_{o}-\Phi^B_{o} \big]
%\\[0.2cm]
%& & \qquad
  - D^2F_o(\overline{U}_{o;0}) [\overline{U}_{o;0}' , \Phi^A_{o}-\Phi^B_{o} \big].
\end{array}
\end{equation}
Upon introducing the expressions
\begin{equation}
\begin{array}{lcl}
d_{I} & = & DF_o\big(\overline{U}_{o;0}+\Phi^B_{o}+(\Phi^A_{o}-\Phi^B_{o}) \big)\partial_\xi\big(\overline{U}_{o;0}+\Phi^B_{o}  \big)
  -DF_o\big(\overline{U}_{o;0}+\Phi^B_{o}\big)\partial_\xi\big(\overline{U}_{o;0}+\Phi^B_{o}\big)
\\[0.2cm]
& & \qquad
   -D^2F_o\big(\overline{U}_{o;0}+\Phi^B_{o}\big)\big[\Phi^A_{o}-\Phi^B_{o},
     \partial_\xi(\overline{U}_{o;0}+\Phi^B_{o})\big] ,
\\[0.2cm]
d_{II} & = &
\Big[DF_o\big(\overline{U}_{o;0}+\Phi^B_{o}+(\Phi^A_{o}-\Phi^B_{o}) \big)
 -  DF_o\big(\overline{U}_{o;0}+\Phi^B_{o}\big) \Big]  \partial_\xi(\Phi^A_o - \Phi^B_o) ,
\end{array}
\end{equation}
we see that
\begin{equation}\begin{array}{lcl}
d_1 + d_2 &= &d_{I} + d_{II}.\end{array}
\end{equation}
Applying  \cite[Thm. 2.8.3]{DUI2004} we obtain the bounds
\begin{equation}
\begin{array}{lcl}
\abs{d_I} & \le &
\frac{1}{2} K_F \abs{\Phi^A_{o}-\Phi^B_{o}}^2
\big[ \abs{\overline{U}_{o;0}'}+\abs{\partial_\xi\Phi^B_{o}} \big],
\\[0.2cm]
\abs{d_{II}} & \le &  K_F  \abs{\Phi^A_{o}-\Phi^B_{o}}
  \abs{  \partial_\xi(\Phi^A_o - \Phi^B_o) }.
\end{array}
\end{equation}
In addition, the expressions
\begin{equation}
\begin{array}{lcl}
d_{III} & = &
\big[DF_o\big(\overline{U}_{o;0}+\Phi^B_{o}\big)-DF_o(\overline{U}_{o;0})\big]
  \partial_\xi \big(\Phi^A_{o}-\Phi^B_{o}\big),
\\[0.2cm]
d_{IV} & = &
D^2F_o\big(\overline{U}_{o;0}+\Phi^B_{o}\big)[\overline{U}_{o;0}' , \Phi^A_{o}-\Phi^B_{o} \big]
- D^2F_o(\overline{U}_{o;0}) [\overline{U}_{o;0}' , \Phi^A_{o}-\Phi^B_{o} \big]  ,
\\[0.2cm]
d_{V} & = &
 D^2F_o\big(\overline{U}_{o;0}+\Phi^B_{o}\big)[\partial_\xi\Phi^B_o, \Phi^A_{o}-\Phi^B_{o} \big]
\end{array}
\end{equation}
allow us to write
\begin{equation}\begin{array}{lcl}
d_3 &=& d_{III} + d_{IV} + d_{V} .\end{array}
\end{equation}
Applying  \cite[Thm. 2.8.3]{DUI2004} we may estimate
\begin{equation}
\begin{array}{lcl}
\abs{d_{III}} & \le &
 K_F \abs{\Phi^B_o} \abs{  \partial_\xi(\Phi^A_o - \Phi^B_o) } ,
\\[0.2cm]
\abs{d_{IV}} & \le &
   K_F \abs{\Phi^B_o} \abs{  \Phi^A_{o}-\Phi^B_{o} } ,
\\[0.2cm]
\abs{d_{V}} & \le &
  K_F \abs{\partial_\xi\Phi^B_o} \abs{  \Phi^A_{o}-\Phi^B_{o} } .
\end{array}
\end{equation}
These bounds can all be absorbed into \sref{eq:ex:est:delta:n:deriv}.
The estimate for $\mathcal{N}_e$ follows similarly.\qed\\

With the above pointwise bounds in hand, we are ready to estimate the nonlinearities
in the appropriate scaled function spaces. To this end,
we introduce the notation
\begin{equation}\begin{array}{lcl}
\mathcal{N}(\Phi) &=& \big( \mathcal{N}_o(\Phi_o), \mathcal{N}_e(\Phi_e) \big)\end{array}
\end{equation}
for any $\Phi = (\Phi_o, \Phi_e) \in \mathbf{H}^1$.

\begin{lemma}\label{theorem1bewijs4.5}
Assume that (\asref{aannamesconstanten}{\text{H}}),
(\asref{aannamesconstanten3}{\text{H}}), (\asref{aannamespuls:even}{\text{H}})
and
(\asref{aannamespuls:odd}{\text{H}}) are satisfied.
There exists a constant $K_{\mathcal{N}}>0$ so that for each $0<\eta\leq \eta_*$, each $\epsilon > 0$
and each triplet $(\Phi,\Phi^A,\Phi^B)\in\mathbf{X}_{\eta;\epsilon}^3$ we have the bounds
\begin{equation}
    \begin{array}{lcl}
         \nrm{\M_{\epsilon}^{1}\mathcal{N}(\Phi)}_{\mathbf{L}^2}&\leq & K_{\mathcal{N}}\eta^2,\\[0.2cm]
         \nrm{\M_{\epsilon}^{1,2}\partial_\xi\mathcal{N}(\Phi)}_{\mathbf{L}^2}&\leq & K_{\mathcal{N}}\eta^2,\\[0.2cm]
         \nrm{\M_{\epsilon}^{1}\big(\mathcal{N}(\Phi^A)-\mathcal{N}(\Phi^B)\big)}_{\mathbf{L}^2}&\leq &
           K_{\mathcal{N}}\eta\nrm{\Phi^A-\Phi^B}_{\mathbf{L}^2},\\[0.2cm]
         \nrm{\M_{\epsilon}^{1,2}\partial_\xi\big(\mathcal{N}(\Phi^A)-\mathcal{N}(\Phi^B)\big)}_{\mathbf{L}^2}&\leq &
           K_{\mathcal{N}}\eta\Big(\nrm{\Phi^A-\Phi^B}_{\mathbf{L}^2}
           +\nrm{\partial_\xi(\Phi^A-\Phi^B)}_{\mathbf{L}^2}\Big).
    \end{array}
\end{equation}
\end{lemma}
\textit{Proof.} All bounds follow immediately from Lemma's
\ref{theorem1bewijs1}-\ref{theorem1bewijs4} upon
using the Sobolev estimate  $\nrm{\phi}_\infty\leq C_1' \nrm{\phi}_{H^1}$
to write
\begin{equation}
    \begin{array}{lclclcl}
       \nrm{\Phi_{o}}_\infty&\leq &C_1' \eta, & \qquad &
           \nrm{\partial_\xi\Phi_{o}}_\infty&\leq &C_1' \frac{\eta}{\epsilon},\\[0.2cm]
       \nrm{\Phi_{e}}_\infty&\leq &C_1' \eta, & \qquad &   \nrm{\partial_\xi\Phi_{e}}_\infty&\leq &C_1'\eta ,
    \end{array}
\end{equation}
with identical bounds for $\Phi^A$ and $\Phi^B$.
%using the estimate
%$\nrm{\phi}_\infty\leq C\nrm{\phi}_{H^1}$ for $\phi\in H^1$ and some $C>0$ which is independent
%of $\eta$ and $\epsilon$.
\qed\\

\begin{lemma}\label{lemma:afschE0}
Assume that (\asref{aannamesconstanten}{\text{H}}),
(\asref{aannamesconstanten3}{\text{H}}),
(\asref{aannamespuls:even}{\text{H}}) and
(\asref{aannamespuls:odd}{\text{H}}) are satisfied. Then there exists a constant $K_{\mathcal{E}}>0$
so that for each $\epsilon > 0$ we have the bound
\begin{equation}
    \begin{array}{lcl}
    \nrm{\M_{\epsilon}^{1}\mathcal{E}_0}_{\mathbf{L}^2}+\nrm{\M_{\epsilon}^{1,2}\partial_\xi \mathcal{E}_0}_{\mathbf{L}^2}&\leq & \epsilon K_{\mathcal{E}}.
    \end{array}
\end{equation}
\end{lemma}
\textit{Proof.}
The structure of the matrix $J$ allows us to bound
\begin{equation}\begin{array}{lclcl}
\norm{ \M_{\epsilon}^{1} \mathcal{E}_0 }_{\mathbf{L}^2}
&\le &\epsilon \norm{ \mathcal{E}_0 }_{\mathbf{L}^2},
\qquad \qquad
\norm{ \M_{\epsilon}^{1,2} \partial_\xi \mathcal{E}_0 }_{\mathbf{L}^2}
&\le &\epsilon \norm{ \partial_\xi \mathcal{E}_0 }_{\mathbf{L}^2}.\end{array}
\end{equation}
The result hence follows from the inclusions
\begin{equation}
\overline{U}_{o;0}' \in \mathbf{H}^1_o,
\qquad \qquad
F_o(\overline{U}_{o;0})\in \mathbf{H}^1_o .
\end{equation}
The first of these can be obtained by differentiating
\sref{nagumolde:even}
and \sref{eq:mr:eq:fro:ovl:w:o:zero}. The second inclusion
follows from the fact that
$\overline{U}_{o;0}$ converges exponentially fast to its
limiting values, which are zeroes of $F_o$.
\qed\\

\begin{lemma}\label{theorem1bewijs5} Assume that (\asref{aannamesconstanten}{\text{H}}),
(\asref{aannamesconstanten3}{\text{H}}), (\asref{aannamespuls:even}{\text{H}}),
(\asref{aannamespuls:odd}{\text{H}}) and (\asref{extraaannamespuls}{\text{H}}) are satisfied.
Then there exists a constant $K_{c}>0$ in such a way that for each $0<\eta\leq \eta_*$,
each $\epsilon > 0$, each $\delta > 0$  and each triplet $(\Phi,\Phi^A,\Phi^B)\in\mathbf{X}_{\eta;\epsilon}^3$
we have the bounds
\begin{equation}
\begin{array}{lcl}
|c_\delta(\Phi_e)-c_0|&\leq &K_{c}\big[\delta \eta+\eta^2\big],\\[0.2cm]
|c_\delta(\Phi^A_e)-c_\delta(\Phi^B_e)|&\leq &K_{c}\big(\delta+\eta\big)\nrm{\Phi^A-\Phi^B}_{\mathbf{L}^2} .
\end{array}
\end{equation}
\end{lemma}
\textit{Proof.}
Since we only need to use regular $L^2$-norms for these estimates,
the proof of \cite[Lemma 4.4]{HJHFHNINFRANGEFULL} also applies here.
\qed\\

\begin{lemma}\label{lemma:afscR}
Assume that (\asref{aannamesconstanten}{\text{H}}),
(\asref{aannamesconstanten3}{\text{H}}),
(\asref{aannamespuls:even}{\text{H}}), (\asref{aannamespuls:odd}{\text{H}}) and
(\asref{extraaannamespuls}{\text{H}}) are satisfied. Then there exists a constant $K_{\mathcal{R}}>0$ in such a way that
for each $0<\eta\leq \eta_*$, each $0 < \epsilon <1$,
each $\delta > 0$  and each triplet $(\Phi,\Phi^A,\Phi^B)\in\mathbf{X}_{\eta;\epsilon}^3$
we have the bound
\begin{equation}
    \begin{array}{lcl}
    \nrm{\M_{\epsilon}^{1}\mathcal{R}(c_\delta(\Phi_e),\Phi)}_{\mathbf{L}^2}
      +\nrm{\M_{\epsilon}^{1,2}\partial_\xi \mathcal{R}(c_\delta(\Phi_e),\Phi)}_{\mathbf{L}^2}
      &\leq & K_{\mathcal{R}}[\delta\eta+\eta^2].
    \end{array}
\end{equation}
Writing
\begin{equation}
    \begin{array}{lcl}
        \Delta_{AB} \mathcal{R}&:=& \mathcal{R}(c_\delta(\Phi^A_e),\Phi^A)-\mathcal{R}(c_\delta(\Phi^B_e),\Phi^B),
    \end{array}
\end{equation}
we also have the bound
\begin{equation}
\label{eq:ex:bnds:on:delta:r}
    \begin{array}{lcl}
    \nrm{\M_{\epsilon}^1 \Delta_{AB} \mathcal{R}}_{\mathbf{L}^2}
      +\nrm{\M_{\epsilon}^{1,2}\partial_\xi \Delta_{AB}\mathcal{R}}_{\mathbf{L}^2}&\leq &
       K_{\mathcal{R}}\big(\delta+\eta)\nrm{\Phi^A-\Phi^B }_{\mathbf{L}^2}
     \\[0.2cm]
       & & \qquad
         + \eta K_{\mathcal{R}}( \eta + \delta) \nrm{\partial_\xi(\Phi^A - \Phi^B)}_{\mathbf{L^2}}
         \\[0.2cm]
       & & \qquad
         +  \eta K_{\mathcal{R}}( \eta + \delta) \nrm{\M_{\epsilon}^{1,2}\partial_\xi^2(\Phi^A - \Phi^B)}_{\mathbf{L^2}}
    \end{array}
\end{equation}
\end{lemma}
\textit{Proof.}
Using Lemma \ref{theorem1bewijs5} we immediately obtain the bound
\begin{equation}
    \begin{array}{lcl}
    \nrm{\M_{\epsilon}^{1}\mathcal{R}(c_\delta(\Phi_e),\Phi)}_{\mathbf{L}^2}&\leq & K_{c}\big[\delta \eta+\eta^2\big]
      \Big(\nrm{\M_{\epsilon}^{1}\partial_\xi\Phi}_{\mathbf{L}^2}+\nrm{\M_{\epsilon}^{1}\overline{U}_0'}_{\mathbf{L}^2}\Big)\\[0.2cm]
    &\leq & K_{c} \big[\delta \eta+\eta^2\big]
      \Big(  \eta + \nrm{\overline{U}_0'}_{\mathbf{L}^2}\Big) ,
%    C[\delta\eta+\eta^2]
    \end{array}
\end{equation}
%for some $C>0$.
together with
\begin{equation}
    \begin{array}{lcl}
    \nrm{\M_{\epsilon}^{1,2}\partial_\xi\mathcal{R}(c_\delta(\Phi_e),\Phi)}_{\mathbf{L}^2}&\leq &
      K_{c}\big[\delta \eta+\eta^2\big]\Big(\nrm{\M_{\epsilon}^{1,2}\partial_\xi^2\Phi}_{\mathbf{L}^2}+\nrm{\M_{\epsilon}^{1,2}\overline{U}_0''}_{\mathbf{L}^2}\Big)\\[0.2cm]
    &\leq &
      K_{c}\big[\delta \eta+\eta^2\big]\Big(
        \eta +\nrm{\overline{U}_0''}_{\mathbf{L}^2}\Big) .
        \\[0.2cm]
    %C[\delta\eta+\eta^2].
    \end{array}
\end{equation}
In addition, we may compute
\begin{equation}
    \begin{array}{lcl}
        \Delta_{AB} \mathcal{R}
        %&:=&
          %\mathcal{R}(c_\epsilon(\Phi^A),\Phi^A)-\mathcal{R}(c_\epsilon(\Phi^B),\Phi^B)\\[0.2cm]
        &=&\big(c_\delta(\Phi^B_e)-c_\delta(\Phi^A_e)\big)\partial_\xi\big( \overline{U}_0 + \Phi^A \big)\\[0.2cm]
         && \qquad
            +\big(c_0 - c_\delta(\Phi^B_e)\big)\partial_\xi(\Phi^A-\Phi^B),
    \end{array}
\end{equation}
which allows us to estimate
\begin{equation}
    \begin{array}{lcl}
          \nrm{\M_{\epsilon}^1\Delta_{AB}\mathcal{R}}_{\mathbf{L}^2}&\leq &
          K_{c}\big(\delta+\eta\big)\nrm{\Phi^A-\Phi^B}_{\mathbf{L}^2}
            \big(\nrm{\M_{\epsilon}^1\overline{U}_0'}_{\mathbf{L}^2}+\nrm{\M_{\epsilon}^1\partial_\xi\Phi^A}_{\mathbf{L}^2}\big)\\[0.2cm]
          &&\qquad +K_{c}\big[\delta \eta+\eta^2\big]\nrm{\M_{\epsilon}^1\partial_\xi(\Phi^A-\Phi^B)}_{\mathbf{L}^2}\\[0.2cm]
          &\leq &
            K_{c}\big(\delta+\eta\big)\nrm{\Phi^A-\Phi^B}_{\mathbf{L}^2}
              \big( \nrm{\overline{U}_0'}_{\mathbf{L}^2} + \eta \big)
          \\[0.2cm]
           & & \qquad
             + K_{c} \big[\delta \eta+\eta^2\big]\nrm{\partial_\xi(\Phi^A-\Phi^B)}_{\mathbf{L}^2} ,
    \end{array}
\end{equation}
together with
\begin{equation}
    \begin{array}{lcl}
          \nrm{\M_{\epsilon}^{1,2}\partial_\xi \Delta_{AB}\mathcal{R}}_{\mathbf{L}^2}
          &\leq &
             K_{c}\big(\delta+\eta\big)\nrm{\Phi^A-\Phi^B}_{\mathbf{L}^2}
            \big(\nrm{\M_{\epsilon}^{1,2}\overline{U}_0''}_{\mathbf{L}^2}+\nrm{\M_{\epsilon}^{1,2}\partial_\xi\Phi^A}_{\mathbf{L}^2}\big)\\[0.2cm]
          &&\qquad +K_{c}\big[\delta \eta+\eta^2\big]\nrm{\M_{\epsilon}^{1,2}\partial_\xi^2(\Phi^A-\Phi^B)}_{\mathbf{L}^2}
\\[0.2cm]
 &\leq &
            K_{c}\big(\delta+\eta\big)\nrm{\Phi^A-\Phi^B}_{\mathbf{L}^2}
              \big( \nrm{\overline{U}_0''}_{\mathbf{L}^2} + \eta \big)
          \\[0.2cm]
           & & \qquad
             + K_{c} \big[\delta \eta+\eta^2\big]\nrm{\M_{\epsilon}^{1,2}\partial_\xi(\Phi^A-\Phi^B)}_{\mathbf{L}^2}.
\\[0.2cm]
    \end{array}
\end{equation}
These terms can all be absorbed into \sref{eq:ex:bnds:on:delta:r}.
\qed\\

\textit{Proof of Theorem \ref{maintheorem}.}
Using Lemma's \ref{theorem1bewijs4.5}, \ref{lemma:afschE0} and \ref{lemma:afscR},
together with the decomposition (\ref{opsplitsennonlinearpart}) and the estimates
(\ref{Tbound1})-(\ref{Tbound2}), we find that there exists a constant $K_{T}>0$
for which the bounds
\begin{equation}
\begin{array}{lcl}
\norm{T_{\epsilon,\delta}(\Phi)}_{\mathbf{X}_{\epsilon} }&\leq & K_{T}\Big[\delta\eta+\eta^2+\epsilon\Big],\\[0.2cm]
\norm{T_{\epsilon,\delta}(\Phi^A) - T_{\epsilon,\delta}(\Phi^B)}_{\mathbf{X}_{\epsilon} }&\leq &K_{T}\Big[\delta+\eta\Big]\nrm{\Phi^A-\Phi^B}_{\mathbf{X}_{\epsilon}}
\end{array}
\end{equation}
hold for any $\eta\leq\eta_*$, any $0<\epsilon<\epsilon_0(\delta)$ and any triplet
$(\Phi, \Phi^A, \Phi^B) \in \mathbf{X}_{\eta;\epsilon}^3$. As such, we fix
\begin{equation}
\begin{array}{lcl}
\delta&=&\frac{1}{3K_{T}},\enskip \eta=\min\{\eta_*,\frac{1}{3K_{T}}\}.
\end{array}
\end{equation}
Finally, we select a small positive $\epsilon_*$ such that $\epsilon_*\leq \epsilon_0(\delta)$ and $\epsilon_*\leq\frac{1}{3K_{T}}\eta$.
We conclude that for each $0<\epsilon\leq\epsilon_*$, $T$ maps $\mathbf{X}_{\eta;\epsilon}$ into itself and is a contraction.
This completes the proof.\qed

\section{Stability of travelling waves}\label{sectionstability}
Introducing the family
\begin{equation}\begin{array}{lcl}
\big(\tilde{U}_{\epsilon}, \tilde{c}_{\epsilon} \big)
 &=& \big( \overline{U}_{\epsilon}, c_{\epsilon} \big),\end{array}
\end{equation}
which satisfies (\asref{familyassumption}{\text{h}})
on account of Theorem \ref{maintheorem},
we see that the theory developed in
{\S}\ref{singularoperator} applies
to the operators
\begin{equation}
\overline{\L}_{\epsilon, \lambda}: \mathbf{H}^1 \to \mathbf{L}^2
\end{equation}
that act as
\begin{equation}\begin{array}{lcl}
\overline{\L}_{\epsilon,\lambda}
&=& c_\epsilon\frac{d}{d \xi}
- \mathcal{M}^1_{1/\epsilon^2} J_{\mathrm{mix}}
- DF (\overline{U}_{\epsilon} ) + \lambda .\end{array}
\end{equation}
We emphasize that these operators
are associated to the linearization of the travelling
wave system \sref{nieuwetravellingwaveeq}
around the wave solutions $(\overline{U}_{\epsilon}, c_{\epsilon})$.
For convenience, we also introduce the shorthand
\begin{equation}
\label{eq:ss:def:ovl:L}\begin{array}{lclcl}
\overline{\L}_{\epsilon}
&= &\overline{\L}_{\epsilon,0}
&=& c_\epsilon\frac{d}{d \xi}
- \mathcal{M}^1_{1/\epsilon^2} J_{\mathrm{mix}}
- DF (\overline{U}_{\epsilon} ) .\end{array}
\end{equation}
We remark that the spectrum of $\overline{\L}_{\epsilon}$
is $2 \pi i c_{\epsilon}$-periodic on account of the identity
\begin{equation}\begin{array}{lcl}
\big(\overline{\L}_{\epsilon} + \lambda \big) e^{2 \pi i \cdot }
&= &e^{2 \pi i \cdot} \big(\overline{\L}_{\epsilon} + \lambda + 2 \pi i c_{\epsilon} \big).\end{array}
\end{equation}
As a final preparation, we note that there exists a constant $\overline{K}_F > 0$ for which the bound
\begin{equation}
\label{eq:ss:kf:def}\begin{array}{lcl}
\nrm{DF_o(\overline{U}_{o;\epsilon})}_{\infty} + \nrm{D^2F_o(\overline{U}_{o;\epsilon})}_{\infty}
+ \nrm{DF_e(\overline{U}_{e;\epsilon})}_{\infty} + \nrm{D^2F_e(\overline{U}_{e;\epsilon})}_{\infty} &\leq &
\overline{K}_F\end{array}
\end{equation}
holds for all $0 < \epsilon < \epsilon_*$.\\

Our main task here is to reverse the parameter dependency
used in {\S}\ref{singularoperator}. In particular,
for a fixed small value of $\epsilon > 0$ we study
the behaviour of the map
$\lambda \mapsto \overline{\L}_{\epsilon, \lambda}$.
This allows us to obtain the main result of this section,
which lifts the spectral stability assumptions (HS1) and (HS2)
to the full system \sref{nieuwetravellingwaveeq}. This
can subsequently be turned into a nonlinear stability result
by applying the theory developed in \cite{HJHSTBFHN}.

\begin{proposition}\label{maintheorem2}
Assume that (\asref{aannamesconstanten}{\text{H}}), (\asref{aannamesconstanten3}{\text{H}}),
(\asref{aannamespuls:even}{\text{H}}), (\asref{aannamespuls:odd}{\text{H}}),
(\asref{extraaannamespuls}{\text{H}}) and
(\asref{extraextraaannamespuls}{\text{H}}) are satisfied.
Then there exists a constant $\epsilon_{**}>0$ so that for each $0<\epsilon<\epsilon_{**}$ and each
$\lambda\in\C \setminus  2 \pi i c_{\epsilon} \mathbb{Z} $
with $\Re\lambda\geq-\lambda_*$, the operator $\OL_{\epsilon,\lambda}$ is invertible.
In addition, we have
\begin{equation}\begin{array}{lcl}
\mathrm{Ker} \big( \overline{\L}_{\epsilon, 0} \big) &=& \mathrm{span}\big( \overline{U}_{\epsilon}' \big)\end{array}
\end{equation}
together with $\overline{U}_{\epsilon}' \notin \mathrm{Range} \big( \overline{\L}_{\epsilon,0} \big)$.
\end{proposition}

\noindent\textit{Proof of Theorem \ref{nonlinearstability}.}
For $j\in\Z$ we introduce the new variables
\begin{equation}\begin{array}{lcl}
\big( u_{j;o}, w_{j;o},   u_{j;e}, w_{j;e} \big)
 &=& \big(u_{2j + 1} , w_{2j + 1}, u_{2j} , w_{2j} \big),\end{array}
\end{equation}
which allows us to reformulate the
2-periodic system (\ref{ditishetprobleem})
as the equivalent $2(n+k)$-component  system
\begin{equation}
 \label{eq:ss:alternatesystem}
    \begin{array}{lcl}
    \dot{u}_{j;o}(t)&=&
  \frac{1}{\epsilon^2}\D\big[u_{j+1;e}(t)+u_{j;e}(t)-2u_{j;o}(t)\big]
 +f_o\big(u_{j;o}(t),w_{j;o}(t)\big),
 \\[0.2cm]
\dot{u}_{j;o}(t)&=&g_o\big(u_{j;o}(t),w_{j;o}(t)\big),
\\[0.2cm]
    \dot{u}_{j;e}(t)&=&
  \D\big[u_{j;o}(t)+u_{j-1;o}(t)-2u_{j;e}(t)\big]
 +f_e\big(u_{j;e}(t),w_{j;e}(t)\big),
 \\[0.2cm]
\dot{w}_{j;e}(t)&=&g_e\big(u_{j;e}(t),w_{j;e}(t)\big) ,
    \end{array}
\end{equation}
which is spatially homogeneous.\\

On account of Theorem \ref{maintheorem} and Proposition \ref{maintheorem2}, it is clear that
\sref{eq:ss:alternatesystem} satisfies the conditions (HV), (HS1)-(HS3) from \cite{HJHSTBFHN}.
An application of \cite[Proposition 2.1]{HJHSTBFHN} immediately yields the desired result.
\qed\\

\subsection{The operator $\overline{\L}_{\epsilon}$}\label{sectionfredholm}
\label{sec:stb:zero}

%Observe first that $\overline{\L}_{\epsilon}$ is a Fredholm operator with index zero on account of
%Lemma \ref{Lepsfredholm}. Using \cite[Thm. A]{MPA}, we see that the same is true
%for the formal adjoint $\overline{\L}_{\epsilon}^{\mathrm{adj}}: \mathbf{H}^1 \to \mathbf{L}^2$
%that acts as
%\begin{equation}
%    \begin{array}{lcl}
%   \overline{\L}_{\epsilon}^{\mathrm{adj}}&=& -c_\epsilon\frac{d}{d \xi}
%- J_{\mathrm{mix}} \mathcal{M}^1_{1/\epsilon^2}
%- DF (\overline{U}_{\epsilon} )^{T} .
%    \end{array}
%\end{equation}
Observe first that $\overline{\L}_{\epsilon}$ is a Fredholm operator with index zero on account of
Lemma \ref{Lepsfredholm}.
Our goal in this subsection is to establish the following characterization
of the kernel and range of
%these two operators
this operator. %$\OL_\epsilon$. 
We note that item (ii)
implies that the zero eigenvalue of $\OL_{\epsilon}$ is simple.
\begin{proposition}\label{samenvattingh5} Assume that (\asref{aannamesconstanten}{\text{H}}),
(\asref{aannamesconstanten3}{\text{H}}), (\asref{aannamespuls:even}{\text{H}}), (\asref{aannamespuls:odd}{\text{H}}),
(\asref{extraaannamespuls}{\text{H}}) and (\asref{extraextraaannamespuls}{\text{H}}) are satisfied.
Then there exists % constants $K_{\mathrm{unif}} > 0$ and 
a constant $\epsilon_{**}>0$,
%together with a family $\Phi_\epsilon^{\mathrm{adj}} \in \mathbf{H}^1$,
%defined for $0<\epsilon<\epsilon_{**}$,
so that the following properties hold
for all $0<\epsilon<\epsilon_{**}$.
%such $\epsilon$.
\begin{enumerate}[label=(\roman*)]
\item We have the identity
\begin{equation}\label{kernvanlh}\begin{array}{lcl}\ker (\OL_\epsilon)&=&\span\{\overline{U}_{\epsilon}' \}. \\[0.2cm]
%&=&\{\Psi\in \mathbf{L}^2:\ip{\Psi,\Theta}_{\mathbf{L}^2}=0\text{ for %all }\Theta\in \Range %(\OL_\epsilon^{\mathrm{adj}})\}
\end{array}\end{equation}
%together with
%\begin{equation}\label{kernvanlhster}\begin{array}{lcl}\ker(\OL_\epsil%on^{\mathrm{adj}})&=&\span\{\Phi_\epsilon^{\mathrm{adj}}\}\\[0.2cm]
%&=&\{\Psi\in \mathbf{L}^2:\ip{\Psi,\Theta}_{\mathbf{L}^2}=0\text{ for %all }\Theta\in \Range (\OL_\epsilon)\}.\end{array}\end{equation}
\item 
We have $ \overline{U}_{\epsilon}'\notin\Range(\OL_\epsilon)$.
\item The function $\xi \mapsto \overline{U}_{\epsilon}'(\xi)$
%and $\xi \mapsto \Phi_\epsilon^{\mathrm{adj}}(\xi)$
together with its derivative
decays exponentially fast as $\abs{\xi} \to \pm \infty$.

\end{enumerate}
\end{proposition}

At times, our discussion closely follows the lines of \cite[\S 4-5]{HJHFHNINFRANGE}.
The novel ingredient here however is that we do not need to modify the spectral convergence
argument from \S\ref{singularoperator} to ensure that it also applies to the adjoint operator.
Indeed, we show that all the essential properties can be obtained
from the following quasi-inverse for $\OL_\epsilon$,
which can be constructed by mimicking the approach of \cite[Prop. 3.2]{HJHBDF}.

\begin{lemma}\label{prop3.2bdf} Assume that (\asref{aannamesconstanten}{\text{H}}), (\asref{aannamesconstanten3}{\text{H}}),
(\asref{aannamespuls:even}{\text{H}}), (\asref{aannamespuls:odd}{\text{H}}), (\asref{extraaannamespuls}{\text{H}}) and
(\asref{extraextraaannamespuls}{\text{H}}) are satisfied
and pick a sufficiently small constant $\epsilon_{**} > 0$.
Then for every $0 < \epsilon < \epsilon_{**}$
%There exists $0<\epsilon_{**}\leq \epsilon_*$ together
there exist linear maps
\begin{equation}\begin{array}{lcl}
\overline{\gamma}_\epsilon& : & \mathbf{L}^2 \rightarrow \R\\[0.2cm]
\overline{\mathcal{L}}_\epsilon^{\mathrm{qinv}}& : &\mathbf{L}^2 \rightarrow \mathbf{H}^1,
\end{array}\end{equation}
so that for all $\Theta\in \mathbf{L}^2$ the pair
\begin{equation}\begin{array}{lcl}(\gamma,\Psi)&=&
(\overline{\gamma}_\epsilon\Theta,\overline{\mathcal{L}}_\epsilon^{\mathrm{qinv}}\Theta)\end{array}\end{equation}
is the unique solution to the problem
\begin{equation}\label{bdf1}\begin{array}{lcl}\OL_\epsilon\Psi&=&\Theta+\gamma \overline{U}_0' \end{array}\end{equation}
that satisfies the normalisation condition
\begin{equation}\label{bdf2}\begin{array}{lcl}\ip{(0,\Phi_{e;0}^{\mathrm{adj}}),\Psi}_{\mathbf{L}^2}&=&0.
\end{array}\end{equation}
In addition, there exists $C>0$ such that for all $0<\epsilon<\epsilon_{**}$ and all $\Theta\in \mathbf{L}^2$ we have the bound
\begin{equation}
\label{eq:st:qinv:unif:bnd}
\begin{array}{lcl}
|\overline{\gamma}_\epsilon\Theta|
+\nrm{\M_{\epsilon}^{1}(\overline{\mathcal{L}}_\epsilon^{\mathrm{qinv}}\Theta)' }_{\mathbf{L}^2}
+ \nrm{\overline{\mathcal{L}}_\epsilon^{\mathrm{qinv}}\Theta}_{\mathbf{L}^2}
&\leq &C\nrm{\M_{\epsilon}^{1}\Theta}_{\mathbf{L}^2}.\end{array}\end{equation}
\end{lemma}
\textit{Proof.}
The proof of \cite[Lem. 4.9]{HJHFHNINFRANGE} remains valid in this setting.\\
\qed\\

We can now concentrate on the kernel of
$\OL_\epsilon$. The quasi-inverse constructed above allows us to develop a
Liapunov-Schmidt argument to exclude kernel elements other
than $\overline{U}_{\epsilon}' $.

\begin{lemma}\label{FredholmeigenschappenLh}  Assume that (\asref{aannamesconstanten}{\text{H}}),
(\asref{aannamesconstanten3}{\text{H}}), (\asref{aannamespuls:even}{\text{H}}),
(\asref{aannamespuls:odd}{\text{H}}), (\asref{extraaannamespuls}{\text{H}}) and
(\asref{extraextraaannamespuls}{\text{H}}) are satisfied.
Then for all sufficiently small $\epsilon > 0$ we have
\begin{equation}\begin{array}{lcl}\ker (\OL_\epsilon)&=&\span\{\overline{U}_{\epsilon}' \}. \\[0.2cm]
%&=&\{\Psi\in \mathbf{L}^2:\ip{\Psi,\Theta}_{\mathbf{L}^2}=0\text{ for %all }\Theta\in \Range (\OL_\epsilon^{\mathrm{adj}})\}.
\end{array}\end{equation}
%where $\OL_\epsilon^*$ is the formal adjoint of $\OL_\epsilon$.
\end{lemma}
\textit{Proof.}
This result can be obtained by following the procedure used in the proof of
\cite[Lem. 4.10-4.11]{HJHFHNINFRANGE}.
\qed\\
%The proof is identical to \cite[Lemma 4.10-Lemma 4.11]{HJHFHNINFRANGE} and will be omitted.\\

We now set out to show that the eigenfunction $\overline{U}_{\epsilon}'$ is in fact simple.
As a technical preparation, we obtain
a lower bound on $\overline{\gamma}_{\epsilon}(\overline{U}_{\epsilon}')$,
which will help us to exploit the quasi-inverse
constructed in Lemma \ref{prop3.2bdf}.

\begin{lemma}\label{unifboundgamma}  Assume that (\asref{aannamesconstanten}{\text{H}}),
(\asref{aannamesconstanten3}{\text{H}}), (\asref{aannamespuls:even}{\text{H}}), (\asref{aannamespuls:odd}{\text{H}}),
(\asref{extraaannamespuls}{\text{H}}) and (\asref{extraextraaannamespuls}{\text{H}}) are satisfied.
Then there exists a constant $\gamma_*>0$ so that the inequality
\begin{equation}
\begin{array}{lcl}
|\overline{\gamma}_\epsilon\overline{U}_\epsilon'|&\geq &\gamma_*
\end{array}
\end{equation}
holds for all sufficiently small $\epsilon > 0$.
\end{lemma}
\textit{Proof.}
We note first that
the limit $\overline{U}_{\epsilon}'  \rightarrow \overline{U}_0'$ in $\mathbf{L}^2$
and the inequality $\ip{ \overline{U}_{e;0}' , \Phi_{e;0}^{\mathrm{adj}}  }_{\mathbf{L}^2_e} \neq 0$
imply that there exists a constant $\nu_* > 0$ so that
\begin{equation}
\label{eq:st:unif:bnd:ip2}\begin{array}{lcl}
\abs{ \langle \overline{U}_{\epsilon}'  , (0,\Phi_{e;0}^{\mathrm{adj}} ) \rangle_{\mathbf{L}^2} } &\ge &\nu_*\end{array}
\end{equation}
for all small $\epsilon > 0$.\\

We now introduce the function
\begin{equation}
    \begin{array}{lcl}
          \Psi_\epsilon & = & \overline{\mathcal{L}}_\epsilon^{\mathrm{qinv}}
             \overline{U}_{\epsilon}'.
    \end{array}
\end{equation}
The uniform bound \sref{eq:st:qinv:unif:bnd} shows that
we may assume an a-priori bound of the form
\begin{equation}\begin{array}{lcl}
\label{eq:st:unif:bnd:psi:eps2}
\norm{\Psi_{\epsilon}}_{\mathbf{L}^2} &\le &C_1'\end{array}
\end{equation}
for some $C_1' > 0$.\\

For any sufficiently small $\delta>0$ and $0<\epsilon<\epsilon_0(\delta)$,
the explicit form of $\overline{\gamma}_\epsilon$ given in \cite[Eq. (4.47)]{HJHFHNINFRANGE}
implies that
\begin{equation}
    \begin{array}{lcl}
         \overline{\gamma}_\epsilon\overline{U}_\epsilon' &=&\frac{\big\langle(0, \Phi_{e;0}^{\mathrm{adj}}),\big(\OL_\epsilon+\delta\big)^{-1}
           \big(\overline{U}_{\epsilon}'  +\delta\Psi_\epsilon\big)\big\rangle_{\mathbf{L}^2}}{\big\langle(0, \Phi_{e;0}^{\mathrm{adj}}),\big(\OL_\epsilon+\delta\big)^{-1}
           \overline{U}_0'\big\rangle_{\mathbf{L}^2}}
\\[0.4cm]
         &=&\frac{\big\langle (0, \Phi_{e;0}^{\mathrm{adj}}),\delta^{-1} \overline{U}_{\epsilon}'
                 +\big(\OL_\epsilon+\delta\big)^{-1}\delta\Psi_\epsilon\big\rangle_{\mathbf{L}^2} }{\big\langle(0, \Phi_{e;0}^{\mathrm{adj}}),\big(\OL_\epsilon+\delta\big)^{-1}
           \overline{U}_0'\big\rangle_{\mathbf{L}^2}}.
    \end{array}
\end{equation}
Since $\big(\OL_\epsilon+\delta\big)^{-1}\delta\Psi_\epsilon$ is uniformly bounded in $\mathbf{L}^2$
for all sufficiently small $\delta > 0$ and $0<\epsilon<\epsilon_0(\delta) $ on account of Corollary \ref{theorem4equivalentspecific2}
and \sref{eq:st:unif:bnd:psi:eps2}, we can use the lower bound \sref{eq:st:unif:bnd:ip2} to assume that
$\delta>0$ is small enough to have
\begin{equation}
    \begin{array}{lcl}
    \big|\big\langle (0, \Phi_{e;0}^{\mathrm{adj}}),\delta^{-1} \overline{U}_{\epsilon}'
                 +\big(\OL_\epsilon+\delta\big)^{-1}\delta\Psi_\epsilon\big\rangle_{\mathbf{L}^2}\big|&\geq &C_2'\delta^{-1}
    \end{array}
\end{equation}
for all such $(\epsilon, \delta)$.
% sufficiently small $\delta>0$ and $0<\epsilon<\epsilon_0(\delta)$.
Moreover, the uniform bound in Corollary \ref{theorem4equivalentspecific2} also yields the
upper bound
\begin{equation}
    \begin{array}{lcl}
    \big|\big\langle(0, \Phi_{e;0}^{\mathrm{adj}}),\big(\OL_\epsilon+\delta\big)^{-1}
           \overline{U}_0'\big\rangle_{\mathbf{L}^2}\big|&\leq &C_3'(1+\delta^{-1})
    \end{array}
\end{equation}
for all such $(\epsilon, \delta)$.
% sufficiently small $\delta>0$ and $0<\epsilon<\epsilon_0(\delta)$.
This gives us the lower bound
\begin{equation}
    \begin{array}{lclcl}
    |\overline{\gamma}_\epsilon\overline{U}_\epsilon'|
    &\geq & \frac{C_2'}{C_3'}\frac{\delta^{-1}}{1+\delta^{-1}}
    %\\[0.2cm]
    %&\geq &%
     &\ge &\gamma_*
    \end{array}
\end{equation}
for some $\gamma_*>0$ that can be chosen independently of $\delta>0$.\qed\\

%Since the Fredholm index of $\OL_{\epsilon}$ is zero we conclude that
%\begin{equation}\begin{array}{lclcl}
%\dim\big(\ker(\OL_\epsilon^{\mathrm{adj}})\big) &=&\dim\big(\ker(\OL_\epsilon)\big) &= &1,\end{array}
%\end{equation}
%which allows us to choose non-zero kernel elements $\Phi_{\epsilon}^{\mathrm{adj}}\in \ker(\OL_\epsilon^{\mathrm{adj}})$.
%We normalize these functions to have
%$\nrm{\Phi_\epsilon^{\mathrm{adj}}}_{\mathbf{H}^1}=1$.

\begin{lemma}\label{0issimpleeigenvalue}  Assume that (\asref{aannamesconstanten}{\text{H}}),
(\asref{aannamesconstanten3}{\text{H}}), (\asref{aannamespuls:even}{\text{H}}), (\asref{aannamespuls:odd}{\text{H}}),
(\asref{extraaannamespuls}{\text{H}}) and (\asref{extraextraaannamespuls}{\text{H}}) are satisfied.
Then for all sufficiently small $\epsilon > 0$ we have $\overline{U}_\epsilon'\notin\Range(\OL_\epsilon)$.
%\begin{equation}\begin{array}{lcl}
%\langle \overline{U}_{\epsilon}' , \Phi_{\epsilon}^{\mathrm{adj}} \rangle_{\mathbf{L}^2} &\neq& 0.\end{array}
%\end{equation}
\end{lemma}
\textit{Proof.}
Arguing by contradiction, let us suppose that there exists
$\Psi_\epsilon\in\mathbf{H}^1$ for which the identity
\begin{equation}
    \begin{array}{lcl}
    \OL_\epsilon\Psi_\epsilon&=&  \overline{U}_{\epsilon}'
    \end{array}
\end{equation}
holds. The observation above allows us to add an appropriate multiple of $\overline{U}_{\epsilon}'$ to $\Psi_{\epsilon}$
to ensure that $\langle \Psi_\epsilon,   (0,\Phi_{e;0}^{\mathrm{adj}} ) \rangle_{\mathbf{L}^2} = 0$.
In particular, Lemma \ref{prop3.2bdf} implies that
\begin{equation}
    \begin{array}{lclcl}
         \overline{\gamma}_\epsilon   \overline{U}_{\epsilon}'&=&0,
         \qquad \qquad
          \overline{\mathcal{L}}_\epsilon^{\mathrm{qinv}}
             \overline{U}_{\epsilon}' &=&\Psi_\epsilon,
    \end{array}
\end{equation}
which immediately contradicts Lemma \ref{unifboundgamma}.
\qed\\

\noindent\textit{Proof of Proposition \ref{samenvattingh5}.} 
Property (iii) follows directly from the results in \cite{MPA}. The rest of the result follows directly from Lemmas \ref{FredholmeigenschappenLh} and \ref{0issimpleeigenvalue}.\qed.\\

\subsection{Spectral stability}

Here we set out to establish the statements in Proposition \ref{maintheorem2}
for $\lambda \notin 2 \pi i c_{\epsilon} \mathbb{Z}$.
In contrast to the setting in \cite{HJHFHNINFRANGE},
the period $2\pi i c_\epsilon$ can be uniformly bounded for  $\epsilon\downarrow 0$.
In particular, we will only consider values of $\epsilon> 0$ that are sufficiently small to ensure that
\begin{equation}\begin{array}{lclcl}
\frac{3}{4}c_0&<&c_\epsilon&<&\frac{3}{2}c_0\end{array}
\end{equation}
holds. Recalling the constant $\lambda_0$ introduced
in Proposition \ref{theorem4equivalentcompact},
this allows us to restrict our spectral analysis to the set
\begin{equation}\begin{array}{lcl}
 \mathcal{R}&:=&\{\lambda\in\C:\Re\lambda\geq-\lambda_0,
   |\Im\lambda|\leq \frac{3}{2}\pi c_0\} \setminus \{ 0 \} .\end{array}
\end{equation}
On account of \ref{Lepsfredholm}, the operators
$\OL_{\epsilon,\lambda}$ are all Fredholm with index $0$ on this set.
We hence only need to establish their injectivity.
\\

In turns out to be
convenient to partition this strip into three $\epsilon$-independent parts.
The first part contains
values of $\lambda$ that are close to $0$, which can be
analyzed using the theory developed in \S\ref{sec:stb:zero}.
The second part contains all values of
$\lambda$ for which $\Re\lambda$ is sufficiently large.
Such values can be excluded from the spectrum by straightforward norm estimates.
The remaining part is compact, which allows us to appeal to
Proposition \ref{theorem4equivalentcompact}.

\begin{lemma}\label{lemmaspectrumklein}
Assume that (\asref{aannamesconstanten}{\text{H}}), (\asref{aannamesconstanten3}{\text{H}}),
(\asref{aannamespuls:even}{\text{H}}), (\asref{aannamespuls:odd}{\text{H}}), (\asref{extraaannamespuls}{\text{H}}) and
(\asref{extraextraaannamespuls}{\text{H}}) are satisfied. There exists constants
$\lambda_{I}> 0$ and $\epsilon_{I} > 0$ so that the operator
$\OL_{\epsilon, \lambda} : \mathbf{H}^1 \to \mathbf{L}^2$ is injective for all $\lambda \in \C$
with $0 < \abs{\lambda} < \lambda_{I}$ and $0 < \epsilon < \epsilon_{I}$.
\end{lemma}
\textit{Proof.}
We argue by contradiction. Pick a small $\lambda_{I}>0$ and $0<\epsilon<\epsilon_{**}$ and
assume that there exists $\Psi\in\mathbf{H}^1$ and $0<|\lambda|<\lambda_{I}$ with
$\Psi\neq 0$ and
\begin{equation}\label{eq:ew:tegenspraak}
    \begin{array}{lcl}
    \OL_{\epsilon}\Psi&=&\lambda\Psi.
    \end{array}
\end{equation}
Aiming to exploit the quasi-inverse in Lemma \ref{prop3.2bdf}, we
use \sref{eq:st:unif:bnd:ip2}
to decompose $\Psi$ as
\begin{equation}
    \begin{array}{lcl}
    \Psi&=&\kappa\overline{U}_{\epsilon}'+\Psi^{\perp}
    \end{array}
\end{equation}
for some $\kappa\in\Real$ and
$\Psi^{\perp}\in\mathbf{H}^1$ that satisfies the normalisation condition
\begin{equation}
    \begin{array}{lcl}
    \ip{(0,\Phi_{e;0}^{\mathrm{adj}}),\Psi^{\perp}}_{\mathbf{L}^2}&=&0 .
    \end{array}
\end{equation}
In view of Lemma \ref{prop3.2bdf},
the identity (\ref{eq:ew:tegenspraak})
implies that
\begin{equation}\label{eq:ew:qinv}
    \begin{array}{lclcl}
    \overline{\gamma}_\epsilon  \big[\kappa\lambda\overline{U}_\epsilon'+\lambda\Psi^{\perp}\big]&=&0,
         \qquad \qquad
          \overline{\mathcal{L}}_\epsilon^{\mathrm{qinv}}
             \big[\kappa\lambda\overline{U}_\epsilon'+\lambda\Psi^{\perp}\big]&=&\Psi^{\perp}.
    \end{array}
\end{equation}
On account of the uniform bound (\ref{eq:st:qinv:unif:bnd}), we can assume that $\lambda_{I}$ is small enough to have
\begin{equation}
\label{eq:restr:size:on:lambda:i}
    \begin{array}{lcl}
    \lambda_{I}\nrm{\overline{\mathcal{L}}_\epsilon^{\mathrm{qinv}}}_{\mathcal{B}(\mathbf{L}^2;\mathbf{L}^2)}&<&\frac{1}{2}.
    \end{array}
\end{equation}
Since $|\lambda|<\lambda_{I}$, this means that we can rewrite (\ref{eq:ew:qinv}) to obtain
\begin{equation}
    \begin{array}{lcl}
    \Psi^{\perp}&=&\big[I-\lambda \overline{\mathcal{L}}_\epsilon^{\mathrm{qinv}}\big]^{-1}
      \overline{\mathcal{L}}_\epsilon^{\mathrm{qinv}}\big[\kappa\lambda\overline{U}_\epsilon'\big].
    \end{array}
\end{equation}
In particular, the first identity in
(\ref{eq:ew:qinv})
allows us to write
\begin{equation}
\label{q:stb:region:i:bnd:with:gamma:tilde}
    \begin{array}{lcl}
  0 &=& \overline{\gamma}_\epsilon  \Big[\kappa\lambda\overline{U}_\epsilon'
    +\lambda\big[I-\lambda \overline{\mathcal{L}}_\epsilon^{\mathrm{qinv}}\big]^{-1}
       \overline{\mathcal{L}}_\epsilon^{\mathrm{qinv}}\big[\kappa\lambda\overline{U}_\epsilon'\big]\Big]\\[0.2cm]
  &=&\kappa\lambda\overline{\gamma}_\epsilon
    \Big[\overline{U}_\epsilon'
       +\lambda\big[I-\lambda \overline{\mathcal{L}}_\epsilon^{\mathrm{qinv}}\big]^{-1}
         \overline{\mathcal{L}}_\epsilon^{\mathrm{qinv}}\big[\overline{U}_\epsilon'\big]\Big].
    \end{array}
\end{equation}
We note that the restriction
\sref{eq:restr:size:on:lambda:i}
ensures that
the second identity
in \sref{eq:ew:qinv}
has no non-zero solutions $\Psi^\perp$
for $\kappa = 0$.
In particular,
\sref{q:stb:region:i:bnd:with:gamma:tilde}
implies that we must have
\begin{equation}
    \begin{array}{lcl}
  \overline{\gamma}_\epsilon  \overline{U}_\epsilon'&=&
    -\lambda\overline{\gamma}_\epsilon
    \Big[\big[I-\lambda \overline{\mathcal{L}}_\epsilon^{\mathrm{qinv}}\big]^{-1}
      \overline{\mathcal{L}}_\epsilon^{\mathrm{qinv}}
       \big[\overline{U}_\epsilon'\big]\Big].
    \end{array}
\end{equation}
On account of (\ref{eq:st:qinv:unif:bnd})
we hence obtain the estimate
\begin{equation}\begin{array}{lclcl}
\abs{ \overline{\gamma}_\epsilon  \overline{U}_\epsilon' }
& \le & C_1' \abs{\lambda} & \le &C_1' \lambda_I
\end{array}
\end{equation}
for some $C_1' > 0$.
However, Lemma \ref{unifboundgamma} shows that  the left-hand side
remains bounded away from zero, which yields
the desired contradiction after restricting the size of $\lambda_I$.
\qed\\

%Now we will consider the part of the spectrum where the real part is very large.

\begin{lemma}\label{spectrumgroot} Assume that (\asref{aannamesconstanten}{\text{H}}),
(\asref{aannamesconstanten3}{\text{H}}), (\asref{aannamespuls:even}{\text{H}}),
(\asref{aannamespuls:odd}{\text{H}}), (\asref{extraaannamespuls}{\text{H}}) and
(\asref{extraextraaannamespuls}{\text{H}}) are satisfied.
There exist constants $\lambda_{II}>0$ and $\epsilon_{II} > 0$
so that the operator
$\OL_{\epsilon, \lambda} : \mathbf{H}^1 \to \mathbf{L}^2$ is injective for all
$\lambda \in \mathcal{R}$ with $\Re \lambda \ge \lambda_{II}$
and $0 < \epsilon < \epsilon_{II}$.
\end{lemma}
\textit{Proof.}
The identity $\overline{\mathcal{L}}_{\epsilon, \lambda} \Phi = 0$
implies that
\begin{equation}\begin{array}{lcl}
c_{\epsilon} \Phi'
&= &\M^{1}_{1/\epsilon^2} J_{\mathrm{mix}} \Phi
 + DF(\overline{U}_{\epsilon} ) \Phi
 - \lambda \Phi.\end{array}
\end{equation}
Taking the inner product with
$\mathcal{M}^{1,2}_{\epsilon^2} \Phi$,
we may use Lemma \ref{lemma6bewijsbound1}
to obtain
\begin{equation}
\begin{array}{lcl}
0 & \le &
- \Re \langle J_{\mathrm{mix}} \Phi , \Phi \rangle_{\mathbf{L}^2}
\\[0.2cm]
& = &
 \Re \langle DF(\overline{U}_{\epsilon} ) \Phi, \M^{1,2}_{\epsilon^2} \Phi \rangle_{\mathbf{L}^2}
 - \Re \lambda \norm{ \M^{1,2}_{\epsilon} \Phi }_{\mathbf{L}^2}
\\[0.2cm]
& \le &
  (K_F - \Re \lambda) \norm{ \M^{1,2}_{\epsilon} \Phi}_{\mathbf{L}^2} .
\end{array}
\end{equation}
For $\Re \lambda \ge K_F$ this hence implies $\Phi = 0$, as desired.
\qed\\

\noindent\textit{Proof of Proposition \ref{maintheorem2}.}
On account of Proposition \ref{samenvattingh5} and Lemma's \ref{lemmaspectrumklein}-\ref{spectrumgroot},
it remains to consider the set
\begin{equation}
\begin{array}{lcl}
M&=&\{\lambda\in\mathcal{R}:|\lambda|\geq\lambda_I,\Re\lambda\leq\lambda_{II}\}.
\end{array}
\end{equation}
Since this set satisfies %(\asrefrev{Massumption}{M_{\lambda_{0}}}),
(h$M_{\lambda_{0}}$),
we can apply Proposition \ref{theorem4equivalentcompact}
to show that for each sufficiently small $\epsilon > 0$,
the operators $\overline{\mathcal{L}}_{\epsilon , \lambda}$ are invertible
for all $\lambda \in M$.
\qed\\

%Spectral stability

\bibliographystyle{plain}
\bibliography{ref}

\end{document}